\documentclass[reqno,12pt,a4paper]{amsart}

\usepackage[margin=2cm]{geometry}

\usepackage{amsmath,amsthm,amssymb,amsfonts,dsfont,ifpdf,mathtools,verbatim}
\usepackage[shortlabels]{enumitem}

\usepackage{mathrsfs}
\usepackage{graphicx}
\usepackage[all]{xy}
\usepackage{csquotes}
\MakeOuterQuote{"}
\usepackage{xcolor}
\usepackage[pagebackref]{hyperref}
\hypersetup{
	colorlinks,
	linkcolor={red!60!black},
	citecolor={blue!60!black},
	urlcolor={blue!90!black}
}

\numberwithin{equation}{section}

\newtheorem {thm}{Theorem}[section]
\newtheorem {lem}[thm]{Lemma}

\newtheorem {cor}[thm]{Corollary}
\newtheorem {prop}[thm]{Proposition}
\newtheorem* {prop*}{Proposition}

\newtheorem*{claim*}{Claim}

\newtheorem*{conj*}{Conjecture}

\theoremstyle{definition}
\newtheorem {defn}[thm]{Definition}

\newtheorem {rmk}[thm]{Remark}
\newtheorem*{rmk*}{Remark}

\newtheorem*{qst*}{Question}

\newtheorem* {problem*}{Problem}

\newcommand{\eq}[1]{\begin{equation*}#1 \end{equation*}}
\newcommand{\eql}[2]{\begin{equation}\label{#1}#2\end{equation}}

\newcommand {\norm}[1] {\left\| {#1} \right\|}

\renewcommand{\d}{\textup{d}}

\newcommand {\supl}   {\sup\limits}

\newcommand {\limsupl}   {\limsup\limits}
\newcommand {\liminfl}   {\liminf\limits}

\newcommand {\A} {{\mathbb A}}

\newcommand {\C} {{\mathbb C}}

\newcommand {\G} {{\mathbb G}}

\newcommand {\N} {{\mathbb N}}
\newcommand {\Q} {{\mathbb Q}}

\newcommand {\R} {{\mathbb R}}
\newcommand {\T} {{\mathbb T}}

\newcommand {\Z} {{\mathbb Z}}
\newcommand {\cA} {{\mathcal A}}
\newcommand {\cB} {{\mathcal B}}

\newcommand {\cD} {{\mathcal D}}
\newcommand {\cE} {{\mathcal E}}

\newcommand {\cI} {{\mathcal I}}
\newcommand {\cJ} {{\mathcal J}}

\newcommand {\cN} {{\mathcal N}}

\newcommand {\cP} {{\mathcal P}}

\newcommand {\cS} {{\mathcal S}}

\newcommand {\cX}  {{\mathcal X}}

\DeclareMathOperator{\hdim}{\dim_H}

\DeclareMathOperator{\ubdim}{\overline{\dim}_B}

\DeclareMathOperator{\supp}{supp}

\DeclareMathOperator{\Or}{O}
\DeclareMathOperator{\SO}{SO}

\DeclareMathOperator{\Id}{Id}

\newcommand{\eps}{\varepsilon}

\newcommand{\bP}{\mathbf{P}}

\newcommand{\bT}{\mathbf{T}}

\newcommand{\sX}{\mathsf{X}}

\newcommand{\sx}{\mathsf{x}}
\newcommand{\sy}{\mathsf{y}}

\newcommand{\bigast}{\mathop{\scalebox{1.5}{\raisebox{-0.2ex}{$\ast$}}}}

\newenvironment{labeledlist}[2][\unskip]
{ \renewcommand{\theenumi}{\textnormal{({#2}\arabic{enumi}{#1})}}
	\renewcommand{\labelenumi}{\textnormal{({#2}\arabic{enumi}{#1})}}
	\begin{enumerate} }
	{ \end{enumerate} }

\begin{document}
	
	\title[Dynamical self-similarity in $\R^d$]{Dynamical self-similarity, $L^{q}$-dimensions and Furstenberg slicing in $\R^d$}
	\author[E.~Corso]{Emilio Corso}
	\address[E. C.]{Pennsylvania State University Mathematics Department, 203 McAllister Building, University Park, State College, PA 16802}
	\email{corsoemilio2@gmail.com}
	
	\author[P.~Shmerkin]{Pablo Shmerkin}
	\address[P.S.]{Department of Mathematics\\
		The University of British Columbia\\
		Room 121, 1984 Mathematics Road\\
		Vancouver, BC\\
		Canada V6T 1Z2
	}
	\email{pshmerkin@math.ubc.ca}
	
	\date{\today}
	
	
	\begin{abstract}
		We extend a theorem of the second author on the $L^q$-dimensions of dynamically driven self-similar measures from the real line to arbitrary dimension. Our approach provides a novel, simpler proof even in the one-dimensional case. As consequences, we show that, under mild separation  conditions, the $L^q$-dimensions of homogeneous self-similar measures in $\R^d$ take the expected values, and we derive higher rank slicing theorems in the spirit of Furstenberg's slicing conjecture.
	\end{abstract}
	
	\maketitle
	
	\setcounter{tocdepth}{2}
	\tableofcontents
	
	\section{Introduction}

	\subsection{The setup: models and dynamically driven self-similar measures}
	\label{subsec:models}
	
	In this paper, we generalize to higher dimensions the framework of dynamically driven self-similar measures on the line developed by the second author in~\cite{Shmerkin-Annals}. In particular, under suitable natural conditions we establish a formula for the $L^q$-dimension of dynamically driven self-similar measures: see Theorem \ref{thm:lqdimension}. This directly generalizes the one-dimensional result in~\cite{Shmerkin-Annals}, whose relevance is given by its many applications, including to the resolution of Furstenberg's slicing conjecture, and to the smoothness of self-similar measures. In this article, we obtain similar applications in higher dimensions.
	
	We now formally introduce the main objects of study in this paper. We define a \emph{topological system} to be a pair $(\sX,\bT)$ where $\sX$ is a non-empty compact metrizable topological space and $\bT\colon \sX \to \sX$ is a continuous map. If $(\sX,\bT)$ is a topological system, a Borel probability measure $\bP$ on $\sX$ is said to be invariant under $\bT$ if $\bT\bP=\bP$; here and throughout, the notation $f\mu$ indicates the push-forward of a measure $\mu$ under a measurable map $f$.
	A topological system $(\sX,\bT)$ is called uniquely ergodic if there exists a unique $\bT$-invariant Borel probability measure $\bP$ on $\sX$; in this case, we refer to the triple $(\sX,\bT,\mathbf{P})$ as a \emph{uniquely ergodic system}. The ergodic-theoretic background required in this article can be mostly found in Einsiedler-Ward's textbook~\cite{Einsiedler-Ward}.

	Let $\cA_d$ denote the set of all Borel probability measures on $\R^d$ with finite support; we endow $\cA_d$ with the weak$^*$ topology. 
	For every $\lambda\in \R$, we let $S_{\lambda}\colon \R^d\to \R^d$ be the scaling map $S_{\lambda}(x)=\lambda x$.
	
	\begin{defn}[Pleasant model, dynamically driven self-similar measures]
		\label{def:pleasantmodel}
		A \emph{pleasant model} in $\R^d$ is a quintuple $\cX=(\sX,\bT,\mathbf{P},\Delta,\lambda)$ where $(\sX,\bT,\mathbf{P})$ is a uniquely ergodic system, $\lambda$ is a real number in the interval $(0,1)$ and $\Delta\colon \sX \to\cA_d$ is a Borel-measurable, $\bP$-almost everywhere continuous map satisfying the following property: there exist an integer $M\geq 1$ and a bounded set $B\subset \R^d$ such that, for every $\sx\in \sX$, the support of $\Delta(\sx)$ consists of at most $M$ points and is contained in $B$.
		
		A pleasant model $\cX=(\sX,\bT,\mathbf{P},\Delta,\lambda)$ in $\R^d$ generates a collection $\bigl(\mu_\sx^{(\cX)}\bigr)_{\sx\in \sX}$ of \emph{dynamically driven self-similar measures}, which are Borel probability measures on $\R^d$ defined as the infinite convolution product
		\begin{equation*}
			\mu^{(\cX)}_\sx=\bigast\limits_{n=0}^{\infty}\;S_{\lambda^n}\Delta(\bT^n\sx)\;, \quad \sx\in \sX.
		\end{equation*}
	\end{defn}
	Since the measure $\bP$ is uniquely determined by the pair $(\sX,\bT)$, we sometimes omit it from the notation.
	When the model $\cX$ is clear from the context, we suppress the superscript and simply write $(\mu_{\sx})_{\sx\in \sX}$ for the dynamically driven self-similar measures generated by $\cX$. Observe that, owing to the uniform boundedness of the supports of the measures in the image of $\Delta$, each measure $\mu_\sx$ is well-defined as weak$^\star$ limit, as $n$ tends to infinity, of the finite convolution products
	\begin{equation*}
		\mu_{\sx,n}\coloneqq\bigast_{i=0}^{n-1}\;S_{\lambda^i}\Delta(\bT^i\sx)\;,
	\end{equation*}  
	and is the law of the random, almost surely convergent infinite sum $\sum_{n=0}^{\infty}\lambda^{n}Z_n(\sx)$
	where $(Z_n(\sx))_{n\geq 0}$ is a sequence of independent $\R^{d}$-valued random variables, $Z_n(\sx)$ having law $ \Delta(\bT^n\sx)$ for every $n\geq 0$.
	
	Observe that the relation
	\begin{equation}
		\label{eq:dynamicalselfsimilar}
		\mu_{\sx}=\mu_{\sx,n}\ast S_{\lambda^{n}}\mu_{\bT^{n}\sx}
	\end{equation}
	holds for every $n\in \N^*$ and $\sx\in \sX$. This is a dynamical self-similarity property: each measure $\mu_{\sx}$ is a convex combination, determined by $\mu_{\sx,n}$, of copies of $\mu_{\bT^n\sx}$, each scaled down by $\lambda^{n}$. 
	
	\begin{rmk}
		The assumptions we place on a pleasant model are amenable to relaxation, provided that the results are appropriately reformulated. Most significantly in view of applications, it is possible to just assume ergodicity of the measure $\bP$ instead of unique ergodicity, and to replace the scaling by $\lambda^n$ in the definition of the measures $\mu_{\sx}^{(\cX)}$ with a scaling by a multiplicative cocycle $\lambda(n,\sx)=\prod_{i=0}^{n-1}\lambda(\bT^i\sx)$, where $\lambda\colon \sX\to (0,1)$ is a measurable function. In a companion article we plan to pursue this greater level of generality, which allows, for instance, to deal with random self-similar measures in addition.
	\end{rmk}

	\subsection{$L^q$-dimensions of a model}
	
	Let $\mu$ be a compactly supported Borel probability measure on $\R^{d}$, $q>1$ a real number.  For every integer $m\geq 1$, let 
	\eq{\cD_m=\biggl\{ \prod_{i=1}^d\bigl[2^{-m}k_i,2^{-m}(k_i+1)\bigr):(k_1,\dots,k_d)\in \Z^d  \biggr\}}
	be the standard partition of $\R^d$ into half-open $2^{-m}$-mesh cubes. The $q$-th moment $\sum_{Q\in \cD_m}\mu(Q)^{q}$ of $\mu$ at scale $2^{-m}$, with the sum involving only a finite number of non-zero terms, quantifies how spread out $\mu$ is among cubes of side-length $2^{-m}$; assuming indeed, upon rescaling, that $\mu$ is supported inside the unit cube $[0,1)^{d}$, it is straightforward to verify (cf.~Lemma~\ref{lem:qnorm}) that
	\begin{equation*}
		2^{-md(q-1)}	\leq \sum_{Q\in \cD_m}\mu(Q)^{q}\leq 1\;,
	\end{equation*}
	the sum in the middle being close to the upper bound when $\mu$ is highly concentrated on a small number of cubes of the generation $\cD_m$, and approaching, on the other hand, the lower bound when the mass of $\mu$ is close to being uniformly distributed among such cubes. The last displayed inequality amounts to
	\begin{equation}
		\label{eq:trivialbound}
		0\leq -\frac{\log{\sum_{Q\in \cD_m} \mu(Q)^{q}}}{(q-1)m}\leq d\;,
	\end{equation}
	where here and throughout the article all logarithms are taken to the base $2$. The asymptotic behavior, as $m$ tends to infinity, of the quantity in the middle of~\eqref{eq:trivialbound} provides therefore a reasonable notion of dimension for the measure $\mu$.
	
	\begin{defn}[$L^{q}$-spectrum and $L^q$-dimensions of a measure]
		Let $\mu$ be a Borel probability measure on $\R^d$. The \emph{$L^q$-spectrum} of $\mu$ is the function $\tau_{\mu}\colon \R_{> 1}\to \R_{\geq 0}$ given by
		\begin{equation}
			\label{eq:Lqspectrum}
			\tau_{\mu}(q)=\liminf_{m\to\infty}-\frac{\log{\sum_{Q\in \cD_m}\mu(Q)^{q}}}{m}\;.
		\end{equation}
		The \emph{$L^{q}$-dimension} of $\mu$ is defined as
		\begin{equation*}
			\dim_{\mu}(q)=\frac{\tau_\mu(q)}{q-1}\;,\quad q>1\;.
		\end{equation*}
		At times we write $\tau(\mu,q)$ and $\dim(\mu,q)$ instead of $\tau_{\mu}(q)$ and $\dim_{\mu}(q)$, respectively.
	\end{defn}
	The $L^{q}$-spectrum can be defined for all real values of $q$; however, its values for $q\leq 1$ are not relevant for the purposes of the present article.
	
	The reader is referred to Falconer's monograph~\cite[Chap.~11]{Falconer-techniques} for an introduction to $L^q$-spectra and their role in multifractal analysis; here we focus on a handful of informative properties. It is known that the function $q\mapsto \tau_{\mu}(q)$ is concave, whence $q\mapsto \dim_{\mu}(q)$ is continuous and decreasing on the half-line $\R_{>1}$. Recall now that, given a finite positive Borel measure $\mu$ on a metric space $(X,d_X)$, its \emph{Hausdorff dimension} is defined as
	\begin{equation*}
		\hdim(\mu)=\inf\{\hdim(A):A\subset X \text{ Borel}, \;\mu(A)>0 \}
	\end{equation*}
	where $\hdim(A)$ denotes the Hausdorff dimension of a set $A$, a classical notion in fractal geometry for which we refer to~\cite[Chap.~10]{Falconer-techniques}. Given a real number $s\geq 0$, we say that $\mu$ has \emph{Frostman exponent} $s$ if there exists $C>0$ such that $\mu(B(x,r))\leq Cr^{s}$ for every $x\in X$ and $r> 0$, where $B(x,r)$ is the closed ball of radius $r$ centered at $x$. In loose terms, $L^{q}$-dimensions interpolate between Frostman exponents and the Hausdorff dimension of $\mu$; to be precise, we have on the one hand that $\hdim(\mu)\geq \lim_{q\to 1^+}\dim_\mu(q)$ (see~\cite[Theorem 1.4]{Fan-Lau-Rao}), and on the other that the inequality $\dim_{\mu}(q)>s$ for certain $q>1$ and $s>0$ implies that $\mu$ has Frostman exponent $(1-1/q)s$ (see \cite[Lemma 1.7]{Shmerkin-Annals}).
	
	We now extend the notion of $L^q$-spectrum and $L^q$-dimension to dynamically driven self-similar measures; the case $d=1$ was treated in \cite{Shmerkin-Annals}. 
	\begin{defn}[$L^q$-spectrum and $L^q$-dimensions of a model]
		Let $\cX=(\sX,\bT,\mathbf{P},\Delta,\lambda)$ be a pleasant model in $\R^d$. The \emph{$L^q$-spectrum} of $\cX$ is the function $T_{\mathcal{X}}\colon \R_{>1}\to\R_{\geq 0}$ given by
		\[
		T_{\cX}(q) = \liminf_{m\to\infty} -\frac{1}{m}\int_{\sX} \log \left(\sum_{Q\in \mathcal{D}_m} \mu_{\sx}(Q)^q \right)\;\d\mathbf{P}(\sx).
		\]
		The \emph{$L^q$-dimension} of $\cX$ is defined as
		\[
		D_{\cX}(q) =\frac{T_{\cX}(q)}{q-1}\;, \quad q>1.
		\]
	\end{defn}

	The following lemma shows that, for pleasant models, this average notion of $L^q$-spectrum  coincides with the standard $L^q$-spectrum of typical measures generated by the model.
	\begin{lem} \label{lem:spectrum-model-equals-spectrum-measure}
		Let $\mathcal{X}=(\sX,\bT,\mathbf{P},\Delta,\lambda)$ be a pleasant model. Then the limit in the definition of $T_{\mathcal{X}}(q)$ exists for every $q>1$.   Moreover, there is a set $\sX_0$ of full $\mathbf{P}$-measure such that $\tau_{\mu_{\sx}}=T_{\cX}$ for all $\sx\in \sX_0$.
	\end{lem}
	See Proposition \ref{prop:Tq} for the proof of this lemma.
	
	\subsection{Projected models, unsaturation on lines, and exponential separation}
	
	Our goal is to state Theorem \ref{thm:lqdimension}, which provides a formula for the $L^q$-dimension of dynamically driven self-similar measures generated by a pleasant model. To this end, we introduce two key notions in this theorem: unsaturation on lines and exponential separation.
	
	We denote by $\G(d,k)$ the Grassmanian of $k$-dimensional linear subspaces of $\R^d$, and by $\G(d)=\cup_{k=0}^{d}G(d,k)$ the collection of all subspaces of $\R^d$.
	
	We will always identify a plane in $\G(d,k)$ with the orthogonal projection on that plane. In particular, $\G(d,0)$ consists either of the trivial subspace or of the trivial map, according to the context, and $\G(d,d)$ is a singleton containing either $\R^d$ or the identity map.
	
	\begin{defn}
		Let $\mathcal{X}=(\sX,\bT,\mathbf{P},\Delta,\lambda)$ be a pleasant model in $\R^d$, and let $\pi\in\G(d,k)$. We define the \emph{projected model}
		\[
		\pi\mathcal{X}= (\sX,\bT,\mathbf{P},\pi\Delta,\lambda)\;.
		\]
	\end{defn}
	A moment's thought reveals that $\pi\mathcal{X}$ is also a pleasant model and, since projections commute with scalings and addition, the measures generated by the projected model are $( \pi\mu_{\sx}^{(\cX)})_{\sx\in\sX}$. Note that, as defined, these measures are supported on $\pi$. Sometimes it is convenient to apply a linear change of coordinates to $\pi\Delta$, and hence to the generated measures, so that $\pi\mathcal{X}$ becomes a pleasant model on $\R^k$ instead. Of course, this can be done so that the change of coordinates depends smoothly on $\pi$.
	
	\begin{defn}[Unsaturation on lines]
		\label{def:unsaturation}
		We say that a model $\mathcal{X}=(\sX,\bT,\mathbf{P},\Delta,\lambda)$ on $\R^d$ is \emph{$q$-unsaturated on lines} for some $q>1$ if
		\[
		D_{\cX}(q) < D_{\pi\cX}(q)+1\quad\text{for all }\pi\in\G(d,d-1).
		\]
	\end{defn}
	This definition can be seen as an $L^q$-analog of Hochman's dimensional unsaturation condition in \cite[Theorem 1.5]{Hochman17}. It is easy to see that $D_{\cX}(q)\le D_{\pi\cX}(q)+1$ for all $\pi\in\G(d,d-1)$; equality for some $\pi$ indicates that, in an $L^q$-sense, the measures $\mu_{\sx}$ are a product of $\pi\mu_{\sx}$ and the renormalized Lebesgue measure on an interval in the orthogonal complement of $\pi$, though we emphasize that this is not literally true. This is the reason for the terminology ``unsaturation''. It turns out that this notion of unsaturation implies a pointwise uniform version that a priori may appear to be much stronger; see Corollary \ref{cor:unsaturation-uniform} below.
	
	Observe also that unsaturation precludes maximal dimension. For a model $\mathcal{Y}$ on $\R^k$, we always have $D_{\mathcal{Y}}(q)\leq k$ for all $q>1$, which ultimately follows from the concavity of $t\mapsto t^q$ as explained in the proof of Proposition \ref{prop:upperbound}. Therefore, if $D_{\cX}(q)=d$, then we necessarily have $D_{\pi \cX}(q)=d-1$ for all $\pi\in\G(d,d-1)$, so that the model is not $q$-unsaturated on lines. In fact, for $d=1$, the condition $D_{\cX}(q)<1$ is equivalent to the model being $q$-unsaturated on lines.
	
	We now turn to the condition of \emph{exponential separation}. Given a vector $y=(y_j)_{1\leq j\leq d}\in \R^d$, we write $|y|_{\infty}=\sup_{1\leq j\leq d}|y_j|$ for its $\ell^{\infty}$-norm, which is particularly suited to our purposes as we work with partitions of $\R^d$ into cubes. For a positive measure $\nu$, the notation $\supp{\nu}$ indicates its support. If $(\Delta_i)_{i\in I}$ is a finite collection of finitely supported Borel probability measures on $\R^d$, $\mu=\ast_{i\in I}\;\Delta_i$ is their convolution and $\eta>0$ is a real number, we say by a mild abuse of terminology that the atoms of $\mu$ are \emph{$\eta$-separated} if 
	\begin{equation*}
		\biggl|\sum_{i\in I}x_i-\sum_{i\in I}y_i\biggr|_{\infty}\geq \eta\quad \text{for all }(x_i)_{i\in I}\neq (y_i)_{i\in I}\in \prod_{i\in I}\supp{\Delta_i}\;.
	\end{equation*}
	Observe that the notion is tailored to counting the elements of $\supp{\mu}$ with multiplicity.

	\begin{defn}[Exponential separation]
		\label{def:expseparation}
		Let $\cX=(\sX,\bT,\mathbf{P},\Delta,\lambda)$ be a pleasant model in $\R^d$, generating a collection $(\mu_{\sx})_{\sx\in \sX}$ of dynamically driven self-similar measures. We say that $\cX$ satisfies \emph{exponential separation} if, for $\bP$-almost every $\sx\in \sX$, there exists $R\in \N^*$ and a subsequence $(n_j)_{j\geq 1}$ such that the atoms of $\mu_{\sx,n_j}$ are $\lambda^{Rn_j}$-separated for every $j\geq 1$.
	\end{defn}
	
	The notions of unsaturation on lines and exponential separation will be discussed in more detail in Section \ref{subsec:applications-prelim}.
	
	\subsection{Main result}
	
	We can now state the main result of the article.
	\begin{thm}
		\label{thm:lqdimension}
		Let $\cX=(\sX,\bT,\bP,\Delta,\lambda)$ be a pleasant model in $\R^d$, generating a collection $(\mu_{\sx})_{\sx\in \sX}$ of dynamically driven self-similar measures. Assume that $\cX$ satisfies exponential separation and is $q$-unsaturated on lines for some $q\in\R_{>1}$. Then,
		\begin{equation}
			\label{eq:Lqdimensionformula}
			D_{\cX}(q) = \frac{\int_{\sX}\log{\norm{\Delta(\sy)}_q^{q}}\;\emph{d}\mathbf{P}(\sy)}{(q-1)\log{\lambda}},
		\end{equation}
		and
		\begin{equation*}
			\lim\limits_{m\to\infty}-\frac{\log{\sum_{Q\in \cD_{m}}\mu_\sx(Q)^{q}}}{m(q-1)}= D_{\cX}(q)
		\end{equation*}
		uniformly in $\sx\in\sX$. In particular, the limit in the definition of the $L^{q}$-dimension of $\mu_\sx$ exists and equals the constant value on the right-hand side of~\eqref{eq:Lqdimensionformula} for all $\sx\in \sX$.
	\end{thm}
	
	
	The case $d=1$ of the above theorem is ~\cite[Theorem 1.11]{Shmerkin-Annals}. As we explain below, the proof of Theorem \ref{thm:lqdimension} adapts many of the ideas in the one-dimensional case, but also entails some substantial differences.
	
	We note that the right-hand side of \eqref{eq:Lqdimensionformula} may \emph{a priori} be larger than $d$, but this cannot happen under the assumptions of the theorem. In fact, even $D_{\cX}(q)=d$ is not possible: else, $q$-unsaturation would be violated, since $D_{\pi\cX}(q) \le d-1$ for all $\pi\in\G(d,d-1)$. Nevertheless, Theorem \ref{thm:lqdimension} can be used to find the $L^q$-dimension of dynamically driven self-similar measures also in the critical and supercritical regime (where one expects $D_{\cX}(q)=d$), simply by arguing by contradiction.
	
	In the subsequent two subsections we present, for illustrative purposes, a number of applications of our main result, Theorem~\ref{thm:lqdimension}; further consequences are exposed in detail in \textsection\ref{sec:ssm}, \textsection\ref{sec:products} and \textsection\ref{sec:slicing}.
	
	\subsection{Homogeneous self-similar measures and their projections}
	\label{subsec:selfsimilar}
	
	Our first application, one of the main motivations behind the present work, concerns the $L^q$-dimension theory of homogeneous self-similar measures and their orthogonal projections. Consider an \emph{iterated function system}, henceforth abbreviated as IFS, on $\R^d$: by definition, this is a finite set $\Phi=\{f_i \}_{i\in I}$ of maps $f_i\colon \R^d\to \R^d$ which are contractions for the Euclidean metric. Given a probability vector $p=(p_i)_{i\in I}\in \bigl(\R_{\geq 0}\bigr)^{I}$, $\sum_{i\in I}p_i=1$, it is well known, and was first proven by Hutchinson~\cite{Hutchinson81} in this level of generality, that there is a unique Borel probability measure $\mu$ on $\R^d$ which is the $p$-weighted average of its images under the $f_i$'s, namely
	\eq{\mu=\sum_{i\in I}p_i \;f_i\mu\;;}
	the topological support of $\mu$ is contained in the unique non-empty compact set $K\subset \R^d$ satisfying
	\eq{K=\bigcup_{i\in I}f_i(K)\;,}
	and coincides with $K$ if $p$ is a \emph{positive} vector, that is, if $p_i>0$ for all $i\in I$. 
	
	If the maps in $\Phi$ are similarities, we shall say that $\Phi$ is a \emph{self-similar IFS} and that $\mu$ is the \emph{self-similar measure} determined by the pair $(\Phi,p)$. By the well known Mazur-Ulam theorem, under such a circumstance the maps in $\Phi$ are affine; if they share the same linear part, we shall say that $\Phi$ and $\mu$ are \emph{homogeneous}, and we refer to $h$ as the orthogonal part of $\Phi$. In this case we assume that the common similarity ratio does not vanish, else $\mu$ is finitely supported and there is not much else to be investigated about it.
	
	Let thus $\Phi=\{f_i\}_{i\in I}$ be a homogeneous self-similar IFS; there exist $\lambda\in (0,1)$, $h\in \Or_d(\R)$, the orthogonal group in $d$ dimensions, and a collection $\{a_i\}_{i\in I}$ of vectors in $\R^d$ such that, for all $i\in I$,
	\eq{f_i(x)=\lambda h(x)+a_i\;, \quad x\in \R^d.}
	Let $\mu$ be the homogeneous self-similar measure generated by $\Phi$ and a probabiliy vector $p=(p_i)_{i\in I}$. Then it is an elementary verification that $\mu$ admits the following expression as an infinite convolution products of finitely supported measures: if 
	\eq{\Delta_0=\sum_{i\in I}p_i\;\delta_{a_i}\;,}
	where $\delta_x$ indicates the Dirac mass at a point $x\in \R^d$, then
	\eq{\mu=\bigast_{n\geq 0}S_{\lambda^n}h^n \Delta_0\;.}
	It is now apparent how to view $\mu$ as a dynamically-driven self-similar measure generated by a pleasant model. Let $\sX$ be the closure of the cyclic subgroup $\langle h\rangle $ of $\Or_d(\R)$ generated by $h$; it is a compact abelian metrizable topological group. Let $\bT\colon \sX\to \sX$ be the translation map $g\mapsto hg$, which is obviously continuous, and let $\bP$ be the unique probability Haar measure on $\sX$. By construction, $h$ generates a dense cyclic subgroup of $\sX$, whence the system $(\sX,\bT,\bP)$ is uniquely ergodic; see, for instance,~\cite[Theorem 4.14]{Einsiedler-Ward}. 
	Define a map 
	\eq{\Delta\colon \sX\to \cA_d\;, \quad g\mapsto g\Delta_0\;;}
	$\Delta$ is ostensibly continuous with respect to the given topology on $\cA_d$, the supports of the measures $g\Delta_0$ have all the same cardinality, and they are all contained in the closed ball centered at the origin with radius $\sup_{i\in I}|a_i|$.
	The model $\cX=(\sX,\bT,\bP,\Delta,\lambda)$ is thus pleasant; the dynamically-driven self-similar measures it generates are 
	\eq{\mu_g=\bigast_{n\geq 0}S_{\lambda^n}\Delta(h^ng)=\bigast_{n\geq 0}S_{\lambda^n}h^ng\Delta_0=g\biggl(\bigast_{n\geq 0}S_{\lambda^n}h^n\Delta_0\biggr)=g\mu\;, \quad g\in \sX,}
	where in the third equality we used the fact that $g$ commutes with all $h^n$ and $S_{\lambda^n}$, and that $g(\ast_{n\geq 0}\;\rho_n)=\ast_{n\geq 0}\;g\rho_n$ for any sequence $(\rho_n)$ of finitely supported measures, which is a consequence of continuity and linearity of $g$. 
	
	Therefore the pleasant model $\cX$ generates the self-similar measure $\mu$ together with all its isometric images $g\mu$, $g\in \sX$. Under the appropriate assumptions of exponential separation and $q$-unsaturation on lines, which is sufficient to formulate for the single measure $\mu$, Theorem~\ref{thm:lqdimension} results at once in an $L^q$-dimension formula for $\mu$; this is the content of Corollary~\ref{cor:ssm}.
	
	Verifying explicitly the $q$-unsaturation condition in concrete instances of self-similar measure is, however, a rather intricate matter. Accordingly, we provide an elementarily checkable condition on the isometric part $h$, which in conjunction with the following "projected" version of exponential separation, yields the sought after $L^q$-dimension formula for $\mu$ as well as for all its orthogonal projections.
	
	\begin{defn}[Projected exponential separation] \label{def:projected-exp-sep}
		Let $\cX=(\sX,\bT,\bP,\Delta,\lambda)$ be a pleasant model in $\R^d$, and let $\pi\in\G(d,k)$. We say that the projected model $\pi\cX$ satisfies \emph{projected exponential separation}, abdridged as the PES property, if it satisfies exponential separation and, additionally, the restriction of $\pi$ to the support of $\Delta(\sx)$ is injective for $\bP$-almost all $\sx\in\sX$.
	\end{defn}

	\begin{thm} \label{thm:ssm}
		Let $\mu$ be a homogeneous self-similar measure generated by an iterated function system $\Phi=\{f_i   \}_{i\in I}$ with orthogonal part $h\in \Or_d(\R)$ and similarity ratio $\lambda\in (0,1)$, and by a probabilty vector $p=(p_i)_{i\in I}$. Assume that $h$ has distinct complex eigenvalues, and that the closed subgroup of $\Or_d(\R)$ generated by $h$ is connected.

		Let $\pi_1,\ldots, \pi_{\ell}$ be the minimal real $h$-invariant subspaces, and suppose that $\pi_j\mu$ satisfies projected exponential separation for all $1\le j\le \ell$. Then, for every integer $k\in \{1,\dots,d\}$, every subspace $\pi\in \G(d,k)$ and every $q>1$, 
		\[
		\dim_{\pi\mu}(q) = \min\left\{\frac{\log{\norm{p}_q^{q}}}{(q-1)\log{\lambda}},k\right\}\;.\]
	\end{thm}

	Observe that any $h$ as in the assumption of the theorem has at most one real eigenvalue, equal to $1$; therefore $\dim \pi_j=2$ for all but at most one $j\in \{1,\dots,\ell\}$, for which $\dim \pi_j=1$. 
	
	Specifying Theorem~\ref{thm:ssm} to the case $d=2$ yields the following corollary. 
	\begin{cor} \label{cor:planar-ssm}
		Let $\mu$ be a homogeneous self-similar measure in $\R^2$, generated by an iterated function system $\Phi$ with orthogonal part $h\in \SO_2(\R)$ and similarity ratio $\lambda\in (0,1)$, and by a probabiity vector $p$. Suppose that $h$ is an irrational rotation and $\Phi$ satisfies exponential separation. Then, for every $q>1$, 
		\[
		\dim_{\mu}(q) =  \min\left\{\frac{\log\|p\|_q^q}{(q-1)\log\lambda},2\right\}
		\]
		and, for every line $\pi\subset \R^2$ through the origin, 
		\eq{	\dim_{\pi\mu}(q) =  \min\left\{\frac{\log\|p\|_q^q}{(q-1)\log\lambda},1\right\}\;.}
	\end{cor}
	Notice that the assumption $h\in \SO_2(\R)$ is unrestrictive since, if $\mu$ is generated by $\Phi=\{f_i \}_{i\in I}$, then it is also generated by $\Phi^2=\{f_i\circ f_j  \}_{i,j\in I}$.

	\begin{rmk}
		In the case where the IFS consists only of two maps of the form $x\mapsto \lambda h(x)+1$, $x\mapsto \lambda h(x)-1$, the self-similar measures from Corollary \ref{cor:planar-ssm} are known as (biased) \emph{complex Bernoulli convolutions}: these are the laws of random infinite sums $\sum_{n=0}^\infty \pm \rho^n$, where $\rho$ is the complex number $\lambda e^{2\pi i \alpha}$ for $\alpha$ the rotation angle of $h$, and the signs $\pm$ are chosen independently with probability weight $(p,1-p)$ for some  $p\in (0,1)$.
	\end{rmk}
	
	As another corollary, we obtain the following improvement of \cite[Theorem B]{ShmerkinSolomyak23}. Let $B_{\R^2}(0,1)$ be the open unit ball in $\R^2$. 
	\begin{cor} \label{cor:abs-cont-parametrized-ssm}
		Fix a collection $\mathbf{a}=\{a_i\}_{i\in I}$ of vectors in $\R^2$. For any complex number $\rho\in B_{\R^2}(0,1)$, let $\Phi_{\rho,\mathbf{a}}$ be the iterated function system consisting of the maps
		\eq{f_i(z)=\rho z+a_i\;, \quad z\in \C}
		for $i\in I$. Then there is a set $E\subset B_{\R^2}(0,1)$ of zero Hausdorff dimension such that, for any $\rho =\lambda e^{2\pi i \alpha}\in B_{\R^2}(0,1)\setminus E$ with $\alpha\notin\Q$, and for any probability vector $p=(p_i)_{i\in I}$ satisfying
		\[
		\sum_{i\in I}p_i^q < \lambda^{q-1},
		\]
		the homogeneous self-similar measure generated by $\Phi_{\rho,\mathbf{a}}$ and $p$ is absolutely continuous with a density in $L^q$.
	\end{cor}
	The proof aligns with the one of~\cite[Theorem B]{ShmerkinSolomyak23}, except that we now appeal to Corollary \ref{cor:planar-ssm} instead of Hochman's theorem~\cite[Theorem 1.5]{Hochman17} on the Hausdorff dimension of self-similar measures.

	\subsection{Furstenberg-type slicing} 
	
	For any integer $p\geq 2$,	let $T_p(x)=px$ denote multiplication by $p$ on the torus $\T=\R/\Z$. Whenever appropriate, we identify $\T$ with $[0,1)\subset \R$. In~\cite{Furstenberg70}, H.~Furstenberg proposed the following "transversality" conjecture: if $p,q\geq 2$ are multiplicatively independent, that is, $\log{p}/\log{q}\notin \Q$, and if $A$ and $B$ are closed subsets of $\T$ which are invariant\footnote{Here we mean forward invariance: $T_p(A)\subset A$ and $T_q(B)\subset B$.} under the maps $T_p$ and $T_q$, respectively, then
	\begin{equation} \label{eq:Furst-slicing-plane}
		\hdim(A\cap g(B)) \le  \max\{\hdim(A)+\hdim(B)-1,0\}
	\end{equation}
	for all affine maps $g\colon \R\to\R$. This conjecture was confirmed independently by the second author~\cite[Theorem~1.2]{Shmerkin-Annals} and by M.~Wu~\cite{Wu19}. Subsequently, H.~Yu~\cite{Yu21} refined and simplified part of Wu's approach, and T.~Austin~\cite{Austin21} found a remarkably streamlined proof of the conjecture based upon Furstenberg's original insights from~\cite{Furstenberg70}.
	
	Yu~\cite[Corollary 9.1]{Yu21} also established the following partial generalization of~\eqref{eq:Furst-slicing-plane}. If $p_1,\ldots,p_d$ are integers such that the family $(\log p_1/\log p_j)_{1\leq j\leq d}$ is linearly independent over $\Q$, and if $A_1,\ldots,A_d$ are closed subsets of $\T$ invariant under the maps $T_{p_1},\ldots,T_{p_d}$, respectively, then
	\[
	\hdim\bigl(g_1 (A_1)\cap g_2 (A_2)\cap \cdots \cap g_d (A_d)\bigr) = 0
	\]
	for all affine maps $g_j\colon \R\to\R$, provided that $\sum_{j=1}^d \hdim(A_j)\le d-1$. Unfortunately, for $d\geq 3$, there is no known $d$-tuple $(p_1,\dots,p_d)$ which satisfies the linear independence condition; in many natural cases, such as for the triple $(2,3,5)$, linear independence of the logarithm ratios is a special case of Schanuel's conjecture (which can be found in the historical note to Chapter III of Lang's transcendental number theory textbook~\cite{Lang}). Due to its reliance on Marstrand's slicing theorem, the approach pursued in~\cite{Wu19,Yu21,Austin21} seems ill-suited to dispense with the linear independence condition, as well as to get estimates for the Hausdorff dimension of intersections with affine subspaces of  intermediate dimension. The newly developed machinery of restricted projections might provide a way to do so, with significant additional work; see the work of Gan, Guo and Wang~\cite{GanGuoWang24}, and the references therein, for restricted projections, and the second author's work~\cite{AlgomShmerkin24} with Algom for an application of restricted projections to projections of self-similar measures. 
	
	The approach in~\cite{Shmerkin-Annals} does not depend on any projection or slicing theorem; instead, it combines the one-dimensional version of Theorem
	\ref{thm:lqdimension} with an elementary lemma relating $L^q$-dimensions of measures to box dimensions of fibers of Lipschitz maps via Frostman exponents; see Lemma~\ref{lem:Frostman-exp-to-small-fiber} and the related discussion. Here we resort to the same approach to prove the following generalization of~\eqref{eq:Furst-slicing-plane} to higher dimensions. Given any $\pi\in \G(d,k)$, $1\leq k\leq d$, let $Z(\pi)$ denote the largest number of zero coordinates of a unit vector contained in $\pi$. We also introduce an approximate version of this quantity as follows: given any real $\eta>0$ and a vector $v=(v_1,\dots,v_d)\in S^{d-1}$, the unit $(d-1)$-sphere, let
	\[
	Z_{\eta}(v) =| \{ j\in\{1,\ldots,d\}: | v_j|<\eta\}|\;.
	\]
	For $\pi\in\G(d,k)$, we then define
	\[
	Z_{\eta}(\pi) = \max\{ Z_{\eta}(v): v\in\pi\cap S^{d-1}\} \;.
	\]
	Note that $Z(\pi)=Z_{\eta}(\pi)$ if $\eta>0$ is small enough (depending on $\pi$). 
	
	We have the following generalization of \eqref{eq:Furst-slicing-plane}. If $P\subset \R^d$ is an affine subspace with associated linear space $V$, then $P^{\perp}$ denotes the linear subspace orthogonal to $V$.
	\begin{thm} \label{thm:Furst-higher-rank}
		Let $p_1,\ldots,p_d\geq 2$ be pairwise multiplicatively independent integers. Let $A_1,\ldots,A_d\subset \T$ be closed sets which are invariant under the maps $T_{p_1},\ldots,T_{p_d}$, respectively. Write
		\[
		s = \sum_{j=1}^d \hdim(A_j)\;.
		\]
		Then, for any $1\leq k\leq d$ and any $(d-k)$-dimensional affine subspace $P\subset \R^d$ with $Z(P^\perp)\le k-1$, 
		\[
		\ubdim\left( \bigl(A_1\times \cdots \times A_d\bigr) \cap P \right) \le  \max\left\{s-k,0\right\}\;.
		\]
		
		Moreover, the previous holds uniformly in the following sense. Fix $\eps>0$ and $\eta>0$; then there is $C_{d,k\eps,\eta}>0$ such that, for all $\pi\in\G(d,k)$ with $Z_{\eta}(\pi)\le k-1$ and all affine subspaces $P$ orthogonal to $\pi$, the inequality
		\[
		\left|\bigl(A_1\times \cdots \times A_d\bigr) \cap P  \right|_{\delta} \le C_{d,k,\eps,\eta}\, \delta^{-\max\{s-k,0\}-\eps}
		\]
		holds for all $0<\delta \leq 1$.
	\end{thm}

	\begin{rmk}
		As will emerge clearly from the proof of the theorem, presented in \textsection \ref{subsubsec:high-rank-Furst}, the sets $A_j$ can be replaced by homogeneous self-similar sets with contraction ratios $\lambda_j$, under the assumption $\log\lambda_i/\log\lambda_j\notin\Q$ for all $i\neq j$, where the generating IFS satisfies the open set condition. In fact, the open set condition can be weakened to exponential separation using the argument in the proof of Theorem \ref{thm:slices-sss}, which illustrates a further application of our results to the dimension of slices of self-similar sets.
	\end{rmk}
	
	Instrumental to the proof of Theorem~\ref{thm:Furst-higher-rank} is an $L^q$-dimension formula for projections of products of self-similar measures, which is the content of Theorem~\ref{thm:proj-product-ssm}.

	The case $k=d-1$ of Theorem \ref{thm:Furst-higher-rank} provides a substantive generalization of the aforementioned result of Yu~\cite[Corollary 9.1]{Yu21}.
	\begin{cor}
		\label{cor:Hugen}
		Let $p_1,\ldots,p_d$, $A_1,\dots,A_d$ and $s$ be as in Theorem~\ref{thm:Furst-higher-rank}.
		Then, for all affine maps $g_1,\ldots,g_d\colon \R\to\R$,
		\[
		\ubdim\bigl( g_1(A_1) \cap \cdots \cap g_d(A_d) \bigr) \le  \max\left\{s-(d-1),0\right\}\;.
		\]
		Moreover, for any $\eps>0$ and $K\ge 1$, there is $C_{\eps,K}>0$ such that, if the slopes of the $g_j$'s are bounded above by $K$, then
		\[
		\left| g_1(A_1)\cap \cdots \cap g_d(A_d) \right|_{\delta} \le C_{\eps,K}\, \delta^{-\max\left\{s-(d-1),0\right\}-\eps}
		\]
		for all $0<\delta\leq 1$.
	\end{cor}
	
	The (rather straightforward) deduction of the corollary from Theorem~\ref{thm:Furst-higher-rank} is given in \textsection \ref{subsubsec:high-rank-Furst}.

	\subsection{Overview of the argument}
	
	Let us now provide an overview of the proof of Theorem \ref{thm:lqdimension}.
	
	For every $a\in \R$ and $B\subset \R^d$, we let $aB=\{ av:v\in B\}$. For any finite Borel measure $\mu$ on $\R^{d}$ and any integer $m\geq 0$, define the $2^{-m}$-discretization of $\mu$ as the measure $\mu^{(m)}$, supported inside the lattice $2^{-m}\Z^{d}=\{2^{-m}k:k\in \Z^d\}$, given by
	\begin{equation*}
		\mu^{(m)}(2^{-m}k)=\mu(2^{-m}k+2^{-m}[0,1)^d)\;, \quad k\in \Z^d.
	\end{equation*}
	If $\nu=\sum_{x\in \R^d} \nu(x)\delta_x$ is a finitely supported measure on $\R^d$, we define the $L^{q}$-norm of $\nu$ as $\norm{\nu}_q=\bigl(\sum_{x\in \R^d}\nu(x)^{q}\bigr)^{1/q}$.
	
	Fix now $q>1$ and $\sx\in \sX$. We may express the inferior limit defining $\tau_{\mu_\sx}(q)$ using $L^{q}$-norms of discretized measures, as
	$\liminf_{n\to\infty}-\frac{1}{n}\log{\norm{\mu_\sx^{(n)}}}_q^q$.
	As is often the case in dimension computations arising in fractal geometry, the upper bound
	\begin{equation}
		\label{eq:upperbound}
		\dim(\mu_{\sx},q)\leq  \frac{\int_{\sX}\log{\norm{\Delta}_q^q}\;\text{d}\bP}{(q-1)\log{\lambda}}
	\end{equation}
	is straightforward to establish and holds for every $\sx\in \sX$ without any additional assumption on the pleasant model $\cX$. See Proposition \ref{prop:upperbound}.
	
	Considerably more delicate is the proof of the lower bound
	\begin{equation*}
		\dim(\mu_{\sx},q)\geq  \frac{\int_{\sX}\log{\norm{\Delta}_q^q}\;\text{d}\bP}{(q-1)\log{\lambda}}\;.
	\end{equation*}
	Let $m(n)$ be chosen so that $\lambda^{n+1}<2^{-m(n)}\leq \lambda^{n}$. The first observation is that the dynamical self-similarity relation in~\eqref{eq:dynamicalselfsimilar} implies that the sequence of functions $ \sx\mapsto \norm{\mu_{\sx}^{(m(n))}}_q^{q}$, $n\in \N$, is sub-multiplicative, up to a constant depending only on $\lambda$ and $q$ (see Proposition~\ref{prop:subadditive}). By the sub-additive ergodic theorem, there is a non-negative real number $T_{\cX}(q)$ such that
	\begin{equation} \label{eq:kingman-Tq}
		\liminfl_{n\to\infty}-\frac{\log{\norm{\mu_{\sx}^{m(n)}}_q^q}}{m(n)}=T_{\cX}(q)
	\end{equation}
	for $\bP$-almost every $\sx\in \sX$. Crucially, unique ergodicity and a variant of the sub-additive ergodic theorem (see Lemma \ref{lem:K-W}) yield that
	\begin{equation*}
		\tau_{\mu_{\sx}}(q)=\liminfl_{n\to\infty}-\frac{\log{\norm{\mu_{\sx}^{m(n)}}_q^q}}{m(n)}\geq T_{\cX}(q)\quad\text{uniformly in }\sx\in\sX,
	\end{equation*}
	cf.~Proposition~\ref{prop:Tq}.
	
	In order to complete the proof of Theorem~\ref{thm:lqdimension}, it suffices then to show that
	\begin{equation}
		\label{eq:secondlb}
		T_{\cX}(q)=   \frac{\int_{\sX}\log{\norm{\Delta}_q^q}\;\text{d}\bP}{\log{\lambda}}\;.
	\end{equation}
	Note that, even though $T_{\cX}(q)$ is defined as a $\mathbf{P}$-almost everywhere limit, thanks to unique ergodicity~\eqref{eq:secondlb} is sufficient to deduce the claim for \emph{every} $\sx\in\sX$.
	
	The key, and most difficult, step in the proof of~\eqref{eq:secondlb} is the following $L^q$-smoothening property of the discretized measures $\mu_{\sx,n}$, which holds assuming $q$-unsaturation: suppose $\nu\in\cA_d$ is not too close to a single atom at scale $\lambda^n \sim 2^{-m(n)}$, in the quantitative form
	\[
	\|\nu^{(m(n))}\|_q \le 2^{-\sigma m}
	\]
	for some $\sigma>0$, which we think of as very small. Then $\nu \ast \mu_{\sx,n}$ is much smoother, in the $L^q$ sense, than $\mu_{\sx,n}$ itself:
	\begin{equation} \label{eq:sketch-flattening}
		\|\nu\ast\mu_{\sx,n}\|_q \le 2^{-(T(q)+\eps)m},
	\end{equation}
	where $\eps>0$ depends only on $\sigma$ and $q$. This is the content of Theorem \ref{thm:lqflattening}.
	
	Equipped with this flattening property, the rest of the proof follows very closely that of \cite[Theorem 1.11]{Shmerkin-Annals}, as well as the overarching strategy in~\cite{Hochman14}. The first consequence is that, for $\bP$-almost every $\sx\in \sX$, the $L^q$-norm of the measure $\mu_{\sx,n}$ at scale $r$ stays roughly constant between the natural scale $\lambda^n$ and much finer scales $\lambda^{Rn}$, where $R$ is a fixed but arbitrarily large parameter. This is the content of Proposition \ref{prop:finerscales}. It should be interpreted as saying that, in an $L^q$-sense, the measure $\mu_{\sx,n}$ has very few atoms that are at distance between $\lambda^{Rn}$ and $\lambda^{n}$ apart. Roughly speaking, this follows from \eqref{eq:sketch-flattening} as follows: if there were many such atoms, then locally (at scale $\lambda^n$) the measure $\mu_{\sx}$ would look like a convolution of $\mu_{\bT^n\sx}$ with a measure $\nu$ arising from these atoms. But then from $L^q$-flattening one would get
	\[
	\|\mu_{\sx}^{(R-1)m(n)}\|_q^q \le 2^{-(T(q)+\eps)(R-1)m(n)},
	\]
	which contradicts \eqref{eq:kingman-Tq} (for $\mathbf{P}$-almost every $\sx$).

	So far the exponential separation assumption has not been used. It now comes into play in the following form: for $\bP$-almost every $\sx\in \sX$, there is $R\ge 1$ and arbitrarily large $n$ such that
	\[
	\| \mu_{\sx,n} \|_q^q = \|\mu_{\sx,n}^{(Rm(n))}\|_q \approx \|\mu_{\sx,n}^{(m(n))}\|_q^q \approx \|\mu_{\sx}^{(m(n))}\|_q^q \approx 2^{-T(q)m(n)},
	\]
	where the equality is simply exponential separation, the first approximation is due to the previous step, and the second approximation is an elementary fact that reflects that, at scale $\lambda^n$, the measures $\mu_{\sx}$ and $\mu_{\sx,n}$ are almost indistinguishable. However, $\|\mu_{x,n}\|_q^q$ can be easily computed as an ergodic sum, yielding the desired expression ~\eqref{eq:secondlb} and completing the proof.
	
	The overall strategy stays close to that of \cite{Shmerkin-Annals}, which in turn was inspired by \cite{Hochman14}. However, the proof of the key flattening property \eqref{eq:sketch-flattening} is substantially different from the corresponding one in \cite{Shmerkin-Annals}.
	Both arguments use an inverse theorem for the $L^q$-norm of convolutions, which delivers a certain multiscale structure for the convolution factors $\nu$ and $\mu_{\sx,n}$ if \eqref{eq:sketch-flattening} does not hold. In the case $d=1$, this structure is of the following form: in a multiscale decomposition of $\mu_{\sx,n}$ and $\nu$, at ``almost all'' scales one of two things happens: either the measure $\mu_{\sx,n}$ is locally "almost uniform" or the measure $\nu$ is locally ``almost atomic'' (the actual statement is far more involved). But the measure $\nu$ cannot be ``almost atomic'' at ``almost all'' scales, because then $\nu$ would be globally ``almost atomic'', contradicting the assumption. Hence, $\mu_{\sx,n}$ has to be ``almost uniform'' at ``positively many'' scales. In \cite{Shmerkin-Annals}, the dynamical self-similarity of $\mu_{\sx,n}$ together with ideas from multifractal analysis are used to rule this out, provided that the derivative $T'_{\cX}(q)$ exists (a condition that is known to hold for almost all $q$).
	
	In higher dimensions, we appeal to the more involved inverse theorem in \cite{Shmerkin23}, in turn inspired by a conceptually similar inverse theorem for the entropy of convolutions in \cite{Hochman17}. Assuming again that \eqref{eq:sketch-flattening} does not hold, the structure is now of the following more complicated form: for ``almost all'' scales, there is a dimension $0\le k\le d$ (depending on the scale) such that, locally, the measure $\mu_{\sx,n}$ is roughly the product of a uniform measure on a $k$-plane and some other arbitrary measure, and $\nu$ is supported inside a $k$-plane. Since $k\in\{0,1\}$ in dimension $1$, this is a generalization of the one-dimensional case. The issue is that knowing that $\nu$ is globally far from atomic no longer implies that $\mu_{\sx,n}$ is ``almost uniform'' at positively many scales; in order to deduce this, we would need to know instead that $\nu$ is locally not concentrated in hyperplanes, a piece of information we lack. Thus, the proof in~\cite{Shmerkin-Annals} breaks down.
	
	To circumvent this issue, we introduce a new idea. What does follow from the inverse theorem is that there are positively many scales at which $\mu_{\sx,n}$ looks roughly like the product of a uniform measure on a line and some other measure. But this is the kind of structure ruled out by $q$-unsaturation on lines. There are many challenges in implementing this idea. For one, unsaturation is a global and non-uniform condition; as a consequence, part of the work towards \eqref{eq:sketch-flattening} is deriving pointwise, local, uniform versions; see Corollary~\ref{cor:unsaturation-uniform} and Proposition~\ref{prop:gain-from-unsaturation}. These use dynamical self-similarity in a crucial way. Even then, there is no immediate contradiction to the inverse theorem - rather, we have to use a refining technique to eventually conclude that there is a set that simultaneously carries too much and too little of the $L^q$-norm of $\mu_{\sx,n}$, which is the desired contradiction.
	
	This novel approach has the advantage of dispensing with the need for $T_{\cX}'(q)$ to exist, which seemed like an artificial condition in the one-dimensional case, as well as with the associated multifractal estimates. A more detailed sketch of the argument is provided after the statement of Theorem~\ref{thm:lqflattening}.

	\subsection{Notation and conventions}
	\label{subsec:notation}
	We conclude this introduction by collecting some relevant notation and terminology to be adopted throughout the manuscript. 
	
	\subsubsection{General notation}
	\label{subsubsec:generalnotation}
	
	We let $\N$ be the set of natural numbers, $\N^{\ast}=\N\setminus\{0\}$.
	
	If $I$ is a finite set, $|I|$ denotes its cardinality. 
	
	A probability vector in $\R^{I}$ is an element $p=(p_i)_{i\in I}\in \R^{I}$ with $p_i\geq 0$ for every $i\in I$ and  $\sum_{i\in I}p_i=1$; we call it positive if $p_i>0$ for every $i \in I$. 
	
	If $q>1$ is a real number, $q'$ denotes its Hölder conjugate, defined by the relation $\frac{1}{q}+\frac{1}{q'}=1$.
	
	If $A,B$ are subsets of an abelian group $(G,+)$, we let $A+B=\{a+b:a\in A,\;b\in B  \}$. When $A=\{a\}$, we write $a+B$. If $\lambda\in \R$ and $A$ is a subset of $\R^d$, then we indicate with $\lambda A$ the set $\{\lambda a:a\in A\}$.

	If $(X,d_X)$ is a metric space (implicit from context), we indicate with $B(x,r)$ the closed ball of radius $r>0$ centered at a point $x\in X$.
	
	We will always use $d\geq 1$ to denote the ambient (Euclidean) dimension. The Euclidean norm of a vector $z\in \R^d$ is denoted by $|z|$.

	For every $\lambda\in \R$, $S_{\lambda}\colon \R^d\to \R^d$ is the scaling map $x\mapsto \lambda x$. 
	
	For every integer $k\in \{0,\dots,d\}$, the Grassmannian of $k$-dimensional vector subspaces of $\R^d$ is denoted $\G(d,k)$, and the affine Grassmannian of $k$-dimensional affine subspaces of $\R^d$ is denoted $\A(d,k)$. We also write $\G(d)$ for the the disjoint union of the $\G(d,k)$ for $1\leq k\leq d$\footnote{We do not include the zero subspace as it is ruled out in any statements where such shorthand notation will be convenient.}.  If $\pi\in \G(d,\ell)$ for some $0\leq \ell \leq d$ and $k\in \{0,\dots,\ell\}$, we indicate with $\G(\pi,k)$ the subset of $\G(d.k)$ whose elements are the $k$-dimensional subspaces of $\pi$.

	We avail ourselves of Landau's notation with the following meaning. If $X$ is an arbitrary set and $f,g\colon X\to \R$ are two functions, we write $f=O(g)$ if there exists a real number $C>0$, referred to as the \emph{implicit constant}, such that $ f(x)\leq Cg(x)$ for every $x\in X$. Whenever there are several functions $(O(g))_1,\dots,(O(g))_k$ appearing in the course of a proof, we omit the subscripts $1,\dots,k$ to lighten notation.
	When $g$ is the function constantly equal to $1$, we simply write $f=O(1)$.
	If $\beta_1,\dots,\beta_s$ are real numbers, then we write $f=O_{\beta_1,\dots,\beta_s}(g)$ if the implicit constant $C$ depends on $\beta_1,\dots,\beta_s$ and nothing else. Analogous conventions we adopt for the symbols $f=\Omega(g)$, equivalent to $g=O(f)$, and $f=\Theta(g)$, equivalent to $f=O(g)$ and $g=O(f)$.
	
	At times we shall also employ Vinogradov's notation $A\ll B$ for two real quantities $A$ and $B$, meaning that there is some implicit constant $C>0$ such that $A\leq C B$.

	\subsubsection{Dyadic cubes}
	
	For every $m\in \N$, we let $\cD_m$ denote the partition of $\R^d$ into half-open $2^{-m}$-mesh cubes, that is,
	\begin{equation} \label{eq:def-mesh-cubes}
		\cD_m=\biggl\{ \prod_{i=1}^{d}\;\biggl[\frac{k_i}{2^{m}},\frac{k_i+1}{2^{m}}\biggr): (k_1,\dots,k_d)\in \Z^d  \biggr\}=\{2^{-m}(k+[0,1)^d):k\in \Z^d \}\;.
	\end{equation}
	If $A$ is a subset of $\R^d$ and $m\in \N$, we indicate with $\cD_{m}(A)$ the collection of cubes $Q\in \cD_m$ such that $Q\cap A\neq \emptyset$, and we let $\cN_{m}(A)=|\cD_{m}(A)|$ be the box-counting number of $A$ at scale $2^{-m}$.

	As we confine ourselves to the consideration of cubes of side length $2^{-m}$, we adopt the convenient convention that all logarithms appearing in the manuscript are to the base $2$.
	
	If $Q$ is a cube and $C>0$ is a constant, by $C\cdot Q$ we denote the cube with the same center as $Q$ and side length $C\cdot \ell(Q)$, where $\ell(Q)$ is the side length of $Q$.

	\subsubsection{Measures and their discretizations}
	\label{subsubsec:measdisc}
	
	A measure on a measurable space $(X,\cA)$ is always meant to be positive and finite.  If $f\colon (X,\cA) \to (Y,\cB)$ is a measurable map and $\mu$ is a measure on $(X,\cA)$, we indicate with $f \mu$ the push-forward of $\mu$ under $f$.

	If $\mu$ is a Borel measure on a topological space $X$, $\supp{\mu}$ denotes its topological support. If $x$ is a point in $X$, the notation $\delta_x$ stands for the Dirac mass at $x$.
	
	For every Borel measure $\mu$ on $\R^d$ and every $n\in \N$, we define the $2^{-n}$-discretization of $\mu$ as the following sum of weighted Dirac masses:
	\begin{equation} \label{eq:def-mu-n}
		\mu^{(n)}=\sum_{k\in \Z^d}\mu\bigl(2^{-n}(k+[0,1)^d)\bigr)\;\delta_{2^{-n}k}\;;
	\end{equation}
	in other words, $\mu^{(n)}$ is obtained by concentrating all the mass of $\mu$ lying inside the $2^{-n}$-cube $2^{-n}(k+[0,1)^d)$ on the single point $2^{-n}k$, for every $k\in \Z^d$.
	
	We let $\cA_d$ be the set of finitely supported Borel probability measures on $\R^d$ equipped with the weak$^*$ topology. 
	In particular, the natural map
	\begin{equation*}
		((x_1,\dots,x_{N+1}),(p_1,\dots,p_{N+1}))\mapsto \sum_{i=1}^{N+1}p_i\;\delta_{x_i}\;.
	\end{equation*}
	is continuous for all $N\in\N$.
	

	If $Y$ is a set and $f\colon Y\to \R_{\geq 0}$ is a function, we write
	\begin{equation*}
		\norm{f}_q=\biggl(\sum_{y\in Y}f(y)^q\biggr)^{1/q}
	\end{equation*}
	for its (somewhat non-standard, possibly infinite) $L^q$-norm, for every real $q\geq 1$. If $\nu$ is a discrete measure on $\R^d$, that is, if every element of its support is an atom, then we identify it canonically with a function $\R^d\to \R_{\geq 0}$, and employ the previously defined notation $\norm{\nu}_q$ accordingly.


	
	\section{Preliminaries}

	\subsection{An inverse theorem for the \texorpdfstring{$L^q$-}{ }norm of convolutions}
	\label{subsec:inversethm}
	
	In this section we present a corollary of the inverse theorem for the $L^q$-norm of convolutions, which is the main result of \cite{Shmerkin23}. This will be one of the key ingredients in the proof of Theorem \ref{thm:lqdimension}. Let us introduce some terminology and notation first.
	\begin{itemize}
		\item If $S\geq 1$ is an integer, we let $[S]=\{0,\dots,S-1\}$.
		\item If $m\geq 1$ is an integer, a \emph{$2^{-m}$-measure} is a probability measure whose support is contained in the lattice $2^{-m}\Z^d\cap [0,1)^d$.
		\item For any $x\in [0,1)^d$ and any $m\in \N$, we let $\cD_{m}(x)$ be the unique element of $\cD_m$ containing the point $x$. 
		\item If $E$ is a subset of $\R^d$ and $r>0$, we indicate by $E^{(r)}$ the open $r$-neighborhood of $E$.
	\end{itemize}
	
	We can  now state the inverse theorem, in the form that will be required in our application.
	\begin{thm}[Inverse theorem for $L^{q}$-norms]
		\label{thm:invthm}
		Fix real numbers $q>1$ and $\delta>0$. There is a positive integer $L_0=L_0(q,\delta)$ such that, for all $L\geq L_0$, there is $\eps=\eps(q,\delta,L)>0$ such that the following holds for every integer $S\geq S_0(q,\delta,L)$:
		
		Let $m=S L$, and let $\mu,\nu$ be $2^{-m}$-measures on $[0,1)^d$ such that
		\begin{equation*}
			\norm{\mu\ast \nu}_q\geq 2^{-\eps m}\norm{\mu}_{q}\;.
		\end{equation*}
		Then, there exist subsets $A\subset \supp{\mu},\;B\subset \supp{\nu}$ such that the following properties hold:
		\begin{labeledlist}{A}
			\item \label{IT:A1} $\norm{\mu|_{A}}_q\geq 2^{-\delta m}\norm{\mu}_q$;
			\item \label{IT:A2} $\mu(x)\leq 2\mu(y)$ for every $x,y\in A$;
			\item \label{IT:A3} for every $s\in [S]$ there exists an integer $R_{s}'$ such that, for every $Q\in \cD_{sL}(A)$,
			\begin{equation*}
				\cN_{(s+1)L}(A\cap Q)=R_{s}'\;;
			\end{equation*}
		\end{labeledlist}
		\begin{labeledlist}{B}
			\item \label{IT:B1}  $\nu(B)\geq 2^{-\delta m}$;
			\item \label{IT:B2} $\nu(x)\leq 2\nu(y)$ for any $x,y\in B$;
			\item \label{IT:B3} for every $s\in [S]$ there exists an integer $R_{s}''$ such that, for every $Q\in \cD_{sL}(B)$,
			\begin{equation*}
				\cN_{(s+1)L}(B\cap Q)=R_{s}''\;;
			\end{equation*}
		\end{labeledlist}
		
		If $\supp\mu\cup\supp\nu\subset [1/3,2/3)^d$ then, upon translating the supports of $\mu$ and $\nu$ by appropriate elements in $2^{-m}\Z^d \cap [-1/3,1/3]^d$, we may additionally assume that
		\begin{labeledlist}{C}
			\item  $x\in \frac{1}{2}\cdot \cD_{sL}(x)$ for every $x\in A\cup B$ and $s\in [S]$.
		\end{labeledlist}

		Furthermore, for each $s\in [S]$, at least one of the following two alternatives holds:
		\begin{enumerate}[\textup{(\roman*)}]
			\item \label{IT:i} $R_{s}''=1$;
			\item \label{IT:ii} for each $Q\in\cD_{sL}(A)$ there is a projection $\pi_Q\in\G(d,d-1)$ such that
			\begin{equation} \label{eq:saturation}
				\cN_{(s+1)L}(\pi_Q(A\cap Q)) \leq C_d\; 2^{(\delta-1)L}\cN_{sL}(A\cap Q) = 2^{(\delta-1)L}R'_s
			\end{equation}
			where $C_d>0$ only depends on $d$.
		\end{enumerate}
		
		Finally,  if $\cS_0=\{s\in [S]: R_{s}''=1  \}$, then
		\begin{equation}
			\label{eq:somebranching}
			|[S]\setminus \cS_0|L\geq  \frac{-\log{\norm{\nu}_q^{q'}}-q'(\delta m+1)}{d} \;.
		\end{equation}
	\end{thm}

	\begin{proof}
		
		All claims are given by~\cite[Theorem 1.2]{Shmerkin23} except for the dichotomy \ref{IT:i}--\ref{IT:ii} and the lower bound \eqref{eq:somebranching}, which we set out to prove.
		
		By \cite[Theorem 1.2]{Shmerkin23} there is a sequence $(k_s)_{s\in[S]}$ of nonnegative integers such that the following two properties are satisfied.
		\begin{enumerate}[\textup{(\alph*)}]
			\item  \label{eq:it-a} For each $s\in[S]$ and each $Q\in\cD_{sL}(A)$, there is $\pi_Q\in\G(d,d-k_s)$ such that
			\[
			\cN_{(s+1)L}\left(\pi_Q(A\cap Q)\right) \le 2^{(\delta-k_s) L}  R'_s.
			\]
			\item \label{eq:it-b} For each $s$ and each $Q\in\cD_{sL}(B)$, there is $V_Q\in\mathbb{A}(d,k_s)$ such that $Q\cap B\subset \bigcup \cD_{(s+1)L}(V_Q)$.
		\end{enumerate}
		
		If $k_s=0$, then \ref{eq:it-b} yields that $R''_s=1$, that is, \ref{IT:i} holds. If $k_s\ge 1$, then we can write $\Pi_Q = \Pi_Q \Pi'_Q$ for some projection $\Pi'_Q\in\G(d,d-1)$ (take an arbitrary $\Pi'_Q\subset \Pi_Q$, identifying projections with their images). Then
		\[
		\cN_{(s+1)L}(\Pi_Q(A')) \geq C_d 2^{-(k_s-1)L}\cN_{(s+1)L}(A')
		\]
		for each $A'\subset \Pi'_Q$, for some $C_d>0$ only depending on $d$. Applying this to $A' = \Pi'_Q(A\cap Q)$ and combining it with \ref{eq:it-a}, we see that \ref{IT:ii} must hold.
		
		To show \eqref{eq:somebranching}, we start by noting that
		\begin{equation}
			\label{eq:Bsetofscales}
			|B|=\prod_{s\in \cS_0}R''_s\prod_{s\in [S]\setminus\cS_0}R''_s=\prod_{s\in [S]\setminus \cS_0}R''_s\leq 2^{d L |[S]\setminus \cS_0|}\;,
		\end{equation}
		where the last inequality follows from the trivial bound $R_s''\leq 2^{d L}$, holding for any $s\in [S]$. Moreover, from \ref{IT:B1}--\ref{IT:B2}  we infer that
		\begin{align*}
			2^{-\delta m}|B|^{1/q} & \leq \nu(B) |B|^{1/q}\leq |B|\biggl(\supl_{x\in B}\nu(x)\biggr)|B|^{1/q}    \\
			& \leq 2|B|\biggl(\inf_{x\in B}\nu(x)\biggr)|B|^{1/q}\leq 2|B|\norm{\nu}_q\;,
		\end{align*}
		that is,
		\begin{equation*}
			|B|\geq 2^{-q'(\delta m +1)}\norm{\nu}_q^{-q'}\;,
		\end{equation*}
		which combined with~\eqref{eq:Bsetofscales} yields
		\begin{equation*}
			|[S]\setminus \cS_0| L\geq \frac{\log{|B|}}{d}\geq \frac{1}{d}(-\log{\norm{\nu}_q^{q'}}-q'(\delta m+1)).
		\end{equation*}
		
	\end{proof}
	
	\subsection{The structure of dynamically driven self-similar measures}
	
	For the remainder of the section, we fix a pleasant model $\cX=(\sX,\bT,\bP,\Delta,\lambda)$ in $\R^d$, generating a family $(\mu_{\sx})_{\sx\in\sX}$ of dynamically driven self-similar measures.

	We highlight that the results of this section are valid without any separation or unsaturation assumption on the model, and are straightforward variants of results from \cite{Shmerkin-Annals}, when not completely elementary. Furthermore, we adopt the convention that all implicit constants are allowed to depend on the ambient dimension $d$. Any other dependencies will be made explicit.

	Scaling and translating a measure does not affect its $L^{q}$-spectrum; therefore, upon rescaling the measures $\Delta(\sx),\;\sx\in \sX$ by a common factor and translating them by a common vector\footnote{Observe that this procedure does not alter the right-hand side of~\eqref{eq:Lqdimensionformula} either, since the latter only gauges how the masses of the measures $\Delta(\sy),\;\sy\in \sX$ are distributed among the respective atoms.}, we may and shall assume, hereinafter, that $B=[0,1-\lambda)^{d}$, where $B$ is as in Definition~\ref{def:pleasantmodel}.
	With this assumption, the support of every measure $\mu_\sx$ is contained in the half-open cube $[0,1)^{d}$.

	\subsubsection{Auxiliary lemmas}
	
	We collect some elementary lemmas that will be used throughout the rest of the paper. 
	\begin{lem}
		\label{lem:qnorm}
		Let $t_1,\dots,t_k$ be non-negative real numbers, $q>1$. Then
		\begin{equation}
			\label{eq:powerofsum}
			k^{-(q-1)}(t_1+\cdots +t_k)^{q} \le	t_1^{q}+\cdots +t_k^{q}\le (t_1+\cdots +t_k)^{q} \;,
		\end{equation}
		where equality holds on the left-hand side if and only if $t=(t_i)_{1\leq i\leq k}$ is a uniform vector, and on the right-hand side if and only if $t_i=0$ for all but one $1\leq i\leq k $.
		
		As a consequence, for every finite Borel measure $\mu$ with finite support on $\R^d$ and every $n\in \N$, we have
		\begin{equation}
			\label{eq:discrlargernorm}
			\norm{\mu^{(n)}}_q^{q}\geq \norm{\mu}_q^q\;.
		\end{equation}
		Furthermore, if $\nu$ is a second finite Borel measure with finite support on $\R^d$, then
		\begin{equation}
			\label{eq:lowerboundqnorms}
			\norm{\mu\ast \nu}_q^{q}\geq \norm{\mu}_q^{q}\norm{\nu}_q^{q},
		\end{equation}
		with equality if and only if $\supp{\mu\ast \nu}$ has maximal cardinality $|\supp{\mu}||\supp{\nu}|$.
	\end{lem}
	\begin{proof}
		This is elementary analysis.
	\end{proof}
	
	The next lemma will allow us to pass between ``comparable'' families of cubes. 
	\begin{lem}
		\label{lem:measurefinitecover}
		Let $(Y,\cB,\mu)$ be a probability space, $\cP$ and $\cS$ subsets of $\cB$. Suppose that there exists an integer $M_0\geq 1$ such that every element of $\cP$ can be covered by at most $M_0$ elements of $\cS$, and every element of $\cS$ intersects non-trivially at most $M_0$ elements of $\cP$. Then, for every\footnote{Both the left-hand and the right-hand side of~\eqref{eq:measurecover} are allowed to be infinite.} $q>1$,
		\begin{equation}
			\label{eq:measurecover}
			\sum_{P\in \cP}\mu(P)^{q}\leq M_0^{q}\sum_{S\in \cS}\mu(S)^{q}\;.
		\end{equation}
	\end{lem}
	We refer to~\cite[Lemma 4.1]{Shmerkin-Annals} for the short proof, consisting of an elementary application of Hölder's inequality. 
	
	We let $\mathscr{D}_r$ denote, for every $r>0$, the collection of $r$-mesh cubes in $\R^d$, that is,
	\begin{equation*}
		\mathscr{D}_r=\bigl\{r\bigl(k+[0,1)^{d}\bigr):k\in \Z^d   \bigr\}\;.
	\end{equation*}
	With our earlier notation, we have $\mathscr{D}_{2^{-n}}=\cD_{n}$ for every $n\in \N$.
	
	\begin{lem}
		\label{lem:equallimits}
		Let $\mu$ be a compactly supported Borel probability measure on $\R^{d}$. Let $(r_n)_{n\in \N}$ be a strictly decreasing sequence in $(0,1)$ such that $r_{n}\to 0$ as $n\to\infty$. If
		\begin{equation}
			\label{eq:expnonlac}
			\lim\limits_{n\to\infty}\frac{\log{r_{n+1}}}{\log{r_n}}=1\;,
		\end{equation}
		then, for every $q>1$,
		\[
		\lim_{n\to\infty}\sup_{r\in [r_{n+1},r_n]} \left| \frac{\log{\sum_{Q\in \mathscr{D}_r}\mu(Q)^{q}}}{\log{r}} - \frac{\log{\sum_{Q\in \mathscr{D}_{r_n}}\mu(Q)^{q}}}{\log{r_n}}\right|  = 0.
		\]
		In particular,
		\begin{equation}
			\label{eq:limsupequal}
			\begin{split}
				\liminfl_{r\to 0}\frac{\log{\sum_{Q\in \mathscr{D}_r}\mu(Q)^{q}}}{\log{r}}&= 	\liminfl_{n\to \infty}\frac{\log{\sum_{Q\in \mathscr{D}_{r_n}}\mu(Q)^{q}}}{\log{r_n}}\;,\\
				\limsupl_{r\to 0}\frac{\log{\sum_{Q\in \mathscr{D}_r}\mu(Q)^{q}}}{\log{r}}&= 	\limsup_{n\to \infty}\frac{\log{\sum_{Q\in \mathscr{D}_{r_n}}\mu(Q)^{q}}}{\log{r_n}}\;.
			\end{split}
		\end{equation}
	\end{lem}
	\begin{proof}
		The statement is a straightforward consequence of Lemma~\ref{lem:measurefinitecover}; we omit the details.
	\end{proof}

	The next lemma, again a corollary of H\"{o}lder's inequality, is \cite[Lemma 4.2]{Shmerkin-Annals}.
	\begin{lem}
		\label{lem:finitesupports}
		Let $(\mu_j)_{j\in J}$ be a family of finite measures of finite support on $\R^{d}$, and let $\mu=\sum_{j\in J}\mu_j$. Suppose that there exists an integer $N\geq 1$ such that each point $x\in \R^d$ belongs to the support of at most $N$ elements of the collection $(\supp\mu_{j})_{j\in J}$. Then, for every $q>1$,
		\begin{equation*}
			\norm{\mu}_q^{q}\leq N^{q-1}\sum_{j\in J}\norm{\mu_j}_q^{q}\;.
		\end{equation*}
	\end{lem}

	The following lemma expresses the fact that, as far as $L^{q}$-norms are concerned, discretization commutes with convolution up to some uniform multiplicative factor. 
	\begin{lem}
		\label{lem:discrconv}
		Let $\mu,\nu$ be compactly supported Borel measures on $\R^d$. Then, for every $n\in \N$,
		\begin{equation*}
			\norm{(\mu\ast \nu)^{(n)}}_q^{q} =\Theta_q(1) \norm{\mu^{(n)}\ast \nu^{(n)}}_q^{q}\;.
		\end{equation*}
	\end{lem}

	The case $d=1$ is~\cite[Lemma 4.3]{Shmerkin-Annals}, and the general case follows with the same proof, relying on Lemma \ref{lem:measurefinitecover}.
	
	\subsubsection{General upper bound for the $L^q$-dimensions}
	We set
	\[
	m(n)=\lceil n\log{\lambda}^{-1}\rceil, \quad n\in \N,
	\]
	so that $\lambda^{n+1}<2^{-m(n)}\leq \lambda^{n}$ for all $n$. The following lemma quantifies the intuitively obvious fact that, since a rescaled measure $S_{\lambda^{n}}\mu_{\bT^n\sx}$ lives at scale $\lambda^n$, which is roughly $2^{-m(n)}$, its contribution to the $L^{q}$-norm of the $2^{-m(n)}$-discretized version of $\mu_{\sx}=\mu_{\sx,n}\ast S_{\lambda^n}\mu_{\bT^n\sx}$ is negligible. The case $d=1$ is \cite[Lemma 4.4]{Shmerkin-Annals}, and the general case is identical.
	\begin{lem}
		\label{lem:normdiscretization}
		For every $\sx\in \sX$ and $n\in \N$,
		\begin{equation*}
			\norm{\mu_\sx^{(m(n))}}_q^{q} =\Theta_{\lambda,q}(1) \norm{\mu_{\sx,n}^{(m(n))}}_q^{q}\;.
		\end{equation*}
	\end{lem}

	The chief aim of this subsection is to prove the easier inequality in~\eqref{eq:Lqdimensionformula}, namely that
	\begin{equation*}
		\dim_{\mu_{\sx}}(q)\leq \min\biggl\{ \frac{\int_{\sX}\log{\norm{\Delta(\sy)}_q^q}\;\text{d}\bP(\sy)}{\log{(q-1)\lambda}} ,d\biggr\}
	\end{equation*}
	for every $q>1$ and every  $\sx\in \sX$; this holds without any further assumption on the pleasant model. In light of Lemmas~\ref{lem:equallimits} and \ref{lem:normdiscretization}, we shall reduce matters to the investigation of the asymptotics of $\log{\norm{\mu_{\sx,n}^{(m(n))}}_q^q}$ as $n\to\infty$, which in turn we will relate to the Birkhoff sum $\sum_{i=0}^{n-1}\log{\norm{\Delta(\bT^i\sx)}_q^q}$ by virtue of the convolution structure of $\mu_{\sx,n}$.
	
	As we pointed out earlier, it is well known that in the setting of uniquely ergodic systems (see~\textsection\ref{subsec:models}), time averages of a continuous observable converge uniformly to its space average. It turns out that in the sub-additive setting this fails to hold, as shown by Furman in~\cite{Furman}. This pathological behavior notwithstanding, it is possible to retain a one-sided inequality, holding uniformly in the case of uniquely ergodic systems. We express the latter in the following proposition, which is formulated in the more general case of almost-everywhere continuous observables, in accordance with the needs of our argument. 
	\begin{prop}
		\label{prop:Furmanadapt}
		Let $(\sX,\bT,\bP)$ be a uniquely ergodic system, $(\psi_n)_{n\geq 1}$ a sequence of bounded measurable functions $\psi_n\colon \sX \to \R$ satisfying the following properties:
		\begin{enumerate}[(\alph*)]
			\item for every $n\geq 1$, $\psi_n$ is continuous $\mathbf{P}$-almost everywhere;
			\item there exists $c\geq 0$ such that, for every $n,n'\geq 1$ and every $\sx\in \sX$,
			\begin{equation}
				\label{eq:subadditivity}
				\psi_{n'+n}(\sx)\leq \psi_{n'}(\bT^n\sx)+\psi_n(\sx)+c\;.
			\end{equation}
		\end{enumerate}
		Let
		\begin{equation*}
			L=\inf_{n\geq 1}\frac{1}{n}\int_{\sX}(\psi_n+c)\;\emph{d}\mathbf{P}\;.
		\end{equation*}
		Then the following hold:
		\begin{enumerate}
			\item 
			\eq{L=\lim_{n\to\infty} \frac{1}{n}\int_{\sX}\psi_n\;\emph{d}\mathbf{P}\;;}
			\item 	for $\mathbf{P}$-almost every $\sx\in \sX$,
			\begin{equation*}
				\lim\limits_{n\to\infty}\frac{1}{n}\psi_n(\sx)=L\;;
			\end{equation*}
			\item the upper bound 
			\begin{equation}
				\label{eq:limsup}
				\limsup_{n\to\infty}\frac{1}{n}\psi_n(\sx)\leq  L
			\end{equation}
			holds uniformly in $\sx\in\sX$.
		\end{enumerate}
	\end{prop}

	The first two assertions do not necessitate the continuity requirement, and are simply a consequence of Kingman's subadditive ergodic theorem~\cite[Theorem 1]{Kingman}	applied to the sequence $(\psi_n-c)_{n\geq 1}$. The last assertion is~\cite[Corollary 4.8]{Shmerkin-Annals}, again applied to the same sequence (it is a consequence of Lemma \ref{lem:K-W} stated below).

	Proposition~\ref{prop:Furmanadapt} shall be equally relevant in the course of the proof of lower bounds for $L^{q}$-dimensions (see the proof Proposition~\ref{prop:Tq}); for the moment, it yields our sought after upper bound, which is the main result of the present subsection.
	
	\begin{prop}
		\label{prop:upperbound}
		For any $q>1$,
		\begin{equation*}
			\limsup\limits_{n\to\infty}-\frac{\log{\sum_{Q\in \cD_{n}}\mu_{\sx}(Q)^{q}}}{(q-1)n}\leq \min\biggl\{\frac{\int_\sX\log{\norm{\Delta(\sy)}_q^{q}}\;\emph{d}\mathbf{P}(\sy)}{(q-1)\log{\lambda}} ,d\biggr\}
		\end{equation*}
		uniformly in $\sx\in \sX$.
	\end{prop}
	
	\begin{proof}
		The proof is nearly identical to the case $d=1$, presented in~\cite[\S 5.3]{Shmerkin-Annals}; we include it for the sake of completeness.
		Fix $q>1$ and $\sx\in \sX$. First, applying Lemma~\ref{lem:equallimits} to the sequences $(2^{-n})_{n\geq 1}$ and $(2^{-m(n)})_{n\geq 1}$ gives
		\begin{equation*}
			\limsup\limits_{n\to\infty}-\frac{\log{\sum_{Q\in \cD_{n}}\mu_{\sx}(Q)^{q}}}{(q-1)n}=	\limsup\limits_{n\to\infty}-\frac{\log{\sum_{Q\in \cD_{m(n)}}\mu_{\sx}(Q)^{q}}}{(q-1)m(n)}
		\end{equation*}
		Secondly, in view of Lemma~\ref{lem:normdiscretization}, we have
		\begin{equation*}
			\limsup\limits_{n\to\infty}-\frac{\log{\sum_{Q\in \cD_{m(n)}}\mu_{\sx}(Q)^{q}}}{(q-1)m(n)}=	\limsup\limits_{n\to\infty}-\frac{\log{\norm{\mu_{\sx,n}^{m(n)}}_q^{q}}}{(q-1)m(n)}\;.
		\end{equation*}
		We have thus reduced matters to estimating from below the quantity $\norm{\mu_{\sx,n}^{(m(n))}}_q^q$.
		
		To begin with, recall that the support of $\mu_{\sx,n}$ is contained in $[0,1)^d$, whence its discretization $\mu_{\sx,n}^{(m(n))}$ has at most $2^{m(n)d}$ atoms. By \eqref{eq:powerofsum},
		\begin{equation*}
			\norm{\mu_{\sx,n}^{(m(n))}}_q^{q}\geq 2^{- m(n)(q-1)d}
		\end{equation*}
		for every $n\geq 1$, so that
		\begin{equation*}
			\limsup\limits_{n\to\infty}-\frac{\log{\norm{\mu_{\sx,n}^{m(n)}}_q^{q}}}{(q-1)m(n)}\leq d\;.
		\end{equation*}
		
		Furthermore, Lemma~\ref{lem:qnorm} allows us to estimate, for every $n\geq 1$,
		\begin{equation*}
			\norm{\mu_{\sx,n}^{(m(n))}}_q^{q}\geq \norm{\mu_{\sx,n}}_q^{q}=\norm{\ast_{i=0}^{n-1}\;S_{\lambda^{i}}\Delta(\bT^i\sx)}_q^{q}\geq \prod_{i=0}^{n-1}\norm{S_{\lambda^i}\Delta(\bT^i\sx)}_q^q=\prod_{i=0}^{n-1}\norm{\Delta(\bT^i\sx)}_q^q\;.
		\end{equation*}
		It follows that
		\begin{equation}
			\label{eq:orbitalsum}
			\limsup_{n\to\infty}-\frac{\log{\norm{\mu_{\sx,n}^{(m(n))}}}_q^q}{m(n)}\leq  \limsup_{n\to\infty}\frac{1}{m(n)}\sum_{i=0}^{n-1}-\log{\norm{\Delta(\bT^i\sx)}_q^{q}}\;.
		\end{equation}
		Since $\cX$ is a pleasant model, there is $M\in \N^*$ such that the support of each measure $\Delta(\sy)$, $\sy\in \sX\;$ has at most $M$ elements. Lemma~\ref{lem:qnorm} then gives $M^{-(q-1)}\leq \norm{\Delta(\sy)}_q^{q}\leq 1$, whence $0\leq -\log{\norm{\Delta(\sy)}_q^q}\leq (q-1)\log{M}$. Therefore, the sequence of orbital sums
		\begin{equation*}
			\sy\mapsto \sum_{i=0}^{n-1}-\log{\norm{\Delta(\bT^i\sy)}_q^q}\;, \quad n\geq 1
		\end{equation*}
		satisfies the hypothesis of Proposition~\ref{prop:Furmanadapt} (equality holds in~\eqref{eq:subadditivity} with $c=0$). The latter thus delivers
		\begin{equation*}
			\limsupl_{n\to\infty}\frac{1}{n}\sum_{i=0}^{n-1}-\log{\norm{\Delta(\bT^i\sx)}_q^q}\leq\int_{\sX}-\log{\norm{\Delta(\sy)}_q^q}\;\text{d}\bP(\sy)
		\end{equation*}
		for every $\sx\in \sX$. Combining this with~\eqref{eq:orbitalsum}, we conclude that
		\begin{equation*}
			\limsup_{n\to\infty}-\frac{\log{\norm{\mu_{\sx,n}^{(m(n))}}}_q^q}{(q-1)m(n)}\leq\frac{\int_\sX\log{\norm{\Delta(\sy)}_q^{q}}\;\text{d}\mathbf{P}(\sy)}{(q-1)\log{\lambda}}
		\end{equation*}
		for every  $\sx\in \sX$. An inspection of the argument shows that all the inequalities are in fact uniform in $\sx\in\sX$. The proof is concluded.
		
	\end{proof}

	\subsection{A sub-multiplicative cocycle}
	
	The next lemma is a simple but key application of dynamical self-similarity. It is implicit in the proof of \cite[Proposition 4.13]{Shmerkin-Annals}, but we include the proof since it will be a key step in the proof of Theorem \ref{thm:lqdimension}.
	\begin{lem} \label{lem:refined-submultiplicative}
		Fix $q>1$. For any $s,m\in\N$, $Q\in\cD_s$, $\mathcal{D}\subset\mathcal{D}_{s+m}(Q)$ and $\sx\in \sX$ the following holds: denoting $n=\lceil (s+2)/\log(\lambda^{-1})\rceil$, there exist points $z_j\in Q-[0,\lambda^n)^d$ and real numbers $p_j>0$ for $1\le j\le J$ such that  $\sum_{j=1}^J p_j \le \mu_{\sx}(2\cdot Q)$ and
		\[
		\sum_{Q'\in\cD}  \mu_{\sx}(Q')^q  \le  \mu_{\sx}(2\cdot Q)^{q-1}\sum_{j=1}^{J} p_j \sum_{Q'\in\cD} \mu_{\bT^n \sx}(\lambda^{-n}(Q'-z_j))^q\;.
		\]
		In particular,
		\[
		\sum_{Q'\in\cD}  \mu_{\sx}(Q')^q  \le  \mu_{\sx}(2\cdot Q)^{q} \cdot \max_{1\le j\le J} \sum_{Q'\in\cD} \mu_{\bT^n \sx}(\lambda^{-n}(Q'-z_j))^q\;.
		\]
	\end{lem}
	
	\begin{proof}
		Write $\widetilde{Q} = Q - [0,\lambda^n)^d$, and let
		\[
		\mu_{\sx,n}|_{\widetilde{Q}} = \sum_{j=1}^J p_j \delta_{z_j},\quad p_j>0.
		\]
		Note that the points $z_j$ are the atoms of $\mu_{\sx,n}$ such that $(z_j+[0,\lambda^n)^d)\cap Q\neq \varnothing$. Since $\delta_z * S_{\lambda^n}\mu_{\bT^n \sx}$ is supported on $z+[0,\lambda^n)^d$, it follows from the self-similarity relation $\mu_{\sx} = \mu_{\sx,n}* S_{\lambda^n}\mu_{\mathbf{T}^n \sx}$  that
		\begin{equation} \label{eq:restricted-self-similarity}
			\mu_{\sx}|_{Q} =(\mu_{\sx,n}|_{\widetilde{Q}}*S_{\lambda^n}\mu_{\mathbf{T}^n \sx})|_{Q}.
		\end{equation}
		Observe now that
		\begin{equation} \label{eq:mass-double-cube}
			p \coloneqq  \sum_j p_j \le \mu_{\sx}(2\cdot Q),
		\end{equation}
		using the fact that the $\supp\mu_{\sx,n}\subset \supp\mu_{\sx}-[0,\lambda^n)^d$, and that $4 \lambda^n \le 2^{-s}$. We can then conclude that
		\begin{align*}
			\sum_{Q'\in\cD}  \mu_{\sx}(Q')^q & \overset{\eqref{eq:restricted-self-similarity}}{=} \sum_{Q'\in\cD} \left( \sum_j p_j \delta_{z_j} * S_{\lambda^n}\mu_{\mathbf{T}^n \sx}(Q')  \right)^q \\
			& = \sum_{Q'\in\cD} \left( \sum_j p_j \,\mu_{\mathbf{T}^n\sx}(\lambda^{-n}(Q'-z_j)) \right)^q                                                            \\
			& \le \sum_{Q'\in\cD} p^{q-1} \sum_j p_j \, \mu_{\mathbf{T}^n \sx}(\lambda^{-n}(Q'-z_j))^q                                                               \\
			& \overset{\eqref{eq:mass-double-cube}}{\le}   \mu_{\sx}(2\cdot Q)^{q-1}\sum_j p_j \sum_{Q'\in\cD} \mu_{\mathbf{T}^n \sx}(\lambda^{-n}(Q'-z_j))^q,
		\end{align*}
		where we used convexity of $t\mapsto t^q$ in the third line. The last claim is now immediate, since $\sum_{j=1}^J p_j \le \mu_{\sx}(2\cdot Q)$.
	\end{proof}
	
	As a first application, we deduce that, for every $q>1$, the sequence of measurable functions $\phi_n^{(q)}\colon \sX\to \R_{>0}$ given by the assignment
	\begin{equation} \label{eq:def-phi}
		\phi_n^{(q)}(\sx)=\norm{\mu_\sx^{(m(n))}}_q^{q}\;, \quad \sx\in \sX, \;n\geq 1
	\end{equation}
	is sub-multiplicative.
	
	\begin{prop}
		\label{prop:subadditive}
		Let $q>1$, and define the sequence $(\phi_n^{(q)})_{n\geq 1}$ as above. There is $C_{\lambda,q}>1$, depending smoothly on $(\lambda,q)$, such that, for every $n,n'\geq 1$ and $\sx\in \sX$,
		\begin{equation*}
			\phi_{n+n'}^{(q)}(\sx)\leq C_{\lambda,q}\; \phi_{n'}^{(q)}(\bT^n\sx)\;\phi_n^{(q)}(\sx)\;.
		\end{equation*}
	\end{prop}
	
	\begin{proof}
		Fix $q>1$, $\sx\in \sX$ and $n,n'\geq 1$. Let $n_0$ satisfy
		\[
		\lambda^{-n_0} \ge 4\cdot 2^{m(n)} > \lambda^{1-n_0}.
		\]
		For each $Q\in\cD_{m(n)}$, we apply Lemma \ref{lem:refined-submultiplicative} to $\cD=\cD_{m(n+n')}(Q)$ to get that
		\begin{equation} \label{eq:subadditive:1}
			\sum_{Q'\in\cD_{m(n+n')}(Q)}  \mu_{\sx}(Q')^q  \le  \mu_{\sx}(2\cdot Q)^{q}\max_j \sum_{Q'\in\cD_{m(n+n')}(Q)} \mu_{\bT^n \sx}(\lambda^{-n_0}(Q'-z_j))^q
		\end{equation}
		for some points $z_j\in Q_0-[0,\lambda^{n_0})^d$. A short calculation shows that
		\[
		\lambda^{-n_0} 2^{-m(n+n')} =\Theta_{\lambda}(1) 2^{-m(n')}\;,
		\]
		and therefore Lemma \ref{lem:measurefinitecover} yields that, for each $j$,
		\begin{equation} \label{eq:subadditive:2}
			\sum_{Q'\in\cD_{m(n+n')}(Q)} \mu_{\bT^n \sx}(\lambda^{-n_0}(Q'-z_j))^q \le C_{\lambda,q} \, \|\mu_{\bT^n(\sx)}^{m(n')}\|_q^q\;.
		\end{equation}
		We emphasize that the constant $C_{\lambda,q}$ is independent of $Q$ and $j$, and can be taken to be smooth on $\lambda,q$. On the other hand, another application of Lemma \ref{lem:measurefinitecover} yields that
		\begin{equation} \label{eq:subadditive:3}
			\sum_{Q\in\cD_{m(n)}}  \mu_{\sx}(2\cdot Q)^{q} \le C_{\lambda,q} \|\mu_{\sx}^{m(n)}\|_q^q.
		\end{equation}
		Combining \eqref{eq:subadditive:1},~\eqref{eq:subadditive:2} and~\eqref{eq:subadditive:3} completes the proof.
	\end{proof}
	
	It is now natural to invoke Kingman's subadditive ergodic theorem to show almost sure existence of the limit in the definition~\eqref{eq:Lqspectrum} of the $L^{q}$-spectrum. To be precise, since we are interested in results holding for every $\sx\in \sX$, we shall resort to the refined version of the sub-additive ergodic theorem for almost surely continuous cocycles framed in Proposition~\ref{prop:Furmanadapt}.
	
	The result reads as follows.

	\begin{prop}
		\label{prop:Tq}
		The following assertions hold for every $q>1$.
		\begin{enumerate}
			\item The limit in the definition of the $L^q$-spectrum $T_{\cX}(q)$ exists; moreover, setting 
			\[
			T_{\cX}^{(n)}(q) =- \frac{1}{m(n)}\int \log{\norm{\mu_\sx^{(m(n))}}_{q}^{q}}\;\d\bP(\sx)
			\]
			for all $n\geq 1$, there is $C'_{\lambda,q}>0$, depending continuously on $(\lambda,q)$, such that
			\begin{equation} \label{eq:Lq-spectrum-sup-integrals}
				T_{\cX}(q) = \sup_{n\geq 1} \left(T_{\cX}^{(n)}(q) -\frac{C'_{\lambda,q}}{m(n)}\right) = \lim_{n\to\infty} T_{\cX}^{(n)}(q)\;.
			\end{equation}
			\item For $\bP$-almost every $\sx\in \sX$,
			\begin{equation}
				\label{eq:convergencetoTq}
				\lim\limits_{n\to\infty}-\frac{\log{\norm{\mu_\sx^{(m(n))}}_{q}^{q}}}{m(n)}=T_{\cX}(q)\;.
			\end{equation}
			\item The lower bound 
			\begin{equation}
				\label{eq:liminfuniform}
				\liminfl_{n\to\infty} -\frac{\log{\norm{\mu_\sx^{(m(n))}}_q^q}}{m(n)}\geq T_{\cX}(q) 
			\end{equation}
			holds uniformly in $\sx\in\sX$.
		\end{enumerate}
		Furthermore, for $\bP$-almost every $\sx\in \sX$,  it holds that $\tau_{\mu_\sx}(q)=T_{\cX}(q)$ for all $q>1$.
	\end{prop}
	
	\begin{proof}
		Fix $q>1$. By virtue of Proposition~\ref{prop:subadditive}, the sequence of measurable functions $\psi_n^{(q)}\colon \sX \to \R_{\geq 0}$ defined by
		\begin{equation*}
			\psi_n^{(q)}(\sx) =\log{\phi_n^{(q)}(\sx)}=\log{\norm{\mu_\sx^{(m(n))}}_q^{q}}\;, \quad \sx\in \sX,\; n\geq 1
		\end{equation*}
		fulfills the weak subadditivity condition in~\eqref{eq:subadditivity} with $c=\log{C_{\lambda,q}}$. Also, since the model is pleasant, $\psi_n^{(q)}$ is continuous $\bP$-almost everywhere, for every $n\geq 1$.
		Finally, all $\psi_n^{(q)}$ are bounded and hence integrable with respect to $\bP$: for every $\sx\in \sX$, the inclusion $\supp{\mu_\sx}\subset [0,1)^{d}$ implies that the support of $\mu_{\sx}^{(m(n))}$ consists of at most $2^{m(n)d}$ elements, so that Lemma \ref{lem:qnorm} yields
		\begin{equation*}
			-d(q-1)m(n)	\leq \log{\norm{\mu_{\sx}^{(m(n))}}_q^q}\leq 0\;.
		\end{equation*}

		Since $n/m(n)\to\frac{1}{\log{1/\lambda}}$ as $n\to\infty$, the first three assertions follow directly from Proposition~\ref{prop:Furmanadapt} applied to the sequence $(\psi^{(q)}_n)_{n\geq 1}$.
		
		As to the last assertion, it follows from Lemma~\ref{lem:equallimits} applied to the sequences $(2^{-n})_{n\geq 1}$ and $(2^{-m(n)})_{n\geq 1}$ that
		\begin{equation*}
			\tau_{\mu_{\sx}}(q)=\liminfl_{n\to\infty}-\frac{\log{\norm{\mu_{\sx}^{(n)}}_q^q}}{n}=\liminf_{n\to\infty}-\frac{\log{\norm{\mu_{\sx}^{(m(n))}}_q^q}}{m(n)}\;,
		\end{equation*}
		the latter quantity being equal to
		\begin{equation*}
			\lim_{n\to\infty}-\frac{\log{\norm{\mu_{\sx}^{(m(n))}}_q^q}}{m(n)}=T_{\cX}(q)
		\end{equation*}
		whenever $\sx\in \sX$ satisfies~\eqref{eq:convergencetoTq}. Hence, for every fixed $q>1$, equality $\tau_{\mu_{\sx}}(q)=T_{\cX}(q)$ holds for $\bP$-almost every $\sx\in \sX$. Standard measure theory allows to upgrade the previous statement to the following: there is $\mathsf{N}\subset \sX$ with $\bP(\mathsf{N})=0$ such that, for every $\sx\notin \mathsf{N}$, equality $\tau_{\mu_{\sx}}(q)=T_{\cX}(q)$ holds on a countable dense subset of $\R_{>1}$. Now convexity of every function of the form $\R\ni q\mapsto a^{q}\in \R,\;a>0\;$ readily yields, together with~\eqref{eq:convergencetoTq}, that $T_{\cX}$ is concave, and thus continuous on $\R_{>1}$. We deduce that $\tau_{\mu_{\sx}}(q)=T_{\cX}(q)$ holds for every $q>1$ and $\sx\notin \mathsf{N}$, as claimed.
	\end{proof}

	\begin{cor} \label{cor:subadditive-single-cube}
		For all $\delta>0$ there is an integer $m_0=m_0(\cX,\delta)$ such that the following holds: for all integers $s\geq 0$ and $m\geq m_0$, and for all $\sx\in\sX$ and $Q\in\cD_s$,
		\[
		\sum_{Q'\in\cD_{s+m}(Q)} \mu_{\sx}(Q')^q \le C_{\lambda,q}\; \mu_{\sx}(2\cdot Q)^q \;  2^{-(T_{\chi}(q)-\delta) m}\;.
		\]
	\end{cor}
	\begin{proof}
		Applying Lemma \ref{lem:refined-submultiplicative} as in the proof of Proposition \ref{prop:subadditive}, we get
		\[
		\sum_{Q'\in\cD_{s+m}(Q)} \mu_{\sx}(Q')^q \le C_{\lambda,q}\, \mu_{\sx}(2\cdot Q)^q \, \, \|\mu_{\bT^n\sx}^{(m)}\|_q^q
		\]
		where $n$ is such that $\lambda^{-n} = \Theta_{\lambda}(2^{m})$. Combining this with \eqref{eq:liminfuniform} yields the claim.
	\end{proof}
	
	Combining Proposition~\ref{prop:upperbound} and Proposition~\ref{prop:Tq} (for the latter, to be precise, we take the inequality in~\eqref{eq:liminfuniform}  together with Lemma~\ref{lem:equallimits}), we draw the conclusion that, for every $q>1$ and every $\sx\in \sX$,
	\begin{equation*}
		T_{\chi}(q) \leq (q-1)\dim_{\mu_{\sx }}(q)\leq \frac{\int_X\log{\norm{\Delta(y)}_q^{q}}\;\text{d}\bP(y)}{\log{\lambda}} \;.
	\end{equation*}
	In fact, this holds uniformly, in the sense that both inequalities hold uniformly in the limits implicit in the definition of $T_{\cX}(q)$ and $\dim_{\mu_{\sx}}(q)$. Therefore, in order to complete the proof of Theorem~\ref{thm:lqdimension}, it suffices to show that
	\begin{equation} \label{eq:Tq-lower-bound}
		T_{\chi}(q)\geq  \frac{\int_X\log{\norm{\Delta(y)}_q^{q}}\;\text{d}\bP(y)}{\log{\lambda}} \;.
	\end{equation}
	This is the crux of the whole argument, in particular it is the place where the assumptions of exponential separation and unsaturation on lines of the model are instrumental to our approach.

	\section{\texorpdfstring{An $L^q$-}{a}smoothening theorem and the proof of Theorem \ref{thm:lqdimension}}
	
	\label{sec:smoothening}

	This section contains the proof of the lower bound for $T_{\cX}(q)$ in~\eqref{eq:Tq-lower-bound}, and thereby achieves the proof of Theorem \ref{thm:lqdimension}. The main stepping stone is a smoothening theorem for $L^q$-dimensions, which is the content of Theorem \ref{thm:lqflattening}. In particular, it is in the proof of this theorem that the argument diverges sharply from the one-dimensional situation treated in \cite{Shmerkin-Annals}.
	
	Throughout this section, unless otherwise stated, we work with a pleasant model $\cX=(\sX,\bT,\bP,\Delta,\lambda)$ in $\R^d$ which is $q$-unsaturated on lines for some fixed real $q>1$, generating a family $(\mu_{\sx})_{\sx\in \sX}$ of dynamically driven self-similar measures. We emphasize that we do \emph{not} assume exponential separation at this stage.
	
	\subsection{On the unsaturation condition}
	
	We begin by recording an estimate for subadditive cocycles going back at least to \cite{Katznelson-Weiss}. In the statement, a \emph{measure-preserving system} is a triple $(\sX,\bT,\bP)$ where $\sX$ is a measurable space, $\bT\colon \sX\to \sX$ is a measurable map and $\bP$ is a $\bT$-invariant probability measure on $\sX$. If $f\colon \sX\to \R$ is a measurable function, we adopt the notation
	\eq{
		\mathbf{A}(f,\sx,n) = \frac{1}{n}\sum_{j=0}^{n-1} f(\bT^j\sx)\;, \quad \sx\in \sX,\;n\geq 1
	}
	for its ergodic averages, and $\norm{f}_{\infty}$ for its essential supremum.
	
	\begin{lem} \label{lem:K-W}
		Let $(\sX,\bT,\bP)$ be a measure-preserving system, $(\psi_n)_{n\geq 1}$ a sequence of bounded measurable functions $\sX\to \R$ satisfying 
		\eq{\psi_{n+n'}(\sx)\le \psi_n(\sx)+\psi_{n'}(\bT^n\sx) \quad \text{for all $n,n'\ge 1$ and $\sx\in \sX$.}}
		Then, for every $n\ge N\ge 1$ and $\sx\in \sX$,
		\[
		\frac{1}{n}\psi_n(\sx) \le \frac{1}{N}\mathbf{A}(\psi_N,\sx,n) + \frac{2}{n}\max_{1\le i\le N}\|\psi_i\|_\infty\,.
		\]
	\end{lem}
	\begin{proof}
		See \cite[p.~294]{Katznelson-Weiss}.
	\end{proof}
	
	In the sequel, we adopt the convention that $\pi\in\G(d,d-1)$ is a projection onto $\R^{d-1}$ with a choice of dyadic grid in $\R^{d-1}$ which is smooth in $\pi$.

	The next proposition does not hinge on any $q$-unsaturation assumption on the model.
	\begin{prop}
		\label{prop:semicontinuity}
		Fix $q>1$ and $k\in\{1,\ldots,d-1\}$. The function 
		\eq{\G(d,k)\to \R_{\geq 0}\;, \quad \pi\mapsto T_{\pi\cX}(q)}
		is lower semicontinuous. Moreover, for every $\pi\in \G(d,k)$ and $\eps>0$ there is a neighborhood $U$ of $\pi$ in $\G(d,k)$ such that
		\[
		\liminf_{m\to\infty} -\frac{1}{m} \log \| (\pi'\mu_{\sx})^{(m)} \|_q^q \ge T_{\pi\cX}(q) -\eps
		\]
		uniformly in $\sx\in\sX$ and $\pi'\in U$.
	\end{prop}
	
	\begin{proof}
		We begin by introducing more regular versions of the functions $(\sx,\pi)\mapsto \log \|(\pi \mu_{\sx})^{(m)}\|_q^q$. For every $n\geq 1$, let $\psi_n\colon \R^{d-1}\to [0,1]$ be a smooth bump function vanishing outside the cube $[-2^{m(n)},2^{m(n)}]^{d-1}$ and taking the constant value $1$ on $[-2^{m(n)+1},2^{m(n)+1}]^{d-1}$. We define
		\[
		\Psi_{n}^{\pi}(\sx) = \log \biggl(\sum_{k\in [2^{m(n)}]^{d-1}} \left(\int \psi_{n}(t+2^{-m(n)}k)\,\text{d}\pi\mu_{\sx}(t) \right)^q\biggr).
		\]
		An elementary calculation using Lemma \ref{lem:measurefinitecover} shows that
		\begin{equation} \label{eq:continuous-approximation}
			\bigl|\log\|(\pi\mu_{\sx})^{(m(n))}\|_q^q - \Psi_{n}^{\pi}(\sx) \bigr| = O_{\lambda,q}(1)
		\end{equation}
		where the quantity $O_{\lambda,q}(1)$ is, crucially for the sequel, independent of $\pi$. By Propositions \ref{prop:subadditive} and \ref{prop:Tq}, applied to the projected models, this implies that, for any $\sx\in \sX$, $\pi\in \G(d,k)$ and $n,n'\geq 1$,
		\begin{equation} \label{eq:Psi-subadditivity}
			\Psi_{n+n'}^{\pi}(\sx) \le \Psi_{n}^{\pi}(\sx) + \Psi_{n'}^{\pi}(\bT^{n}\sx) + O_{\lambda,q}(1)\,,
		\end{equation}
		and
		\begin{equation} \label{eq:Psi-limit}
			T_{\pi\cX}(q) = \sup_n \frac{1}{n\log\lambda} \int\left( \Psi_{n}^{\pi}(\sx) + C_{q,\lambda}\right)\,d\bP(\sx)= \lim_{n\to\infty} \frac{1}{n\log\lambda} \int \Psi_{n}^{\pi}(\sx)\,d\bP(\sx)
		\end{equation}
		
		Since the model is pleasant, the function $(\sx,\pi)\to\Psi_{n}^{\pi}(\sx)$ is bounded, almost everywhere continuous in $\sx$, and uniformly continuous in $\pi$. By a well known compactness argument (see e.g. the proof of \cite[Lemma 4.7]{Shmerkin-Annals}),  there exists a continuous function $\widetilde{\Psi}_n:\sX\times \G(d,d-1)\to \R$ such that, writing $\widetilde{\Psi}^\pi_n(\sx)=\widetilde{\Psi}_n(\sx,\pi)$, we have
		\begin{enumerate}[(\alph*)]
			\item \label{it:Psi:a} $\Psi_n^{\pi}(\sx) \le \widetilde{\Psi}^{\pi}_n(\sx)$ for all $\sx\in\sX$ and $\pi\in\G(d,d-1)$,
			\item \label{it:Psi:b} $\int \widetilde{\Psi}^{\pi}_n(\sx)\; \d\bP(\sx) \le \int \Psi_n^{\pi}(\sx)\; \d\bP(\sx) + 1$ for all $\pi\in\G(d,d-1)$.
		\end{enumerate}
		Note that \eqref{eq:Psi-limit} continues to hold with $\widetilde{\Psi}_n^{\pi}$ in place of $\Psi_n^{\pi}$ (and a different constant $C_{\lambda,q}$). This shows that $T_{\pi\cX}(q)$ is a supremum of continuous functions of $\pi$, and hence is lower semi-continuous in $\pi$.

		Now fix $\pi\in \G(d,d-1)$ and $\eps>0$. Using \eqref{eq:Psi-limit} and property \ref{it:Psi:b} of $\widetilde{\Psi}_{n}^{\pi}$, we find that there exists $N=N(\pi,q,\eps)$ such that
		\[
		\int \frac{1}{N\log\lambda}\widetilde{\Psi}_{N}^{\pi}(\sx)\; \d\bP(\sx) \ge T_{\pi\cX}(q) - \frac{\eps}{4}\;.
		\]
		Since $\widetilde{\Psi}_N$ is continuous and  $\sX$ is compact, there is a neighborhood $U$ of $\pi$ such that
		\begin{equation} \label{eq:unif-continuity-U}
			\frac{1}{N|\log\lambda|}\left|\Psi_{N}^{\pi'}(\sx) - \Psi_{N}^{\pi}(\sx)\right| \le \frac{\eps}{4}, \quad \pi'\in U,\; \sx\in\sX\;.
		\end{equation}
		In particular,
		\begin{equation} \label{eq:continuity-in-pi}
			\int \frac{1}{N\log\lambda}\widetilde{\Psi}_{N}^{\pi'}(\sx)\; \d\bP(\sx) \ge T_{\pi\cX}(q)  -\frac{\eps}{2}\;, \quad \pi'\in U\;.
		\end{equation}

		By Lemma \ref{lem:K-W}, the subadditivity \eqref{eq:Psi-subadditivity} implies that, for $n\ge N$, and $\sx\in\sX$,
		\begin{equation} \label{eq:application-KW}
			\frac{1}{n}\Psi_n^{\pi}(\sx) \le \frac{1}{N}\mathbf{A}(\Psi_N^{\pi},\sx,n) + \frac{2}{n}\max_{1\le i\le N}\|\Psi_i^{\pi}\|_\infty + \frac{1}{n}O_{\lambda,q}(1)\,.
		\end{equation}
		Write
		\[
		C = 2\max_{1\le j\le N} \|\Psi_j^{\pi}\|_\infty + O_{\lambda,q}(1).
		\]
		It follows from \eqref{eq:application-KW} and property \ref{it:Psi:a} that, for all $n\ge N$,
		\begin{equation} \label{eq:pointwise-bound-by-ergodic-average}
			\frac{1}{n}\Psi_n^{\pi'}(\sx) \le  \frac{1}{N} \mathbf{A}\left(\widetilde{\Psi}_{N}^{\pi'},\sx,n\right) + \frac{C}{n}, \quad \pi'\in U\;,\sx\in\sX\;.
		\end{equation}
		It follows from \eqref{eq:unif-continuity-U} that
		\begin{equation} \label{eq:unif-continuity-ergodic-average}
			\frac{1}{N|\log\lambda|}\left| \mathbf{A}\left(\widetilde{\Psi}_{N}^{\pi'},\sx,n\right)-  \mathbf{A}\left(\widetilde{\Psi}_{N}^{\pi},\sx,n\right)\right| \le \frac{\eps}{4}, \quad \pi'\in U,\; \sx\in\sX\;.
		\end{equation}
		By the unique ergodicity of $(\sX,\bT,\bP)$,
		\begin{equation} \label{eq:uniform-limit-unique-ergodicity}
			\lim_{n\to\infty} \mathbf{A}\left(\widetilde{\Psi}_{N}^{\pi},\sx,n\right) =
			\int \widetilde{\Psi}_{N}^{\pi}(\sx)\; \d\bP(\sx) \quad \text{uniformly in } \sx\in\sX\;.
		\end{equation}
		Combining \eqref{eq:pointwise-bound-by-ergodic-average},
		\eqref{eq:unif-continuity-ergodic-average}, and  \eqref{eq:uniform-limit-unique-ergodicity}, we deduce that if $n$ is large enough, then for  all $\pi'\in U$ and all $\sx\in\sX$,
		\[
		\frac{1}{n\log\lambda}\Psi_n^{\pi'}(\sx) \ge \int \frac{1}{N\log\lambda}\widetilde{\Psi}_{N}^{\pi'}(\sx)\; \d\bP(\sx) - \frac{\eps}{2}\;.
		\]
		Recalling \eqref{eq:continuity-in-pi} and property \ref{it:Psi:b} above, we conclude that
		\[
		\liminf_{n\to\infty}\frac{1}{n\log\lambda}\Psi_n^{\pi'}(\sx) \ge T_{\pi\cX}(q)  -\eps\;, \quad \pi'\in U\;.
		\]
		In light of \eqref{eq:continuous-approximation}, this finishes the proof.
	\end{proof}
	
	As an immediate corollary of the above proposition and of compactness of $\G(d,d-1)$, we obtain the following uniform pointwise improvement of $q$-unsaturation on lines.
	\begin{cor}
		\label{cor:unsaturation-uniform}
		Assume that $\cX$ is $q$-unsaturated on lines for some $q>1$. Then there is $\eta=\eta(\cX,q)>0$ such that
		\[
		\liminf_{m\to\infty} \inf_{\pi\in\G(d,d-1)} -\frac{1}{m} \log \| (\pi\mu_{\sx})^{(m)} \|_q^q \ge T_{\cX}(q) - (q-1) + \eta,
		\]
		uniformly in $\sx\in\sX$.
	\end{cor}

	We deduce the following ``box-counting'' consequence of unsaturation on lines. Recall that $D_{\cX}(q)=T_{\cX}(q)/(q-1)$ denotes the $L^q$-dimension of the model $\cX$.
	\begin{prop} \label{prop:discrete-unsaturated}
		There exist $\eta>0$ and an integer, $m_0\geq 1$, both depending on $\cX$ and $q$, such that the following holds for all $m\ge m_0$.
		Let $\cD\subset\cD_{m}$ be a family of cubes such that, for some $\sx\in \sX$,
		\begin{enumerate}[(\alph*)]
			\item $2^{-j}<\mu_{\sx}(Q)\leq 2^{1-j}$ for all $Q\in\cD$ and some integer $j\geq 1$ and
			\item $\sum_{Q\in\cD}\mu_{\sx}(Q)^{q} \ge 2^{-(D_{\cX}(q)+\eta)(q-1) m}$.
		\end{enumerate}
		Then, for all  projections $\pi\in\G(d,d-1)$,
		\[
		\cN_m(\pi(\cup\cD)) \ge 2^{(\eta-1) m}|\cD|\;.
		\]
	\end{prop}

	Notice that, in the case $d=1$, the left-hand side of the last displayed inequality equals $1$.
	\begin{proof}.
		Set
		\begin{equation*}
			\psi_n^{\pi}(\sx)  = \log\|(\pi\mu_{\sx})^{m(n)}\|_q^q,\quad \pi\in\G(d,d-1),\; n\geq 1, \;\sx\in \sX.
		\end{equation*}
		Corollary~\ref{cor:unsaturation-uniform} provides $\eta=\eta(\cX,q)>0$ and an integer $n_0=n_0(\cX,q)$ such that
		\[
		-\frac{1}{m(n)}\psi_{n}^{\pi}(\sx) > (D_{\cX}(q)-1+2\eta)(q-1)\]
		for all $\sx\in\sX,\; n\ge n_0$ and $ \pi\in\G(d,d-1)$.
		Fix $\sx\in\sX$ and $m=m(n)$ for some $n\ge n_0$. Using that $m(n)$ has bounded gaps, it is enough to establish the claim for such $m$. Then
		\begin{equation} \label{eq:use-unsaturation}
			\log\|(\pi\mu_{\sx})^{(m)}\|_q^q \le -(D_{\cX}(q)-1+2\eta)(q-1)m\,.
		\end{equation}
		Now let $\cD\subset\cD_{m}$ be as in the statement. Using both assumptions on $\cD$, we get
		\begin{equation} \label{eq:q-sum-bounds}
			2^{(1-j)q} |\cD| \ge \sum_{Q\in\cD} \mu_{\sx}(Q)^q \ge 2^{-(D_{\cX}(q)+\eta)(q-1) m} \,.
		\end{equation}
		Fix $\pi\in\G(d,d-1)$. For each $2^{-m}$-dyadic cube $R$ in $\pi$ (identified with its range), let $N_R$ be the number of dyadic cubes in $\cD$ whose $\pi$-projection intersects $R$. Note that there is a constant $C$ such that if $Q\in\cD$, $\pi(Q)\cap R\neq \emptyset$, then $\pi(Q)\subset C \cdot R$. Then, using Lemma \ref{lem:measurefinitecover} and the first assumption on $\cD$, we get
		\begin{equation}
			\begin{split} \label{eq:box-counting-fibre}
				\|(\pi\mu_{\sx})^{(m)}\|_q^q & \ge \Theta_q(1) \sum_R \pi\mu_{\sx}(C\cdot R)^q                    \\
				& \ge \sum_R \big(2^{-j} N_R\big)^q = 2^{-jq} \sum_R N_R^q\,.
			\end{split}
		\end{equation}
		Using H\"{o}lder's inequality, and noting that $\cN_{m}(\pi\cD)=|\{R:N_R>0\}|$, we see that
		\[
		|\mathcal{D}|^q \le  \bigl(\sum_R N_R\bigr)^q  \le  \bigl(\cN_{m}(\pi\cD)\bigr)^{q-1}  \sum_R N_R^q\,.
		\]
		We conclude that
		\begin{align*}
			\bigl(\cN_{m}(\pi\cD)\bigr)^{1-q} |\mathcal{D}|^q & \le    \sum_R N_R^q   \overset{\eqref{eq:box-counting-fibre}}{\le}  O_q(1) 2^{jq} \|(\pi\mu_{\sx})^{(m)}\|_q^q \\
			& \overset{\eqref{eq:use-unsaturation}}{\le} O_q(1) 2^{jq}\cdot 2^{-(D_{\cX}(q)-1+2\eta)(q-1)m}                  \\
			& \overset{\eqref{eq:q-sum-bounds}}{\le} O_q(1)\, |\cD| \, 2^{(1-\eta)(q-1)m}  \,.
		\end{align*}
		Rearranging the terms of the previous inequality appropriately gives the claim.
	\end{proof}

	\begin{cor} \label{cor:discrete-unsaturated}
		There exist $\eta>0$ and an integer $m_0\geq 1$, both depending on $\cX$ and $q$, such that the following holds for all $\sx\in\sX$ and  all $m\ge m_0$.
		Let $\cD\subset\cD_{m}$ be a family of cubes such that, for some projection $\pi\in\G(d,d-1)$,
		\begin{equation} \label{eq:saturation-assumption}
			\cN_{m}\bigl(\pi(\cup\cD)\bigr)  < 2^{(\eta-1) m}|\cD|\;.
		\end{equation}
		Then there exists a subfamily $\cD'\subset\cD$ with $|\cD\setminus \cD'|\le 2^{-\eta m}|\cD|$ and
		\[
		\sum_{Q\in\cD'} \mu_{\sx}(Q)^q \le 2^{-(D_{\cX}(q)+\eta)(q-1)m}\,.
		\]
	\end{cor}
	\begin{proof}
		Fix an integer $m\geq 1$, a collection $\cD\subset \cD_m$ and a projection $\pi\in \G(d,d-1)$ satisfying~\eqref{eq:saturation-assumption},  and let $\sx\in\sX$.  Let
		\[
		J = \lceil (q-1)(D_{\sX}(q)+1)+d\rceil\,,
		\]
		and note that
		\begin{equation} \label{eq:lq-from-tiny-mass}
			\sum \{ \mu(Q)^q: Q\in\cD, \mu(Q)< 2^{-Jm}\} \le 2^{dm}\,2^{-Jmq} \le  2^{-(D_{\sX}(q)+1)(q-1) m}\,.
		\end{equation}

		For each $j\in\N$, let $\cD_j = \{Q\in\cD: \mu_{\sx}(Q)\in (2^{-j},2^{1-j}]\}$. Let
		\[
		\cJ = \left\{ 1\le j\le J m: \sum_{Q\in\cD_j}\mu_{\sx}(Q)^{q} \ge \frac{1}{2Jm}\cdot 2^{-(D_{\cX}(q)+\eta)(q-1) m}\right\}.
		\]
		By Proposition \ref{prop:discrete-unsaturated}, if $\eta=\eta(\cX,q)$ is sufficiently small and $m_0=m_0(\cX,q)$ is sufficiently large, then
		\[
		\cN_{m}\bigl(\pi(\cup\cD_j)\bigr) >  2^{2\eta m}\cdot 2^{(\eta-1) m}|\cD_j|,\quad j\in \cJ,
		\]
		and so, by the assumption \eqref{eq:saturation-assumption}, $|\cD_j|\le 2^{-2\eta m}|\cD|$ for each $j\in \cJ$, and therefore
		\[
		\left|\bigcup_{j\in \cJ}\cD_j\right| \le Jm\cdot 2^{-2\eta m}|\cD| \le 2^{-\eta m}|\cD|,
		\]
		if $m$ is large enough. Setting $\cD' = \cD\setminus \cup_{j\in \cJ}\cD_j$, we conclude from \eqref{eq:lq-from-tiny-mass} and the definition of the set $\cJ$ that
		\[
		\sum_{Q\in\cD'} \mu_{\sX}(Q)^q \le 2^{-(D_{\sx}(q)+1)(q-1) m} + \frac{1}{2}\cdot 2^{-(D_{\sX}(q)+\eta)(q-1) m} \le 2^{-(D_{\sX}(q)+\eta)(q-1) m},
		\]
		as claimed.
	\end{proof}
	
	The next proposition can be seen as a local version of Corollary \ref{cor:discrete-unsaturated}.
	\begin{prop}
		\label{prop:gain-from-unsaturation}
		There exist $\eta=\eta(\cX,q)>0$ and $m_0=m_0(\cX,q)$ such that the following holds for all $s\in\N$, $m\in \N_{\ge m_0}$, $Q\in\cD_s$ and $\sx\in\sX$.
		
		Let $\cD\subset\cD_{s+m}(Q)$ be a family such that
		\begin{equation} \label{eq:local-saturation-assumption}
			\cN_{s+m}\bigl(\pi(\cup\cD)\bigr) \le 2^{(\eta-1)m}|\cD| \quad\text{for some }\pi\in\G(d,d-1).
		\end{equation}
		Then there is a collection $\cD'\subset \cD$ such that $|\cD\setminus\cD'|\le 2^{-\eta}m|\cD|$ and
		\begin{equation*}
			\sum_{Q'\in \cD'}\mu_{\sx}(Q')^{q}\leq 2^{-(D_{\sX}(q)+\eta)(q-1)m} \, \mu_{\sx}(2\cdot Q)^{q}\;.
		\end{equation*}
	\end{prop}
	\begin{proof}
		We apply Lemma \ref{lem:refined-submultiplicative} to obtain numbers $p_j>0$ with $\sum_j p_j \le \mu_{\sx}(2\cdot Q)$ and points $z_j\in Q$ (in both cases for $1\le j\le J$) such that
		\begin{equation} \label{eq:application-subadditive}
			\sum_{Q'\in \cD}\mu_{\sx}(Q')^{q} \leq  \mu_{\sx}(2\cdot Q)^{q-1}\sum_{j=1}^{J} p_j \sum_{Q'\in\cD} \mu_{\bT^n(\sx)}(\lambda^{-n}(Q'-z_j))^q,
		\end{equation}
		where $n$ is such that $2^{-s}=\Theta_{\lambda}(1)\lambda^{n}$. Fix $j$ for the time being. For each cube $Q'\in\cD$ and each $j$, let $R_j(Q')$ be the collection of cubes in $\cD_m$ that intersect $\lambda^{-n}(Q'-z_j)$; note that $|R_j(Q')|=O_{\lambda}(1)$. Write
		\[
		\widetilde{\cD}_j =\left\{ \bigcup R_j(Q'): Q'\in\cD \right\}.
		\]
		Note that $|\widetilde{\cD}_j|\le \Theta_{\lambda}(1)|\cD|$. Upon rescaling by $2^{s}=\Theta_1(\lambda^{-n})$, we deduce from the assumption \eqref{eq:local-saturation-assumption} that, for some $\pi\in\G(d,d-1)$ that we fix from now on,
		\[
		\mathcal{N}_m\bigl(\pi(\cup\widetilde{\cD}_j)\bigr) \le C_{\lambda} \, 2^{(\eta-1)m} \, |\widetilde{\cD}_j|.
		\]
		Now Corollary \ref{cor:discrete-unsaturated} ensures that if $\eta,m_0$ are taking respectively small and large enough in terms of $\cX,q$ only, then there are families $\widetilde{\cD}'_j\subset\widetilde{\cD}_j$ with $|\widetilde{\cD}_j\setminus\widetilde{\cD}'_j|\le 2^{-\eta m}|\widetilde{\cD}_j|$ and such that
		\begin{equation} \label{eq:tilde-family-gain}
			\sum_{Q''\in\widetilde{\cD}'_j} \mu_{\bT^n(\sx)}(Q'')^q \le 2^{-(D_{\sx}(q)+\eta)(q-1)m},\quad\sx\in\sX\,.
		\end{equation}
		Let
		\begin{equation} \label{eq:def-D'-j}
			\cD'_j = \left\{ Q'\in\cD: \bigcup R_j(Q')\subset\widetilde{\cD}'_j \right\}\,.
		\end{equation}
		Since the families $\{ R_j(Q'):Q\in\cD_{s+m}\}$  have $O_{\lambda}(1)$ overlapping, we have that
		\begin{equation} \label{eq:D'-j-large}
			|\cD\setminus\cD'_j|\le O_\lambda(1)\, 2^{-\eta m}\,|\cD| \le 2^{-\tfrac{1}{2}\eta m}|\cD| \,,
		\end{equation}
		for $m$ large enough. For each $j$, appealing to Lemma \ref{lem:measurefinitecover} once gain, we have
		\begin{equation} \label{eq:good-subfamily-j}
			\begin{split}
				\sum_{Q'\in\cD'_j} \mu_{\bT^n(\sx)}(\lambda^{-n}(Q'-z_j))^q &\leq C_{\lambda,q} \sum_{Q'\in\cD'_j}\sum_{Q''\in R_j(Q')} \mu_{\bT^n \sx}(Q'')^q \\
				&\overset{\eqref{eq:def-D'-j}}{\le} C_{\lambda,q} \sum_{Q''\in\widetilde{\cD}_j} \mu_{\bT^n \sx}(Q'')^q \overset{\eqref{eq:tilde-family-gain}}{\le} C_{\lambda,q} 2^{-(D_{\sX}(q)+\eta)(q-1)m},
			\end{split}
		\end{equation}
		for all $\sx\in\sX$.
		
		We also assume that $m_0$ is large enough in terms of $\eta, q$ and $\cX$ that, invoking \eqref{eq:liminfuniform} followed by yet another application of Lemma \ref{lem:measurefinitecover},
		\begin{equation} \label{eq:small-loss-orig-family}
			\sum_{Q'\in\cD} \mu_{\bT^n(\sx)}(\lambda^{-n}(Q'-z_j))^q \le 2^{-(D_{\cX}(q)-\tfrac{\eta}{8})(q-1)m}, \quad\sx\in\sX\,.
		\end{equation}
		
		Next, for each $Q'\in\cD$ let
		\begin{equation} \label{eq:def-p-Q'}
			p(Q') = \sum\{ p_j: Q'\in \cD\setminus\cD'_j\}.
		\end{equation}
		Writing $p=\sum_{j=1}^J p_j \le \mu_{\sx}(2\cdot Q)$, we have
		\[
		\sum_{Q'\in\cD} p(Q') = \sum_{j=1}^J p_j |\cD\setminus \cD'_j| \overset{\eqref{eq:D'-j-large}}{\le} p\, 2^{-\tfrac{\eta}{2} m}\, |\cD|,
		\]
		and so, by Markov's inequality,
		\[
		|\cD\setminus \cD'| \le 2^{-\tfrac{\eta}{4}m} |\cD|,\quad\text{where }\cD'=\left\{ Q'\in\cD: p(Q') \le 2^{-\tfrac{\eta}{4}m} \mu_{\sx}(2\cdot Q)  \right\}.
		\]
		We split the sum we want to control as
		\[
		\sum_{j=1}^J p_j  \sum_{Q'\in\cD'} \mu_{\bT^n(\sx)}(\lambda^{-n}(Q'-z_j))^q  \le \sum_{j=1}^J p_j  \left[\sum_{Q'\in\cD'_j} + \sum_{Q'\in \cD'\setminus\cD'_j}\right]   =: \sum_{j=1}^{J} p_j A_j + p_j B_j\,.
		\]
		To control the first sum, we use \eqref{eq:good-subfamily-j} to get
		\[
		\sum_{j=1}^{J} p_j A_j \le p \max_{j=1}^{J} A_j \le \mu_{\sx}(2\cdot Q) \, C_{\lambda,q} 2^{-(D_{\cX}(q)+\eta)(q-1)m}\,.
		\]
		For the second sum, we estimate
		\begin{align*}
			\sum_{j=1}^{J} p_j B_j & \overset{\eqref{eq:def-p-Q'}}{\le}  \sum_{Q'\in\cD'} p(Q') \sum_{j=1}^{J} \mu_{\bT^n(\sx)}(\lambda^{-n}(Q'-z_j))^q             \\
			& \le 2^{-\tfrac{\eta}{4}m} \mu_{\sx}(2\cdot Q) \max_{j=1}^{J} \sum_{Q'\in\cD'} \mu_{\bT^n(\sx)}(\lambda^{-n}(Q'-z_j))^q         \\
			& \overset{\eqref{eq:small-loss-orig-family}}{\le} \mu_{\sx}(2\cdot Q) \, 2^{-\left(D_{\cX}(q)+\tfrac{\eta}{8}\right)(q-1)m} \,.
		\end{align*}
		
		Combining these estimates with \eqref{eq:application-subadditive} finishes the proof (with $\eta/10$ in place of $\eta$ and making $m_0$ even larger in terms of $\lambda,q$ if needed).
	\end{proof}

	\subsection{Exponential flattening of the $L^{q}$-norm under convolution}

	The following theorem, to which we shall appeal crucially in the proof of Proposition~\ref{prop:finerscales}, asserts that convolving with dynamically driven self-similar measures, discretized at finite scales, increases regularity. Recall that a $2^{-m}$-measure ($m\geq 1$ an integer) is a probability measure with finite support contained in $2^{-m}\Z^d$. Also, $q'$ indicates the Hölder conjugate of $q$.
	\begin{thm}
		\label{thm:lqflattening}
		For every $\sigma>0$ there exist $\eta>0$ and $m_0\in\N$ (both depending on $\cX, q$ and $\sigma$) such that, for every $\sx\in \sX$, every $m\geq m_0$ and every $2^{-m}$-measure $\nu$ supported inside $[0,1)^d$ and satisfying $\norm{\nu}_q^{q'}\leq 2^{-\sigma m}$,
		\begin{equation}
			\label{eq:lqflattening}
			\norm{\nu\ast \mu_{\sx}^{(m)}}_q^{q'}\leq 2^{-(T_{\cX}(q)+\eps)m}\;.
		\end{equation}
	\end{thm}

	In a nutshell, here is the overarching strategy of the proof. If the inequality in~\eqref{eq:lqflattening} is violated, then the hypotheses of Theorem~\ref{thm:invthm} are satisfied. The latter yields, combined with the unsaturation assumption on the model, a dyadic multiscale decomposition of the $2^{-m}$-measures $\nu$ and $\mu_{\sx}$ of the following type: a huge portion of the $L^{q}$-norm of $\mu_{\sx}^{(m)}$ and the $L^{1}$-norm (that is, the mass) of $\nu$ are captured by subsets $A$ and $B$ of the respective supports verifying that, at each intermediate scale, there is either no branching for $B$ or line saturation for $A$. The inequality $\norm{\nu}_q^{q'}\leq 2^{-\sigma m}$ implies that, for a positive proportion of scales, there must be some non-trivial branching for $B$, which thus corresponds to line saturation for $A$. Using Proposition \ref{prop:gain-from-unsaturation} we will show that, after refining the set $A$ to a ``dense'' set $A'$, each such scale entails a loss of $L^q$ norm for $A'$, and because there is a positive proportion of such scales, $A'$ carries only a small (in exponential sense) part of the $L^{q}$ dimension $\mu_{\sx}^{(m)}$. But on the other hand, $A'$, being a large subset of $A$ on which $\mu$ is roughly constant, should capture a large part of the $L^{q}$ norm of $\mu_{\sx}^{(m)}$. This contradiction demonstrates that \eqref{eq:lqflattening} must hold.

	\subsection{Proof of Theorem~\ref{thm:lqflattening}}
	
	\subsubsection{Setup and counter-assumption}
	Throughout the proof, we shall keep track of the interdependencies between the various parameters through round brackets, e.g.~we shall write $\delta=\delta(\alpha,\beta,\gamma)$ if $\delta$ depends on $\alpha,\beta,\gamma$.

	Assume $0<D_{\cX}(q)<d$. Fix $\sigma>0$, an integer $m\geq 1$ and a point $\sx\in \sX$, and let $\nu$ be a $2^{-m}$-measure satisfying $\norm{\nu}_q^{q'}\leq 2^{-\sigma m}$. During the proof, we drop the $\cX$ subscript from $D_{\cX}(q)$ and $T_{\cX}(q)$ for brevity.

	The statement of Theorem~\ref{thm:lqflattening} is plainly invariant under dilation of any involved measure by a common scaling factor, and also by translation; therefore, it suffices to prove it assuming that every measure $\mu_{\sx}$ ($\sx\in \sX$), as well as $\nu$, are supported inside $[0,1)^d$.
	
	In the course of the proof, we shall preliminarily choose parameters in order to be in a position to apply the needed results from previous sections, and impose restrictions on them according to the successive needs emerging from the argument. To avert circularity in such a selection process, we will collect all the mutual dependencies between the various parameters at the end of the proof.
	
	Fix $\delta>0$ (to be determined later), and let
	\begin{equation} \label{eq:choice-D0}
		L_0=L_0(\delta,q)\in\N,
	\end{equation}
	be the parameters given by  Theorem~\ref{thm:invthm}. Let
	\begin{equation} \label{eq:choice-L}
		L\ge L_0(\delta,q)
	\end{equation}
	be another to parameter to be determined during the proof, and let
	\begin{equation} \label{eq:choice-eps}
		\eps=\eps(\delta,q,L) > 0
	\end{equation}
	be the value provided by Theorem~\ref{thm:invthm}. Further, let
	\[
	S_0=S_0(q,\delta,L,\eps)\geq 1
	\]
	be again given by Theorem~\ref{thm:invthm}.

	We shall prove the theorem with the value
	\begin{equation}
		\label{eq:epsprime}
		\eps'=\eps q/3
	\end{equation}
	in place of $\eps$. For the sake of contradiction, assume therefore that, for some $\sx\in\sX$,
	\begin{equation}
		\label{eq:contradassumption}
		\norm{\nu\ast \mu^{(m)}_{\sx}}_q^{q}>2^{-(T(q)+\eps')m}\;.
	\end{equation}
	
	By Proposition~\ref{prop:Tq}(3) and Lemma~\ref{lem:equallimits},
	\begin{equation*}
		\liminf_{m\to\infty}-\frac{\log{\norm{\mu_{\sy}^{(m)}}_q^q}}{m}\ge T(q),
	\end{equation*}
	uniformly in $\sy\in \sX$. Hence, there is an integer $m_1=m_1(q,\eps)\geq 1$ such that, for every integer $m\geq m_1$,
	\begin{equation}
		\label{eq:boundslqnorm}
		2^{-(T(q)+\eps q/2)m}<\norm{\mu_{\sx}^{(m)}}_q^q<2^{-(T(q)-\eps q/2)m}\;.
	\end{equation}
	The left-hand side inequality follows from our counter-assumption \eqref{eq:contradassumption}, Lemma \ref{lem:discrconv} and Young's inequality, in the form $\|\mu_{\sx}^{(m)}*\nu\|_q \le \|\mu_{\sx}^{(m)}\|_q\|\nu\|_1$.
	
	From now on, $m$ is always assumed to satisfy $m\geq m_1$. Combining the right-most inequality in~\eqref{eq:boundslqnorm} with~\eqref{eq:contradassumption}, we get
	\begin{equation*}
		\norm{\nu\ast \mu^{(m)}_{\sx}}_q^{q}> 2^{-(T(q)+\eps q/3)m}>2^{-\eps q m}2^{-(T(q)-\eps q/2)m}>2^{-\eps q m}\norm{\mu_{\sx}^{(m)}}_q^q\;,
	\end{equation*}
	that is,
	\begin{equation} \label{eq:noflattening}
		\norm{\nu\ast \mu_{\sx}^{(m)}}_q> 2^{-\eps m}\norm{\mu_{\sx}^{(m)}}_{q}\;.
	\end{equation}
	
	\subsubsection{Application of the Inverse Theorem}
	
	We aim to apply Theorem~\ref{thm:invthm} to the $2^{-m}$-measures $\nu$ and $\mu_{\sx}^{(m)}$, with $\delta$ as above.  Observe that we may assume without loss of generality that $m$ is an integer multiple of $L$. Thus, we assume from now on
	that
	\begin{equation}
		m= LS \ge LS_0.
	\end{equation}
	We now deduce from Theorem~\ref{thm:invthm}, which is indeed applicable by \eqref{eq:noflattening},  the existence of subsets $A\subset \supp{\mu_{\sx}^{(m)}}$, $B\subset \supp{\nu}$  satisfying all the properties listed in Theorem~\ref{thm:invthm}. Let
	\begin{equation*}
		\cS_0=\{s\in [S]:R''_s=1  \},\quad  \cS_1= [S]\setminus \cS_0 \;
	\end{equation*}
	be the sets of scales at which  $B$ has no branching and some branching, respectively. (Recall from \S\ref{subsec:inversethm} that the notation $[S]$ stands for the set $\{0,\dots,S-1\}$.) We deduce from the hypothesis $\norm{\nu}_q^{q'}\leq 2^{-\sigma m}$ and from~\eqref{eq:somebranching} in Theorem~\ref{thm:invthm} that
	\begin{equation}
		\label{eq:lbsomebranching}
		|\cS_1|\geq  \frac{-\log{\norm{\nu}_q^{q'}}-q'(\delta m+1)}{dm}\; S\geq \frac{\sigma-q'\delta-q'/m}{d}\;S\;.
	\end{equation}
	In the sequel, suppose that $m\ge \delta^{-1}$. From~\eqref{eq:lbsomebranching} we obtain the lower bound
	\begin{equation}
		\label{eq:scaleswithsomebranching}
		|\cS_1|\geq \frac{\sigma-2\delta q'}{d}\;S\;.
	\end{equation}
	Suppose now $\delta$ is chosen so that
	\begin{equation}
		\label{eq:firstdelta}
		\delta<\frac{\sigma}{4q'}=\frac{\sigma(q-1)}{4q}\;;
	\end{equation}
	then~\eqref{eq:scaleswithsomebranching} leads to the lower bound
	\begin{equation} \label{eq:S1-lower-bound}
		|\cS_1|\geq \frac{\sigma}{2d}\;S\;.
	\end{equation}
	Recall that the line saturation property \eqref{eq:saturation} holds for all scales $s\in\cS_1$; we have seen this is the case for a positive proportion of all scales.
	
	\subsubsection{Refinement of the set $A$}
	
	We will now use Proposition \ref{prop:gain-from-unsaturation} to refine $A$ to a subset $A'$ from which we will eventually extract a contradiction.
	
	We trim the tree from the top. Set $A'_{(0)}=A$. Suppose $A'_{(s)}$ has been defined for some $s\in [S-1]$, as a union of sets $Q\cap A$, $Q\in\cD_{sL}(A'_{(s)})$; in particular, if $Q\in\cD_{sL}(A'_{(s)})$, then
	\begin{equation} \label{eq:QcapA}
		Q\cap A'_{(s)}=Q\cap A.
	\end{equation}
	If $s\in\mathcal{S}_0$, then we set $A'_{(s+1)}=A'_{(s)}$. Otherwise, if $s\in\mathcal{S}_1$, we proceed as follows. For each $Q\in\cD_{sL}(A'_{(s)})$, we know from \eqref{eq:QcapA} and \eqref{eq:saturation} that there exists $\pi_Q\in\G(d,d-1)$ such that
	\[
	\cN_{(s+1)L}\bigl(\pi_Q(Q\cap A'_{(s)})\bigr) \lesssim 2^{(\delta-1)L}\cN_{sL}(Q\cap A'_{(s)}).
	\]
	This shows that the saturation assumption \eqref{eq:local-saturation-assumption} of Proposition \ref{prop:gain-from-unsaturation} is satisfied, with $sL$ in place of $s$ and $L$ in place of $m$. In order to be able to apply the proposition, we assume henceforth that
	\begin{equation}\label{eq:choice-L-eta-delta}
		L \ge m_0(\cX,q)\;,\quad \eta = \eta(\cX,q)\;,\quad \delta \le \frac{(q-1)\sigma\eta}{14d}\;,
	\end{equation}
	where $\eta(\cX,q)$, $m_0(\cX,q)$ are the thresholds from Proposition \ref{prop:gain-from-unsaturation}. Now the proposition yields, for each $Q\in\cD_{sL}(A'_{(s)})$, a set $A'_Q\subset A'_{(s)}\cap Q$, which is a union of cubes in $\cD_{(s+1)L}$, such that
	\begin{equation} \label{eq:refinement-size}
		\cN_{(s+1)L}(A'_Q) \ge \frac{1}{2} \cN_{(s+1)L}\bigl(A'_{(s)}\cap Q\bigr) = \frac{R'_s}{2},
	\end{equation}
	and
	\begin{equation} \label{eq:refinement-Lq-loss}
		\sum_{Q'\in\cD_{(s+1)L}(A'_Q)} \mu_{\sx}(Q')^q \le 2^{-(D(q)+\eta)(q-1)L} \, \mu_{\sx}(2\cdot Q)^{q}\,.
	\end{equation}
	We define
	\begin{equation} \label{eq:def-inductive-A'}
		A'_{(s+1)} = \bigcup \bigl\{ A'_Q : Q\in\cD_{sL}(A'_{(s)})\bigr\}\;.
	\end{equation}
	We stop the construction when $s=S-1$ and set $A' = A'_{(S-1)}$.
	
	It follows from \eqref{eq:refinement-size} that for each $s\in [S]$ and each $Q\in\cD_{sL}(A')$, the branching number $\cN_{(s+1)L}(A'\cap Q)$ is at least $R_s/2$. From this a simple induction in $s$ shows that $|A'|\ge 2^{-S}|A|$. Using properties (A1)--(A2) of $A$ from Theorem \ref{thm:invthm}, we deduce that
	\begin{align*}
		\sum_{Q\in\cD_{m}(A')} \mu_{\sx}(Q)^q & \ge \min_{Q\in\cD_{m}(A')}\mu_{\sx}(Q)^q \, |A'| \ge 2^{-q}\, \max_{Q\in\cD_{m}(A)}\mu_{\sx}(Q)^q\,|A'| \\
		& \ge 2^{-(S+q)}\, \max_{Q\in\cD_{m}(A)}\mu_{\sx}(Q)^q\,|A|
		\\
		& \ge 2^{-(L+q)}\,\|\mu_{\sx}^{(m)}|_A\|_q^q                                                              \\
		& \overset{\ref{IT:A1}}{\ge} 2^{-(S+q+\delta m)} \|\mu_{\sx}^{(m)}\|_q^q                                  \\
		& \overset{\eqref{eq:boundslqnorm}}{\ge} 2^{-(T(q)+L^{-1}+qm^{-1}+\delta+\eps q/2)m}.
	\end{align*}
	By taking
	\begin{equation} \label{eq:choice-ell0-eps-L}
		\eps\le (2/q)\delta\;, \quad S_0 \ge q\delta^{-1}\;, \quad L^{-1}<\delta\;,
	\end{equation}
	the above bound simplifies to
	\begin{equation} \label{eq:lower-bound-qnorm-A'}
		\|\mu_{\sx}|_{A'}^{(m)}\|_q^q \geq  2^{-(T(q)+4\delta)m}.
	\end{equation}
	That is, the refinement $A'$ still captures a substantial part of the $L^q$ norm of $\mu_{\sx}^{(m)}$. In the remainder of the proof we will invoke Proposition \ref{prop:gain-from-unsaturation} (in the form of \eqref{eq:refinement-Lq-loss} above) to show this cannot in fact happen.
	
	\subsubsection{Loss of $L^q$-norm: obtaining a contradiction}
	
	For $s\in[S]$, let
	\[
	K_{s} = \bigl\|\mu_{\sx}|_{A'}^{(sL)}\bigr\|_q^q .
	\]
	We will estimate $K_{s}$ inductively, starting with $L_0\le 1$. For $s\in\mathcal{S}_0$, we aim to apply Corollary \ref{cor:subadditive-single-cube} with $Ls$ in place of $s$ and $L$ in place of $m$; this will ensure that the loss at these scales is sufficiently small compared to the gain for the scales in $\mathcal{S}_1$. In order for this to be valid, we further assume that
	\begin{equation} \label{eq:choice-L-2}
		L  \ge m_0(\cX,\delta),
	\end{equation}
	where $m_0(\cX,\delta)$ is as in Corollary \ref{cor:subadditive-single-cube}. Now invoking the corollary for each $Q\in\cD_{sL}(A')$ and then adding up over all such $Q$, we get
	\begin{equation} \label{eq:bound-Ls-0}
		K_{s+1}\le C_{\lambda,q} 2^{-(T(q)-\delta)L} \sum_{Q\in\cD_{sL}(A')} \mu(2\cdot Q)^q \lesssim_{\lambda,q} 2^{-(T(q)-\delta)L} K_{s}\,,
	\end{equation}
	where we used Lemma \ref{lem:measurefinitecover} to replace $2\cdot Q$ by $Q$ in the last inequality (at the price of a multiplicative constant that we absorbed into $\lesssim_{\lambda,q}$).
	
	Suppose now that $s\in\mathcal{S}_1$, and fix $Q\in\cD_{sL}(A')$. By \eqref{eq:refinement-Lq-loss} and \eqref{eq:def-inductive-A'}, we have that
	\[
	\sum_{Q'\in\cD_{(s+1)L}(A'\cap Q)} \mu(Q')^{q} \le 2^{-(T(q)+(q-1)\eta)L} \mu(2\cdot Q)^{q}\,.
	\]
	Adding up over all $Q\in\cD_{sL}(A')$ and using  Lemma \ref{lem:measurefinitecover} once again, we deduce that
	\begin{equation} \label{eq:bound-Ls}
		K_{s+1}\lesssim_{\lambda,q}  2^{-(T(q)+(q-1)\eta)L}  K_{s}.
	\end{equation}
	Combining $L_0\le 1$, \eqref{eq:bound-Ls-0} and \eqref{eq:bound-Ls}, we finally bound
	\[
	\bigl\|\mu_{\sx}|_{A'}^{(m)}\bigr\|_q^q = L_{S} \le O_{\lambda,q}(1)^{S} 2^{\left(|\mathcal{S}_0|\delta-|\mathcal{S}_1|(q-1)\eta\right) L} 2^{-T(q)m}\,.
	\]
	Taking $S_0$ large enough in terms of $\cX$ and $\delta$, we can bound $O_{\lambda,q}(1)^{S} \le 2^{\delta m}$, while obviously $|\mathcal{S}_0|L\le m$, so we can further infer
	\[
	\bigl\|\mu_{\sx}|_{A'}^{(m)}\bigr\|_q^q \le 2^{2\delta m} 2^{-|\mathcal{S}_1|(q-1)\eta L} 2^{-T(q)m}\,.
	\]
	Now recall from \eqref{eq:S1-lower-bound} and \eqref{eq:choice-L-eta-delta} that
	\[
	|\mathcal{S}_1|(q-1)  \eta L \ge \frac{\sigma(q-1)\eta}{2d} m \ge 7\delta m
	\]
	The last two estimates allow us to conclude
	\[
	\bigl\|\mu_{\sx}|_{A'}^{(m)}\bigr\|_q^q \le 2^{-(T(q)+5\delta)m}.
	\]
	This, however, contradicts \eqref{eq:lower-bound-qnorm-A'}. Thus the counter-assumption \eqref{eq:contradassumption} cannot hold, and this completes the proof.
	
	\subsubsection{Recapitulating the selection of parameters}
	
	As previously announced, we conclude this subsection by checking the consistency of our successive choices of parameters. We start by the given data in the hypotheses of Theorem~\ref{thm:lqflattening}, namely  the model $\cX=(\sX,\bT,\bP,\Delta,\lambda)$, the real numbers $q>1$ and $\sigma>0$ and the ambient dimension $d$. All subsequent parameters can depend on this data, without further mention; parameters that depend only on the given data are taken as absolute.
	
	The parameter $\eta$ is defined in \eqref{eq:choice-L-eta-delta}, so it is absolute. Next, $\delta$  must satisfy the inequalities \eqref{eq:firstdelta} and \eqref{eq:choice-L-eta-delta}, so that $\delta=\delta(\eta)$. The integer $L_0$ is defined in \eqref{eq:choice-D0}; thus $L_0=L_0(\delta)$.  Next, $L$ must satisfy  \eqref{eq:choice-L},  \eqref{eq:choice-L-eta-delta}, \eqref{eq:choice-ell0-eps-L}, and \eqref{eq:choice-L-2}, and therefore $L=L(L_0,\delta)$. The parameter $\eps$ is determined by \eqref{eq:choice-eps} and \eqref{eq:choice-ell0-eps-L}, thus $\eps=\eps(\delta,L)$. Then $\eps'$, which is the outcome of the theorem and is defined in \eqref{eq:epsprime}, satisfies $\eps'=\eps'(\eps)$ and so ultimately depends only on the data in the statement, as desired.

	Finally, $m=LS\ge LS_0$ or, equivalently, $S_0$, is at many places taken large enough in terms of all the previously mentioned parameters.
	
	\subsection{Conclusion of the proof of Theorem~\ref{thm:lqdimension}}

	The remaining of the proof of Theorem \ref{thm:lqdimension} is virtually identical to the one-dimensional case treated in \cite{Shmerkin-Annals}.
	
	We continue working with a pleasant model $\cX=(\sX,\mathbf{T},\bP,\Delta,\lambda)$ in $\R^d$, generating a collection $(\mu_{\sx})_{\sx\in \sX}$ of dynamically driven self-similar measures. We assume $\cX$ is $q$-unsaturated on lines for some fixed $q\in\R_{>1}$. Recall that this implies that $D_{\sX}(q)<d$.

	The next proposition states that there is no loss in the exponential rate of decay of the $L^{q}$-norms of the discretizations, when looking at exponentially finer scales. The case $d=1$ is \cite[Proposition 5.2]{Shmerkin-Annals}. The same proof applies to the general case, appealing to Theorem \ref{thm:lqflattening} instead of \cite[Theorem 5.1]{Shmerkin-Annals}.
	\begin{prop}
		\label{prop:finerscales}
		Let $\sx\in\sX$ be such that
		\begin{equation}
			\label{eq:rightconvergence}
			\lim\limits_{n\to\infty}-\frac{\log{\norm{\mu_\sx^{(m(n))}}_q^q}}{m(n)}=T_{\cX}(q)\;.
		\end{equation}
		Then, for every integer $R\geq 1$,
		\begin{equation*}
			\lim\limits_{n\to\infty}-\frac{\log{\norm{\mu_{\sx,n}^{(Rm(n))}}_q^q}}{m(n)}=T_{\cX}(q)\;.
		\end{equation*}
	\end{prop}
	
	With this proposition in hand, the proof of \eqref{eq:Tq-lower-bound} and therefore of Theorem~\ref{thm:lqdimension} is completed exactly as in \cite[\S 5.3]{Shmerkin-Annals}. We emphasize that exponential separation is only used in this very last step, in the following way: if $\sx\in\sX$, $n\in\N$ and $R\ge 1$ are such that the atoms of $\mu_{\sx,n}$ are $\lambda^{Rn}$-separated, then
	\begin{equation} \label{eq:consequence-exp-sep}
		\|\mu_{\sx,n}^{(Rm(n))}\|_q^q = \|\mu_{\sx,n}\|_q^q
	\end{equation}
	Exponential separation ensures that this happens for $\bP$-almost all $\sx$, some $R=R(\sx)\ge 1$, and infinitely many $n$ (depending on $\sx$). In particular, Proposition \ref{prop:Tq}(ii) ensures that we can fix $\sx\in\sX$ so that \eqref{eq:rightconvergence} holds, and \eqref{eq:consequence-exp-sep} holds for infinitely many $n$, and some fixed $R\ge 1$. The claim   \eqref{eq:Tq-lower-bound} is now a consequence of \eqref{eq:consequence-exp-sep} and Proposition \ref{prop:finerscales}.

	\section{Homogeneous self-similar measures and their projections}
	
	\label{sec:ssm}
	
	\subsection{On exponential separation and unsaturation on lines}
	\label{subsec:applications-prelim}
	
	Given $q>1$ and a pleasant model $\cX=(\sX,\bT,\bP,\Delta,\lambda)$ in $\R^d$, we denote by
	\[
	D_{\cX}^{\text{s}}(q) = \frac{\int_{\sX} \log\|\Delta(\sy)\|_q^q \text{d}\bP(\sy) }{(q-1)\lambda}
	\]
	the ``symbolic'' or ``expected'' $L^q$-dimension of the measures generated by $\cX$. Theorem \ref{thm:lqdimension} then asserts that, if the model is $q$-unsaturated on lines and has exponential separation, then
	\begin{equation} \label{eq:dim-equals-symbolic-dim}
		D_{\cX}(q) = D_{\cX}^{\text{s}}(q)\;,
	\end{equation}
	and hence the $L^q$-dimension of the measures generated by $\cX$ takes the ``expected'' value.
	We begin this section by making some general comments on how to verify exponential separation and $q$-unsaturation in practice.
	
	As mentioned earlier, the exponential separation condition is a variant of the fundamental notion introduced by M.~Hochman~\cite{Hochman14,Hochman17}, and a natural higher-dimensional analog of the notion introduced in~\cite{Shmerkin-Annals}.
	The main difference with Hochman's definition is that the latter applies to single self-similar measures, whereas ours contemplates more generally dynamically driven self-similar measures. For several models of interest, the requirement of exponential separation holding only for \emph{almost all} measures makes the verification of the condition far easier. We will see examples in the proof of Theorem \ref{thm:ssm} and in Lemma \ref{lem:convolution-exp-sep} below.
	
	We now discuss the $q$-unsaturation condition.  As remarked in the introduction, for any projection $\pi\in\bigsqcup_{0\leq k\leq d}\G(d,k)$, the projected model $\pi\cX$ is pleasant whenever $\cX$ is. It will be crucial to understand how exponential separation behaves under projections. To this effect, the variant of exponential separation for projected models $\pi\cX$ introduced in Definition~\ref{def:projected-exp-sep}, which we termed \emph{projected exponential separation} (PES property, in short) and postulates, in addition to exponential separation, that $\pi|_{\supp \Delta(\sx)}$ is injective for $\bP$-almost every $\sx\in \sX$, is particularly suited in that it takes into account that the model arises from a projection.

	Note that knowledge of exponential separation for $\pi\cX$ does not provide useful information about the original model $\cX$ when $\pi$ fails to be injective on $\supp\Delta(\sx)$ for a positive $\bP$-measure set of $\sx\in \sX$. On the other hand, we will shortly see that a lot of information about $\cX$ can be extracted from the projected model when operating under the PES property.
	
	We record two direct and yet important consequences of projected exponential separation.
	\begin{enumerate}[(\alph*)]
		\item \label{it:PES:a} If $\pi\cX$ has the PES property, then $\|\pi\Delta(\sx)\|_q^q =\|\Delta(\sx)\|_q^q$ for  $\bP$-almost all $\sx\in\sX$, and therefore  $D_{\pi\cX}^{\text{s}}(q)  = D_{\cX}^{\text{s}}(q)$.
		\item \label{it:PES:b} If $\pi\subset\Pi$ are in $\bigsqcup_{0\leq k\leq d}\G(d,k)$ and $\pi\cX$ has the PES property, then so does $\Pi\cX$.
	\end{enumerate}
	
	Projected exponential separation does not follow, in general, from exponential separation; however, as we shall shortly see, it often does - sometimes for all $\pi$, sometimes for all $\pi$ outside some sparse set. This affords a sort of ``project and induct'' strategy for verifying $q$-unsaturation, which is summarized in the next lemma. Recall the notation $\G(d)$ and $\G(\pi,k)$ from \textsection \ref{subsubsec:generalnotation}.
	\begin{lem} \label{lem:one-dim-exp-sep}
		$\cX=(\sX,\bT,\bP,\Delta,\lambda)$ be a pleasant model in $\R^d$. Let
		\[
		\mathcal{E} = \left\{ \pi\in\G(d,1) : \pi\cX \text{ satisfies projected exponential separation} \right\}\;.
		\]
		Then, for all $\pi\in \G(d)$ such that $\G(\pi,1)\subset \cE$, 
		\begin{equation} \label{eq:inductive-conclusion}
			D_{\pi\cX}(q) = \min\left\{ k, D_{\cX}^{\text{s}}(q)\right\} \quad \text{for all $q>1$.}
		\end{equation}
		In particular, if $\mathcal{E}=\G(d,1)$, then \eqref{eq:inductive-conclusion} holds for all $\pi\in\G(d)$.
	\end{lem}
	\begin{proof}
		Fix $q>1$, and call $s=D_{\cX}^{\text{s}}(q)$. By Proposition \ref{prop:upperbound}, the inequality
		\begin{equation} \label{eq:one-dim-exp-sep-upperbound}
			D_{\pi\cX}(q)\le \min\{k,s\}
		\end{equation}
		holds for all $\pi\in \G(d)$; hence, we need to show the opposite inequality.
		
		We proceed by induction on $k$. Consider the base case $k=1$. Let $\pi\in\mathcal{E}$. If $D_{\pi\cX}(q)=1$, we are done. Otherwise, $D_{\pi\cX}$ is $q$-unsaturated on lines (recall that in dimension $1$ this is \emph{equivalent} to $D_{\pi\cX}(q)<1$), and therefore $D_{\pi \cX}(q)=s$ by virtue of~ \ref{it:PES:a} and Theorem~\ref{thm:lqdimension}.
		
		Assume the claim holds for some $1\le k<d$. Let $\pi\in\G(d,k+1)$ satisfy $\G(\pi,1)\subset \mathcal{E}$. Since $\pi'\cX$ satisfies projected exponential separation for $\pi'\in\G(\pi,1)$, then so does $\pi\cX$. If $D_{\pi\cX}(q)=k+1$ then we are done, so assume otherwise. Note that if $\pi'\in\G(\pi,k)$, then $\pi'$ satisfies $\G(\pi',1)\subset\mathcal{E}$ as well. By \eqref{eq:one-dim-exp-sep-upperbound} and the inductive hypothesis, for all $\pi'\in\G(\pi,k)$ we have
		\[
		\dim_{\pi'\cX}(q)+1 =\min \{ k+1,s+1\} > \dim_{\pi\cX}(q)\;.
		\]
		It follows that $\pi\cX$ is $q$-unsaturated on lines, and we can apply Theorem \ref{thm:lqdimension} to conclude that $D_{\pi\cX}(q)=s$.
	\end{proof}

	\subsection{\texorpdfstring{$L^q$}{L^q}-dimensions of homogeneous self-similar measures and their projections}
	
	\label{subsec:ssm}
	
	In this section, we apply Theorem \ref{thm:lqdimension} to establish the $L^q$-spectrum of self-similar measures and their projections, under suitable conditions.
	
	There has been great interest in computing the dimension of self-similar sets and measures with overlaps in higher dimensions. The most general results for the Hausdorff dimension are due to Hochman \cite[Theorems 1.4 and 1.5]{Hochman17}, who showed that under a suitable notion of exponential separation, if there is no proper non-trivial linear subspace invariant under all orthogonal parts of the similarities, then the Hausdorff dimension of the self-similar set, and of all the self-similar measures it supports, takes the expected value. The more classical \emph{transversality method} can be used to show that, in many parametrized families of self-similar measures satisfying the so-called transversality conditions, the $L^q$-dimension takes the expected value for almost all values, but \emph{only} for the range $1<q\leq 2$. We refer to the recent monograph~\cite[Chapter 6]{BaranySimonSolomyak23} by Bárány, Simon and Solomyak for an introduction to the transversality method.
	
	Computing the dimension of projections of self-similar sets and measures is also an active area of inquiry. For general Borel sets and measures, the transversality method provides the Hausdorff and $L^q$-dimension, when $1<q\leq 2$, for Lebesgue-almost all projections; in the case of self-similar measures, we expect to be able to say something about \emph{all} projections. Assuming the strong separation condition and transitivity of the action of the orthogonal parts on the Grassmannian, M.~Hochman and the second author \cite[Theorem 1.6]{HochmanShmerkin12} proved preservation of Hausdorff dimension for all projections. K.~Falconer and X.~Jin \cite{FalconerJin14} removed all separation conditions, under the assumption that the orthogonal parts generate a dense subgroup of the special orthogonal group $\SO_d(\R)$. Very recently, A.~Algom and the second author \cite{AlgomShmerkin24} substantially weakened the transitivity assumption, in particular finding sharp conditions on the group generated by the orthogonal parts that ensure preservation of Hausdorff dimension for all projections to lines and hyperplanes. All these results apply only to Hausdorff dimension. For $L^q$-dimensions, \cite{Shmerkin-Annals} has some rather complete results, but only for projections from $\R^2$ to lines.
	
	In this section, we compute the $L^q$-dimensions of many self-similar measures in $\R^d$, as well as of their projections; our method, however, can only handle the case of homogeneous self-similar measures (the class defined in \textsection \ref{subsec:selfsimilar}), so we focus on this case from now on.
	
	As discussed in \textsection \ref{subsec:selfsimilar}, any homogenoeus self-similar measure in $\R^d$ admits the following description. Let $h\in \Or_d(\R)$, $\lambda\in (0,1)$, $(a_i)_{i\in \R^d}$ a collection of vectors in $\R^d$, $p=(p_i)_{i\in I}$ a probability vector; set
	\[
	\Delta_0 = \sum_{i\in I} p_i\, \delta_{a_i}\;.
	\]
	Then the infinite convolution product
	\begin{equation} \label{eq:homogeneous-ssm}
		\mu = \bigast_{n=0}^\infty S_{\lambda^n}h^n\Delta_0\;.
	\end{equation}
	is a homogeneous self-similar measure, and any such (infinitely supported, in view of the condition $\lambda>0$) measure can be obtained in this way.
	
	To avoid trivialities, we assume that the support of $\Delta_0$ is not a singleton.  We have already established that $\mu$ arises as a dynamically driven self-similar measure generated by the pleasant model $\cX=(\sX,\bT,\bP,\Delta,\lambda)$ where $\sX$ is the closed subgroup of $\Or_d(\R)$ generated by $h$, $\bT\colon \sX\to \sX$ is the translation map by $h$, $\bP$ is the probability Haar measure on $\sX$ and $\Delta(g)=g\Delta_0$ for all $g\in \sX$. The measures generated by $\cX$ are $\mu_{g}=g\mu$, $g\in \sX$.
	
	We now formulate the conditions of $q$-unsaturation and exponential separation directly in terms of $\mu$ and of the IFS generating it; the latter is, we recall, the collection of similarities $\Phi=\{f_i \}_{i\in I}$ on $\R^d$ given by
	\eq{f_i(x)=\lambda h(x)+a_i\;, \quad i\in I.}
	The IFS $\Phi$ satisfies \emph{exponential separation} if there is $c>0$ and an increasing sequence of integers $(n_j)_{j\geq 1}$ such that, for all $j\geq 1$, 
	\eq{|f_{i_1}\circ \cdots \circ f_{i_{n_j}}(0)-f_{i'_1}\circ \cdots \circ f_{i'_{n_j}}(0)|\geq c^{n_j} \quad \text{ for all }(i_1,\dots,i_{n_j})\neq (i'_1,\dots,i'_{n_j})\in I^{n_j}\;.}
	Given $q>1$, we say that the measure $\mu$ is \emph{$q$-unsaturated on lines} if 
	\eq{\dim(\pi\mu,q)>\dim(\mu,q)-1}
	for all $\pi\in \G(d,d-1)$.
		
	
	We are now in a position to phrase the following corollary of Theorem~\ref{thm:lqdimension}.
	\begin{cor} \label{cor:ssm}
		Let $\mu$ be a homogeneous self-similar measure in $\R^d$ generated, as above, by an iterated function system $\Phi$ and a probability vector $p$. Suppose $\Phi$ satisfies exponential separation, and $\mu$ is $q$-unsaturated on lines for some $q>1$. 
		Then
		\begin{equation*}
			\dim(\mu,q)= \frac{\log{\norm{p}_q^{q}}}{(q-1)\log{\lambda}}\; .
		\end{equation*}
	\end{cor}
	\begin{proof}
		It is a straightforward consequence of Theorem~\ref{thm:lqdimension}. Since all measures $\mu_{g}$ are isometric images of $\mu$, and likewise for  the approximations $\mu_{g,n}$, exponential separation and $q$-unsaturation on lines for the model $\cX$ generating $\mu$ are inherited from the corresponding assumptions on $\mu$. Finally, 
		\eq{\|\Delta(g)\|_q^q=\norm{\sum_{i\in I}p_i\; \delta_{g(a_i)}}_q^q =\norm{p}_q^q}
		for all $g\in\sX$, from which it follows that
		\eq{ D_{\cX}^{\text{s}}(q) = \frac{\log\|p\|_q^q}{(q-1)\log\lambda}\;.}
		Theorem~\ref{thm:lqdimension} delivers the conclusion. 
	\end{proof}
	
	If the orthogonal part $h$ has repeated complex eigenvalues, then some projections will always fail to have the PES property; it may happen that $\pi$ fails to be injective on $\supp{\Delta(\sx)}$ for all $\sx\in \sX$. As a matter of fact, the conclusion of Corollary \ref{cor:ssm} may fail. See \cite[Example 1.2]{Hochman17}, which discusses Hausdorff dimension of self-similar sets, but the same phenomenon occurs for $L^q$-dimensions of self-similar measures. Hence, verifying $q$-unsaturation is delicate and should be done on a case-by-case basis.
	
	Under the assumption that $h$ has distinct complex eigenvalues and generates a subgroup with connected closure, Theorem~\ref{thm:ssm}, whose proof occupies us for almost the entirety of the remainder of this section, gives a mild condition on the self-similar measure so that the assumptions of Corollary~\ref{cor:ssm} are satisfied.

	Before proceeding to the proof of Theorem~\ref{thm:ssm}, we make some remarks.
	\begin{rmk}
		The condition that $h$ has distinct complex eigenvalues and generates a group with connected closure can be formulated explicitly in terms of the eigenvalues of $h$. Since $h$ is orthogonal, its eigenvalues are of the form $e^{2\pi i \alpha_i}$, $i=1,\ldots,d$, where $\alpha_i\in\R$. The condition is then that $\alpha_i$ and $\alpha_i-\alpha_j$ are irrational for all $i\neq j$. In particular, there are no real eigenvalues if $d$ is even, and $1$ is a simple eigenvalue if $d$ is odd.
	\end{rmk}
	
	\begin{rmk}
		We do not know what happens when $h$ has distinct complex eigenvalues but the closure of the group generated by $h$ is not connected. This includes, for instance, the case of rational rotations on the plane. Exponential separation for $\pi\cX$ is not guaranteed in this case, and yet it may still hold true that $\mu$ is $q$-unsaturated whenever it satisfies exponential separation and has $L^q$-dimension less than $d$.
	\end{rmk}
	
	\begin{rmk}
		In view of~\cite[Lemma 6.30]{Hochman17}, if all entries of the matrix of $h$ (in the standard basis of $\R^d$), as well as the entries of the elements of $\supp\Delta_0$, are algebraic over $\Q$, then the projections $\pi_j \mu$ satisfy projected exponential separation if and only if they have no exact overlaps, that is, if and only if the support of 
		\eq{\bigast_{i=0}^{n-1}S_{\lambda^i}\pi_j h^{i}\Delta_0}
		has maximal cardinality for all $n\geq 1$ and $1\leq j\leq \ell$.
		
		Likewise, under mild transversality conditions, projected exponential separation holds in parametrized families of self-similar measures, all of which have the same minimal invariant subspaces, outside a set of parameters of Hausdorff (and even packing) co-dimension at least $1$. For this we refer to~\cite[Theorem 1.10]{Hochman17}, which is applicable to each of the projections $\pi_j\mu_t$ for $t$ in the parameter space.
	\end{rmk}
	
	Let us now delve into the proof of Theorem~\ref{thm:ssm}.
	\begin{proof}[Proof of Theorem \ref{thm:ssm}]
		By Lemma~\ref{lem:one-dim-exp-sep}, it is enough to verify that $\pi\cX$ has PES for all $\pi\in\G(d,1)$, where $\cX$ is the pleasant model generating $\mu$.
		
		Fix, then, $\pi\in\G(d,1)$. Let $\Pi$ the smallest $h$-invariant subspace containing $\pi$. Note that $\Pi\mu$ is also a self-similar measure satsfying the assumptions of the theorem. Moreover, $\Pi\cX$ has the PES property (since the projection to any $\pi_j\subset\Pi$ does), and hence the value of $\|\Delta\|_q^q$ does not change.  Therefore, we may and do assume that $\Pi=\R^d$.
		
		Note that $\cX$ satisfies exponential separation, since $\pi_1\cX$ satisfies PES. This is the point where the hypothesis is used (because we have replaced $h$ by $h|_{\Pi}$, we may now be working with $h|_{\pi_j}$).Thus, there exist $R>0$ and an infinite set of $n$ such that
		\begin{equation} \label{eq:exp-sep-upstairs}
			|a-b|\geq e^{-Rn}\quad \text{for all distinct } a,b\in \supp\mu_{\Id,n}\;.
		\end{equation}
		Recall that the measures generated by $\cX$ are all rotations of each
		other, so it is indeed enough to consider $\mu_{\Id,n}$.
		
		Let $\mathfrak{g}$ be the Lie algebra of $\sX$. Since $\sX$ is connected by hypothesis, the exponential map $\exp\colon \mathfrak{g}\to \sX$ is surjective. Let $A\in\mathfrak{g}$ be such that $e^A=h$. Since by assumption $\pi$ does not belong to any proper $h$-invariant subspace, the same is true for $A$, and therefore we have
		\begin{equation} \label{eq:irrationality-to-full-rank}
			\langle e^{tA} A^j\pi : j=0,\ldots,d-1\rangle = \R^d\;,
		\end{equation}
		for all $t\in\R$. Given $v\in S^{d-1}$, consider the curve $\gamma_{v}(t) = \pi e^{t A} v  \subset \R$. Note that
		\[
		\gamma_{v}^{(j)}(t) = \pi e^{t A}A^j v\;.
		\]
		It follows from \eqref{eq:irrationality-to-full-rank} that $\gamma_{v}$ vanishes to order at most $d-1$, uniformly in $v$, that is,
		\[
		\max\bigl\{ |\gamma_{v}^{(j)}(t)| : j=0,\ldots,d-1\bigr\} \ge c\;,
		\]
		for all $v\in S^{d-1}$ and some constant $c>0$. Indeed, for each $v$ there is $0\le j\le d-1$ such that $e^{t A}A^j \pi$ makes an angle bounded away from $0$ with $v^{\perp}$, uniformly in $t$ and $v$ by compactness.
		
		Denoting the Lebesgue measure on $[0,1]$ by $\mathcal{L}$, it follows from, e.g., \cite[Lemma 5.8]{Hochman14} that
		\[
		\mathcal{L}\bigl( \bigl\{ t: |\pi e^{t A} g v| \le \delta \bigr\} \bigr) \le C\,\delta^{2^{-d}}\;,\quad g\in \sX\;, v\in S^{d-1}\;,
		\]
		for all $\delta>0$ and some $C>0$ independent of $v$ and $g$. Note that the push-forward of $\bP\times \mathcal{L}$ under $(g,t)\mapsto e^{t A} g$ is a translation-invariant Borel probability measure on $\sX$, and thus necessarily equals $\bP$. Applying Fubini's Theorem, we deduce that
		\[
		\bP \bigl(\bigl\{ g: |\pi g v| \le \delta \bigr\}\bigr) \le C\,\delta^{2^{-d}}\;,\quad \delta>0\;.
		\]
		We apply this to each $v=|a-b|^{-1}(a-b)$, where $a,b\in \supp\mu_{\Id,n}$ are distinct. There are at most $ |\cI|^{2n}$ such pairs, and using \eqref{eq:exp-sep-upstairs} we deduce that, for any $\kappa\in (0,1)$,
		\[
		\bP \bigl(\bigl\{ g: |\pi g a - \pi g b| \le \kappa^n \text{ for some }a,b\in\supp\mu_{\Id,n}  \bigr\} \bigr)\le C |\cI|^{2n} e^{Rn}\,\kappa^{2^{-d} n}\;.
		\]
		Taking $\kappa < \left(|\cI|^{-2} e^{-R}\right)^{2^d}$, we see that the above probability is summable in $n$, and therefore by Borel-Cantelli there are infinitely many $n$ such that
		\[
		|\pi g a - \pi g b| \ge \kappa^n \quad \text{for all distinct } a,b\in\supp\mu_{\Id,n}\;.
		\]
		Since the measures generated by $\pi\cX$ are $\mu^{\pi\cX}_g= \pi g\mu_{\Id}$, we conclude that $\pi\cX$ satisfies projected exponential separation, as we wanted to show.
	\end{proof}

	We conclude this section by observing that the proof of Theorem~\ref{thm:ssm} yields the following statement.
	\begin{prop}
		Let $\mu$ and $h$ satisfy all the assumptions of Theorem~\ref{thm:ssm} with the exception of projected exponential separation. Fix $\pi\in\G(d,1)$, and let $\Pi$ be the smallest $h$-invariant subspace containing $\pi$. If $\Pi\mu$ satisfies projected exponential separation, then so does $\pi\mu$.
	\end{prop}
	
	Note that the subspace $\Pi$ in the proposition equals $\R^d$ for a Zariski open (thus, analytically open and dense) subset of $\G(d,1)$. Therefore, under the given algebraic assumptions on $h$, exponential separation for the original measure $\mu$ implies projected exponential separation for all projections to $1$-dimensional subspaces outside a sparse set of exceptions.

	\section{Products of self-similar measures and their projections}
	\label{sec:products}
	
	We now investigate projections of products of homogeneous self-similar measures on the real line. In one dimension, the orthogonal part $h$ of a homogeneous self-similar IFS $\Phi=\{f_i \}_{i\in I}$ is $\pm \text{id}$, and upon replacing $\Phi$ with $\Phi^{2}=\{f_{i}\circ f_{j}\}_{i,j\in I}$, which generates any self-similar measure which is geneated by $\Phi$, we may and shall assume throughout this section that $h=\text{id}$.
	
	Consider a collection of $d$ self-similar measures
	\eql{eq:homselfsim}{\mu^{(j)}=\bigast_{n\geq 0}S_{\lambda_j^n}\Delta_j\;, \quad 1\leq j\leq d}
	on the real line. To begin with, we show that the product measure $\mu^{(1)}\times\cdots\times \mu^{(d)}$ can be recast as a dynamically driven self-similar measure. The argument runs along similar lines to the appearing in the proof of~\cite[Theorem 7.5]{Shmerkin-Annals},  which concerns one-dimensional linear images of such products. Slightly abusing notation, we write $a\cdot \nu$ for the scaling of a measure $\nu$ by a factor $a>0$, i.e., $a\cdot \nu = S_a\nu$. The $d$-torus $\T^d=\R^d/\Z^d$ is identified canonically with the half-open cube $[0,1)^{d}$.
	\begin{lem} \label{lem:model-product-ssm}
		Let $(\mu^{(j)})_{1\leq j\leq d}$ be as above, ordered so that $\lambda_d = \max_{j=1,\dots,d}\lambda_j$. There exists a pleasant model $\cX=(\sX,\bT,\bP,\Delta,\lambda_d)$ in $\R^d$ such that:
		\begin{enumerate}
			\item $\sX$ is a closed subgroup of $\T^{d-1}$, the map $\bT$ is a translation on $\sX$, and $\bP$ is the unique probability Haar measure on $\sX$;
			\item the generated measures are given by
			\[
			\mu_{\sx} =  \lambda_1^{-\sx_1}\cdot \mu^{(1)}\times\cdots\times \lambda_{d-1}^{-\sx_{d-1}}\cdot \mu^{(d-1)}\times \mu^{(d)} 
			\]
			for all $\sx=(x_1,\dots,x_{d-1})\in \sX$.
			In particular, for the identity element we obtain 
			\[
			\mu_{0} = \mu^{(1)}\times\cdots\times \mu^{(d)}\;.
			\]
		\end{enumerate}
		
	\end{lem}
	\begin{proof}
		We write $a_j=|\log\lambda_j|$. Let $\bT:\mathbb{T}^{d-1}\to \mathbb{T}^{d-1}$ be the translation given by
		\[
		\bT(\sx_1,\ldots,\sx_{d-1}) = \left(\sx_1+\frac{a_d}{a_1},\ldots, \sx_{d-1}+\frac{a_d}{a_{d-1}}\right)\;.
		\]
		Let $\sX\subset \mathbb{T}^d$ be the closure of the orbit of $0$ under $\bT$, and let $\bP$ be the Haar measure on $\sX$. Given $\sx\in \sX$, we let
		\[
		J(\sx)=\left\{ j\in\{1,\ldots,d-1\}: \sx_j \in \left[0,\frac{a_d}{a_j}\right)\right\}\;,
		\]
		and define $\Delta:\sX\to \cA_d$ as
		\[
		\Delta(\sx) = \left( \bigtimes_{j\in  J(\sx)}\lambda_j^{-\sx_j} \cdot \Delta_j \right) \bigtimes \left(\bigtimes_{j\notin J(\sx)}\delta_0\right)\bigtimes \Delta_d\;,
		\]
		where $\delta_0$ is the Dirac mass at $0\in \R$.
		The model $\sX=(\sX,\bT,\bP,\Delta,\lambda)$ is now readily checked to be pleasant (unique ergodicity follows from the density of the orbit of $0$ in $\sX$).
		
		Let $\bT_j$ be translation by $a_d/a_j$ on the $1$-torus. For $1\le j\le d-1$, $y\in [0,1)$ and $n\in\N$, let
		\[
		n'_j(y)=\left|\left\{i\in[1,n]: \bT_j^i(y)\in [0,\frac{a_d}{a_j})\right\}\right|
		\]
		be the number of times the orbit of $y$ under $\bT_j$ wraps around the circle. Then,
		\[
		\bT_j^n(y) = y  +n \frac{a_d}{a_j} - n'(y) \in [0,1)\;,
		\]
		so that
		\[
		\lambda_j^{-\bT_j^n(y)} \lambda_d^{n} = \lambda_j^{-y} \cdot \lambda_j^{n'(y)}\;.
		\]
		Therefore,
		\begin{equation} \label{eq:model-for-conv-ssm}
			\bigast_{i=1}^n \Delta(\bT^i \sx) =  \left[\bigtimes_{j=1}^{d-1} \lambda_j^{-\sx_j}\cdot \left(\bigast_{i=1}^{n'_j(\sx_j)}\cdot S_{\lambda_j^i}\Delta_j\right) \right] \bigtimes \left(\bigast_{i=1}^{n} S_{\lambda_d}^n\Delta_d\right)\;.
		\end{equation}
		The claim follows by convolving with $\Delta(\sx)$ to get $\mu_{n+1,\sx}$, and then letting $n\to\infty$.
	\end{proof}
	
	If each of the IFS's generating the measures $\mu^{(j)}$ satisfies exponential separation, then~\cite[Theorem 6.2]{Shmerkin-Annals} already delivers the dimension formula for $\dim(\mu^{(j)},q)$. Using the fact that the limit in the definition of $L^q$ dimension exists, a simple calculation shows that
	\[
	\dim\bigl(\mu^{(1)}\times\cdots\times \mu^{(d)},q\bigr) = \sum_{j=1}^d\dim(\mu^{(j)},q)\;.
	\]
	Unlike the case of self-similar measures, it is not thus the $L^q$-dimension of the product measure that we are after, but rather of its projections. The motivation for this quest is given by Furstenberg-type slicing results, which we discuss in the next section.
	
	Further contrasting with the case of self-similar measures is the presence here of exact overlaps, resulting from the non-injectivity of every coordinate projection $\pi$ restricted to the support $\Delta(\sx)$; as a consequence, the strategy of verifying $q$-unsaturation by proving projected exponential separation of all one-dimensional projections is not available. However, the product structure, together with an induction argument in the dimension, can be used to verify $q$-unsaturation.
	
	Note that if all the $\lambda_j$'s are equal, then the product measure is itself a homogeneous self-similar measure on $\R^d$ and, as discussed previously, such measures may fail to be unsaturated on lines, and may have projections with $L^q$-dimension drop. In fact, this is the case as soon as $\lambda_i=\lambda_j$ for some $i\neq j$, or even just $\lambda_i/\lambda_j\in\Q$ for some $i\neq j$ (the latter condition being, upon suitable iterations of the corresponding models, equivalent to the former); such pathology occurs for the very same reason it can emerge for self-similar measures with repeated eigenvalues. 
	
	For the rest of this section, we fix $q>1$, and let $\mu^{(1)},\dots,\mu^{(d)}$ be homogeneous self-similar measures on $\R$, described as in~\eqref{eq:homselfsim}, with the following properties:
	\begin{enumerate}[(P\arabic*)]
		\item \label{it:ssmprod:i} $\lambda_i/\lambda_j$ is irrational for all $1\leq i\neq j\leq d$;
		\item \label{it:ssmprod:ii} the IFS generating $\mu^{(j)}$ satisfies exponential separation for all $1\leq j\leq d$.
	\end{enumerate}
	Let $\cX$ be the pleasant model provided by Lemma \ref{lem:model-product-ssm}.  For a line $\pi\in\G(d,1)$, we write $Z(\pi)$ for the maximal number of vanishing coordinates of a unit vector in $\pi$.
	\begin{lem} \label{lem:convolution-exp-sep}
		If $\pi\in\G(d,1)$ satisfies $Z(\pi)=0$, then $\pi\cX$ satisfies projected exponential separation.
	\end{lem}
	
	\begin{rmk}
		The assumption $Z(\pi)=0$ is necessary in general: if $\pi$ is contained in the hyperplane $\{x=(x_i)_{1\leq i\leq d}\in \R^d:x_j=0\}$ for some $1\leq j\leq d$, then $\pi$ fails to be injective on the support $\Delta(\sx)$ for all $\sx\in \sX$, due to the exact overlaps coming from the $j$-th coordinate. Moreover, in this case we have, appealing for instance to Proposition \ref{prop:upperbound},
		\eql{eq:dimsumbound}
		{D_{\pi\cX}(q) \le \sum_{k\in\{1,\ldots,d\}\setminus\{j\}} \dim(\mu^{(k)})\;,}
		and thus $D_{\pi\cX}(q)<D_{\cX}^{\text{s}}(q)$ whenever the right-hand side of~\eqref{eq:dimsumbound} is strictly smaller than $1$.
	\end{rmk}
	
	\begin{proof}[Proof of Lemma \ref{lem:convolution-exp-sep}]
		By the irrationality assumption \ref{it:ssmprod:i}, the phase space $\sX$ given by Lemma \ref{lem:model-product-ssm} is a sub-torus of $\mathbb{T}^d$. Let $\ell=\dim\sX$. After rearranging the $\mathcal{S}_j$, we may assume that
		\[
		\sX = \left\{ (\sy,L(\sy)): \sy\in\mathbb{T}^{\ell} \right\}\;,
		\]
		for some linear map $L\in\mathbb{Z}^{d-\ell\times \ell}$. Moreover, $\bP$ is the push forward of Haar measure $m_{\ell}$ on $\mathbb{T}^\ell$ under $\sy\mapsto (\sy,L\sy)$.
		
		Let $v$ be a unit vector in direction $\pi$; since $Z(\pi)=0$, we have $v_i\neq 0$ for all $1\le i\le d$. By the assumption \ref{it:ssmprod:ii}, the measure $\mu^{(1)}$ satisfies exponential separation. Hence, there are $R$ and an infinite set of $n$ such that
		\begin{equation} \label{eq:exp-sep-mu}
			|a_1-b_1| \ge e^{-Rn} \quad \text{for all distinct } a_1,b_1\in \supp\mu^{(1)}_{n}\;.
		\end{equation}
		Now the atoms of $\mu_{\sx,n}$ (referring to the measures generated by $\cX$) are of the form
		\[
		\left(\lambda_1^{-\sx_1} a_1, \ldots, \lambda_{d-1}^{-\sx_{d-1}} a_{d-1}, a_d \right)\;,
		\]
		for $a_j\in\supp\mu^{(j)}_n$, and therefore the atoms of $\pi\mu_{\sx,n}$ are of the form
		\[
		p((a_j),\sx) \coloneqq \lambda_1^{-\sx_1} a_1 v_1 + \cdots + \lambda_{d-1}^{-\sx_{d-1}}  a_{d-1} v_{d-1} + a_d  v_{d}\;,
		\]
		where $a_j\in\supp\mu^{(j)}_n$. The key observation is that, for fixed $a_j, b_j\in\supp\mu^{(j)}_n$ and $\sy_2,\ldots, \sy_{\ell}\in\mathbb{T}^{\ell}$, the function
		\[
		\phi(\sy_1) = p\bigl((a_j),(\sy,L(\sy))\bigr) - p\bigl((b_j),(\sy,L(\sy))\bigr)
		\]
		is smooth on $[0,1)$ with derivative bound below by
		\[
		\phi'(\sy_1) \ge \log(1/\lambda_1) v_1 (a_1-b_1) \overset{\eqref{eq:exp-sep-mu}}{\gg} e^{-Rn},
		\]
		allowing the implicit constant to depend on $\lambda_1,v_1$ (but not on $\sy_j$ or $n$). Then
		\begin{align*}
			m_1 \Big\{ \sy_1: & \left|p\bigl((a_j),(\sy,L(\sy))\bigr) - p\bigl((b_j),(\sy,L(\sy))\bigr) \right| \le \kappa^n \,\text{for some }a_j\neq b_j\in \supp(\mu^{j}_n)   \Big\} \\
			& \ll e^{Rn}\left(|\supp\Delta_1|\cdots|\supp\Delta_d|\right)^{2n}  \kappa^n\;.
		\end{align*}
		Take $\kappa< e^{-R}\prod_{j=1}^{d}|\supp\Delta_j|^{-2}$. Combining this with Borel-Cantelli and Fubini applied to $m_1\times m_{\ell-1}$, we deduce that for $m_\ell$-almost all $\sy=(\sy_1,\ldots,\sy_d)$,
		\[
		\left|p\bigl((a_j),(\sy,L(\sy))\bigr) - p\bigl((b_j),(\sy,L(\sy))\bigr) \right| \ge \kappa^n \quad \text{for all } a_j\neq b_j\in \supp(\mu^{(j)}_n)\;.
		\]
		This is what we wanted to show.
	\end{proof}
	
	For $\pi\in\G(d)$, let us define
	\[
	Z(\pi) = \sup\bigl\{Z(\pi_0):\pi_0\in \G(d,1),\;\pi_0\subset \pi\bigr\}\;,
	\]
	that is, $Z(\pi)$ is the largest number of zero coordinates of a unit vector contained in $\pi$. We are now able to state the main result of this section.
	\begin{thm} \label{thm:proj-product-ssm}
		Let $k\in\{1,\dots,d\}$, and suppose $\pi\in\G(d,k)$ satisfies $Z(\pi)\le k-1$. Then,
		\[
		D_{\pi\cX}(q) = \min\left\{k, D_{\cX}(q) \right\} = \min \left\{k, \sum_{j=1}^d \dim(\mu^{(j)},q) \right\}\;.
		\]
	\end{thm}
	
	\begin{rmk}
		It is easy to check that one always has $Z(\pi)\ge k-1$. The assumption $Z(\pi)\le k-1$ is necessary in general: consider the case in which $\pi$ is contained in a coordinate hyperplane, say $\{x_d=0\}$, which implies $Z(\pi)\ge k$. Then $\pi\cX$ does not see the self-similar measure $\mu^{(d)}$, whence a formula for $D_{\pi\cX}(q)$ cannot involve it.
	\end{rmk}
	
	\begin{rmk}
		The assumption $Z(\pi)\le k-1$ can be restated as follows: every $(d-k)$-dimensional coordinate subspace $V$ (i.e., any subspace $V$ which is spanned by $d-k$ of the standard basis vectors), intersects $\pi$ in the trivial subspace. Thus, this can be seen as a transversality condition with respect to the coordinate directions.
	\end{rmk}

	\begin{proof}[Proof of Theorem \ref{thm:proj-product-ssm}]
		We will prove a more general statement that is better suited to induction. Let
		\[
		s(\ell) = \min \left\{ \sum_{j\in J} \dim(\mu^{(j)},q) : |J|=\ell \ \right\}\;.
		\]
		Note that $s(\ell)$ is increasing in $\ell$, $s(d)=\sum_{j=1}^{d} \dim(\mu^{(j)},q)$, and
		\begin{equation} \label{eq:s-increases-by-at-most-1}
			s(\ell)  \le s(\ell-1) + 1,
		\end{equation}
		since $\dim(\nu,q)\le 1$ for any $\nu\in\cP(\R)$.

		We will show that, for any $\pi\in \G(d,k)$,
		\begin{equation} \label{eq:general-proj-ineq}
			D_{\pi\cX}(q)  \ge \min\bigl\{k,s\bigl(d+(k-1)-Z(\pi)\bigr) \bigr\} \;.
		\end{equation}
		Note that this clearly implies the claim of the theorem.
		
		We prove \eqref{eq:general-proj-ineq} by induction in $k$. For the base case $k=1$, let $\ell=Z(\pi)$, and assume without loss of generality that the direction vector $v$ of $\pi$ satisfies
		\[
		v_j  = \left\{ \begin{array}{ll}
			\neq 0 & \text{for } 1 \le j \le d - \ell     \\
			=0     & \text{for } d - \ell + 1 \le j \le d
		\end{array} \right..
		\]
		Then $\pi\cX$ can be identified with $\widetilde{\pi}\widetilde{\cX}$, where $\widetilde{\cX}$ is the model on $\R^{d-\ell}$ associated to the tuple $\mu^{(1)},\ldots, \mu^{(d-\ell)}$, and $\widetilde{\pi}\in\G(d-\ell,1)$ is the line with direction $(v_1,\ldots,v_{d-\ell})$. Applying Lemma \ref{lem:convolution-exp-sep} and Theorem \ref{thm:lqdimension} to $\widetilde{\pi}\widetilde{\cX}$, we deduce that
		\[
		D_{\pi\cX}(q) = D_{\widetilde{\pi}\widetilde{\cX}}(q) = \sum_{j=1}^{d-\ell} \dim(\mu^{(j)},q) \ge s(d-\ell)\;,
		\]
		which is the case $k=1$ of \eqref{eq:general-proj-ineq}.
		
		Assume now that the claim \eqref{eq:general-proj-ineq} has been verified for $k-1\in \{1,\ldots,d-1\}$, and let $\pi\in\G(d,k)$. Assume, for the sake of contradiction, that
		\begin{equation} \label{eq:small-proj-counter-assumption}
			D_{\pi\cX}(q)  < \min\bigl\{k,s\bigl(d+(k-1)-Z(\pi)\bigr) \bigr\} \;.
		\end{equation}
		Similar to the case $k=1$, let $\Pi$ be the smallest coordinate subspace (i.e., the smallest subspace generated by canonical basis vectors) containing $\pi$. Upon replacing $\R^d$ by $\Pi$, the latter identified with $\R^{\dim\Pi}$, and $\pi$ by $\pi|_{\Pi}$, we may assume that $\Pi=\R^d$. Note that this change has the effect of reducing both $d$ and $Z(\pi)$ by $d-\dim(\Pi)$ while preserving $k$, so \eqref{eq:general-proj-ineq} is unchanged.
		
		Since, after this reduction, $\pi$ is not contained in any proper coordinate hyperplane, there is $\pi_0\in\G(\pi,1)$ which is also not contained in any coordinate hyperplane. By Lemma \ref{lem:convolution-exp-sep}, $\pi_0\cX$ satisfies projected exponential separation, and hence so does $\pi\cX$.
		
		Fix an arbitrary $\pi'\in\G(\pi,k-1)$. Since, trivially, $Z(\pi')\le \min\{k-2,Z(\pi)\}$, the inductive hypothesis yields
		\[
		D_{\pi'\cX}(q) \ge \min\bigl\{k-1, s\bigl(d+(k-2)-Z(\pi)\bigr) \bigr\}\;.
		\]
		Applying inequality \eqref{eq:s-increases-by-at-most-1} with $\ell=d+(k-1)-Z(\pi)$, we deduce from the counter-assumption \eqref{eq:small-proj-counter-assumption} that $\pi\cX$ is $q$-unsaturated on lines.
		
		We have verified that $\pi\cX$ satisfies the assumptions of Theorem \ref{thm:lqdimension}. Therefore, $D_{\pi\cX}(q) = \min\{k,s(d)\}$, which, however, contradicts \eqref{eq:small-proj-counter-assumption}. We conclude that \eqref{eq:general-proj-ineq} holds also for $k$, achieving the induction.
		
	\end{proof}

	\section{Furstenberg-type slicing results}
	
	\label{sec:slicing}
	
	\subsection{From \texorpdfstring{$L^q$}{L^q}-dimension to dimension of slices}
	
	There is a simple connection between the $L^q$-dimension of projections and the box dimension of fibers, which runs via Frostman exponents. For the reader's convenience, we restate a minor variant of \cite[Lemma 1.8]{Shmerkin-Annals}. Given a metric space $X$ and a real number $\delta>0$, the notation $|X|_{\delta}$ stands for the $\delta$-packing number of $X$, that is, the largest cardinality of a $\delta$-separated subset of $X$. When $X\subset\R^d$ is bounded, it is well known that $\cN_m(X)=\Theta_d(1) |X|_{2^{-m}}$ for all integers $m\geq 1$. Recall that the upper box dimension of a totally bounded set $X$ in a metric space is defined as
	\eql{eq:defbox}{
		\ubdim(X) = \limsup_{\delta\to 0} - \frac{\log |X|_{\delta}}{\log \delta}\;.
	}
	We refer, for instance, to~\cite[\S 2.1]{Falconer-techniques} for the main properties of the upper box dimension. Here we simply recall that 
	\eq{\hdim(X)\leq \ubdim (X)}
	for any totally bounded set $X$ in a metric space.

	\begin{lem} \label{lem:Frostman-exp-to-small-fiber}
		Let $(X,d)$ be a metric space. Suppose $\pi\colon X\to \R^k$ is an $L$-Lipschitz map, $L>0$. Let $\mu$ be a Borel probability measure on $X$ with the following property: there are real numbers $0\leq t\leq s$, $C_1,C_2>0$ and an integer $m_0\geq 1$ such that
		\begin{equation*}
			\mu(B(x,r))  \ge C_1 \;r^s         
		\end{equation*}
		for all $x\in X$ and $0< r\leq 2^{-m_0}$ and
		\begin{equation*}
			\pi\mu(Q)    \le C_2\; 2^{-mt}
		\end{equation*}
		for all $m\geq m_0$ and all $Q\in \cD_{m}$.
		
		Then, for any $m\ge m_0$ and any closed Euclidean ball $B\subset \R^k$ of radius $2^{-m}$,
		\[
		\left| \pi^{-1}(B)\right|_{2^{-m}}=O_{k,L,C}(1)\, 2^{m(s-t)}
		\]
		for $C=C_2/C_1$.
		In particular, for any $y\in\R^k$,
		\[
		\ubdim\bigl(\pi^{-1}(y)\bigr) \le s-t\;.
		\]
	\end{lem}
	\begin{proof}
		
		Fix an integer $m\geq m_0$ and a closed Euclidean ball $B=B(y,2^{-m})$, $y\in \R^k$. Let $\{x_j  \}_{j\in J}$ be a maximal $2^{-m}$-separated subset of $\pi^{-1}(B)$, so that $|J|=|\pi^{-1}(B)|_{2^{-m}}$. Using the assumption on $\mu$, we estimate
		\eql{eq:firstest}{C_12^{-ms}|\pi^{-1}(B)|_{2^{-m}}\leq \mu\biggl(\bigsqcup_{j\in J}B(x_j,2^{-m})\biggr)\;.}
		Suppose now $x\in X$ is contained in the last displayed disjoint union, say $y\in B(x_j,2^{-m})$; then
		\eq{|\pi(x)-y|\leq |\pi(x)-\pi(x_j)|+|\pi(x_j)-y|\leq Ld(x,x_j)+2^{-m}\leq (L+1)2^{-m}\;, } 
		using in the second-to-last step the fact that $\pi(x_j)\in B$. It follows from~\eqref{eq:firstest} that 
		\eq{C_12^{-ms}|\pi^{-1}(B)|_{2^{-m}}\leq \mu\bigl(\pi^{-1}\bigl(B(y,(L+1)2^{-m})\bigr)\bigr)=\pi\mu\bigl(B(y,(L+1)2^{-m})\bigr)\;;}
		the last displayed Euclidean ball can be covered  by $O_{k,L}(1)$ cubes in $\cD_m$; combining this with the assumption on $\pi\mu$, we deduce that 
		\eq{C_12^{-ms}|\pi^{-1}(B)|_{2^{-m}}=O_{k,L}(1)C_2\;2^{-mt}\;,}
		from which the first assertion of the lemma follows. 
		
		The upper bound on the box dimension of $\pi$-fibers is then an automatic consequence, taking into account that the superior limit in the definition~\eqref{eq:defbox} can be equivalently taken over the subsequence $(2^{-m})$.
	\end{proof}

	\subsection{Higher rank Furstenberg slicing}
	
	\label{subsubsec:high-rank-Furst}
	
	We are now in a position to establish Theorem~\ref{thm:Furst-higher-rank} and Corollary~\ref{cor:Hugen}.
	
	\begin{proof}[Proof of Theorem \ref{thm:Furst-higher-rank}]
		
		Firstly, any compact $T_p$-invariant set of dimension $s$ can be embedded into a $T_{p^j}$-invariant self-similar set of dimension $<s+\eps$, where $j=j(\eps)$; see, e.g., the proof of \cite[Theorem 1.2]{Shmerkin-Annals}  for this standard fact. Therefore, after replacing $p_j$ by suitable powers, we may assume that the $A_j$'s are restricted-digits sets, i.e.,
		\[
		A_j = A_j(p_j,D_j)
		= \left\{ x\in [0,1): x = \sum_{n=1}^{\infty} a_{n} p_j^{-n} \text{ for some } a_{n}\in D_j \right\}\;,
		\]
		for some $D_j\subset \{0,\ldots,p_j-1\}$ with $|D_j|\ge 2$. In what follows, we use some standard facts about iterated function systems satisfying the open set condition, for which we refer, e.g., to~\cite{Hutchinson81}. Let $\mu^{(j)}$ be the natural uniform self-similar measure on $A_j$, that is,
		\[
		\mu^{(j)}=\bigast_{n\geq 0}S_{p_j^{-n}} \biggl(\frac{1}{|D_j|} \sum_{a\in D_j} \delta_{a/p_j}\biggr)\coloneqq  \bigast_{n\geq 0}S_{p_j^{-n}}\Delta_j\;.
		\]
		The associated IFS $\{x\mapsto (x+j)/p_j: j\in D_j\}$ satisfies the open set condition with open set $(0,1)$; as a consequence, the measure $\mu^{(j)}$ satisfies the Ahlfors regularity condition
		\[
		r^{\hdim(A_j)}\ll	\mu^{(j)}(B(x,r)) \ll r^{\hdim(A_j)} \quad \text{for all } x\in A_j \text{ and } 0<r\le 1\;,
		\]
		the implicit constants not depending on $x$ nor $r$, 
		and
		\[
		\hdim(A_j) = \frac{\log |D_j|}{\log p_j}\;, \]
		where the last quantity equals
		\[ \frac{\log\|\Delta_j\|_q^q}{(q-1)\log p_j^{-1}}
		\]
		for all $q>1$.
		Let $\cX=(\sX,\bT,\bP,\Delta,\lambda)$ be the product model provided by Lemma \ref{lem:model-product-ssm}. In particular,
		\[
		\mu_{0} = \mu^{(1)}\times \cdots \times \mu^{(d)}\;.
		\]
		Then, for any $q>1$ we have $D_{\cX}(q) = s$, and
		\begin{equation}  \label{eq:product-ssm-ahlfors}
			r^{s}\ll\mu_{0}(B(y,r)) \ll r^s \quad \text{for all } y\in \supp\mu_{0} \text{ and } 0<r\le 1\;.
		\end{equation}
		Here the implicit constants depend on the $p_j$'s, but not on $y, r$.

		Set
		\[
		\G'_{\eta}(d,k)= \bigl\{ \pi\in\G(d,k): Z_{\eta}(\pi)\le k-1 \bigr\}\;.
		\]
		Note that this is a compact subset of $\G(d,k)$. Theorem \ref{thm:proj-product-ssm} implies that $D_{\pi\cX}(q) = \min\{k,s\}$ for all $q>1$ and $\pi\in \G'_{\eta}(d,k)$. Take $q$ large enough that $1/q'<\eps/(3d)$. By Proposition \ref{prop:semicontinuity}, for large enough $n$ (depending on $\eta$ but not on $\pi$) and all $Q_0\in\cD_{m(n)}$, we have
		\[
		\pi\mu_{0}(Q_0)^q \le \sum_{Q\in\cD_{m(n)}} \mu(Q)^q \le   2^{-m(n)(q-1)(\min\{k,s\}-\eps/3)}\quad\text{for all }\pi\in \G'_{\eta}(d,k)\;.
		\]
		Therefore, by our choice of $q$,
		\[
		\pi\mu_{0}(Q_0) \le 2^{-m(n)(\min\{k,s\}-2\eps/3)}\;,\quad  Q_0\in\cD_{m(n)},\;  \pi\in \G'_{\eta}(d,k).
		\]
		Since $m(n)$ has bounded gaps, this extends to all $m$ (using $2^{m\eps/3}$ to absorb the constant factor).  Combining this with \eqref{eq:product-ssm-ahlfors} and Lemma \ref{lem:Frostman-exp-to-small-fiber} applied with $t=\min\{k,s\}-\eps$, we conclude that for all affine subspaces $P$ orthogonal to some $\pi\in\G'_\eta(d,k)$,
		\[
		\left|\bigl(A_1\times \cdots \times A_d\bigr) \cap P  \right|_{\delta} =O_{d,k,\eps,\eta}(1)\; \delta^{-\max\{s-k,0\}-\eps}\;.
		\]
		This is what we wanted to prove.
	\end{proof}

	\begin{proof}[Proof of Corollary \ref{cor:Hugen}]
		If some $g_j$ is constant, the claim is trivial, so assume otherwise. Up to an affine change of coordinates depending smoothly on the $g_j$,
		\[
		g(A_1)\cap \cdots \cap g_d(A_d)=(A_1\times\cdots\times A_d)\cap P,
		\]
		where $P$ is the affine line $\{ (g_1^{-1}(y),\ldots,g_d^{-1}(y)): y\in\R\}$. Since each $g_j$ has non-zero slope, the hyperplane $\pi=P^{\perp}$ satisfies $Z(\pi)\le d-2$. Replacing $g_j(x)$ by $g_j(p_j^{\ell}x)$ for a suitable $\ell\in\N$, we may assume that all the slopes of the $g_j$'s are $\ge 1$. The claim now follows from Theorem \ref{thm:Furst-higher-rank}.
	\end{proof}

	\subsection{Slices of self-similar sets}
	
	To conclude, we note that the same argument in the proof of Theorem \ref{thm:Furst-higher-rank} can be applied to obtain slicing results for self-similar sets, using the results from Section \ref{sec:ssm}. We state just one such result, which is a generalization of \cite[Corollary 8.3]{Shmerkin-Annals}.
	
	\begin{thm} \label{thm:slices-sss}
		Suppose $h\in\Or_d(\R)$ has distinct complex eigenvalues and generates a subgroup with connected closure. Let $a_1,\ldots,a_{\ell}\in\R^d$ and $\lambda\in (0,1)$. Denote the associated self-similar set by $A$, that is,
		\[
		A = \bigcup_{j=1}^{\ell} \lambda h A + a_j\;.
		\]
		
		Suppose that for all minimal $h$-invariant subspaces $\pi$, the self-similar set in $\pi$ with contraction ratio $\lambda$, orthogonal part $h|_{\pi}$ and translation vectors $\pi(a_1),\dots,\pi(a_{\ell})$ satisfies exponential separation.
		
		Then, for each $k\geq 1$ and $\eps>0$ there is  $C_{d,k,\eps}>0$ such that, for all $k$-dimensional affine subspaces $P\leq\R^d$, the inequality
		\[
		\left| A\cap P \right|_{\delta} \le C_{d,k,\eps} \delta^{-\max\{\hdim(A)-k,0\}-\eps}
		\]
		holds for all $0<\delta \leq 1$.
	\end{thm}
	\begin{proof}
		If the IFS $\{x\mapsto  \lambda h x + a_j:1\leq j\leq \ell\}$ satisfies the open set condition, then the same argument from the proof of Theorem \ref{thm:Furst-higher-rank} applies. In fact, the situation is slightly easier because $\G(d,d-k)$ is already compact.
		
		In the general case, instead of considering $\pi$ we consider the maps $\pi\circ \Psi: \{1,\ldots,\ell\}^{\N} \to \R^{d}$, where $\Psi$ is the coding map for the given IFS, that is,
		\[
		\Psi(\omega) = \sum_{n=0}^{\infty} \lambda^n h^n a_{\omega_n}\;, \quad \omega=(\omega_n)_{n\geq 0}\in \{1,\dots,\ell\}^{\N}.
		\]
		If we endow $\{1,\ldots,\ell\}^{\N}$ with the uniform Bernoulli measure $\overline{\mu}$ and the metric
		\[
		d(\omega,\omega') = \lambda^{\min\{n\in \N: \omega_n\neq \omega'_n\}}\;,
		\]
		then 
		\eq{r^{\hdim(A)}\ll\overline{\mu}(B(\omega,r)) \ll r^{\hdim(A)}} for all $\omega$ and $0<r\le 1$. Applying Lemma \ref{lem:Frostman-exp-to-small-fiber} to the map $\pi\circ \Psi$, rather than to $\pi$ itself, we get that for any $\pi\in\G(d,d-k)$ and any $y\in\pi$,
		\[ 
		\left| \Psi^{-1}\pi^{-1}(y) \right|_{\delta} \le C_{d,k,\eps} \,\delta^{-\min\{k,\hdim(A)\}-\eps}\quad \text{for all $0<\delta\le 1$.}
		\] 
		But $\Psi$ is Lipschitz onto $\supp\mu$, so it increases the $\delta$-packing number by at most a constant factor. This gives the claim.
	\end{proof}
	
	\bibliographystyle{plain}
	\bibliography{biblio}

\begin{thebibliography}{10}

\bibitem{AlgomShmerkin24}
Amir Algom and Pablo Shmerkin.
\newblock On the dimension of orthogonal projections of self-similar measures.
\newblock Preprint, arXiv:2407.16262, 2024.

\bibitem{Austin21}
Tim Austin.
\newblock A new dynamical proof of the {S}hmerkin-{W}u theorem.
\newblock {\em Journal of Modern Dynamics}, 18:1--11, 2021.

\bibitem{BaranySimonSolomyak23}
Bal\'{a}zs B\'{a}r\'{a}ny, K\'{a}roly Simon, and Boris Solomyak.
\newblock {\em Self-similar and self-affine sets and measures}, volume 276 of
  {\em Mathematical Surveys and Monographs}.
\newblock American Mathematical Society, Providence, RI, 2023.

\bibitem{Einsiedler-Ward}
Manfred Einsiedler and Thomas Ward.
\newblock {\em Ergodic theory with a view towards number theory}, volume 259 of
  {\em Graduate Texts in Mathematics}.
\newblock Springer-Verlag London, Ltd., London, 2011.

\bibitem{Falconer-techniques}
Kenneth Falconer.
\newblock {\em Techniques in fractal geometry}.
\newblock John Wiley \& Sons, Ltd., Chichester, 1997.

\bibitem{FalconerJin14}
Kenneth~J. Falconer and Xiong Jin.
\newblock Exact dimensionality and projections of random self-similar measures
  and sets.
\newblock {\em J. Lond. Math. Soc. (2)}, 90(2):388--412, 2014.

\bibitem{Fan-Lau-Rao}
Ai-Hua Fan, Ka-Sing Lau, and Hui Rao.
\newblock Relationships between different dimensions of a measure.
\newblock {\em Monatsh. Math.}, 135(3):191--201, 2002.

\bibitem{Furman}
Alex Furman.
\newblock On the multiplicative ergodic theorem for uniquely ergodic systems.
\newblock {\em Ann. Inst. H. Poincar\'{e} Probab. Statist.}, 33(6):797--815,
  1997.

\bibitem{Furstenberg70}
Harry Furstenberg.
\newblock Intersections of {C}antor sets and transversality of semigroups.
\newblock In {\em Problems in analysis ({S}ympos. {S}alomon {B}ochner,
  {P}rinceton {U}niv., {P}rinceton, {N}.{J}., 1969)}, pages 41--59. Princeton
  Univ. Press, Princeton, N.J., 1970.

\bibitem{GanGuoWang24}
Shengwen Gan, Shaoming Guo, and Hong Wang.
\newblock A restricted projection problem for fractal sets in $\mathbb{R}^n$.
\newblock {\em Cambridge Journal of Mathematics}, 2024.
\newblock Accepted for publication.

\bibitem{Hochman14}
Michael Hochman.
\newblock On self-similar sets with overlaps and inverse theorems for entropy.
\newblock {\em Ann. of Math. (2)}, 180(2):773--822, 2014.

\bibitem{Hochman17}
Michael Hochman.
\newblock On self-similar sets with overlaps and inverse theorems for entropy
  in $\mathbb{R}^d$.
\newblock {\em Mem. Amer. Math. Soc.}, in press, 2017.

\bibitem{HochmanShmerkin12}
Michael Hochman and Pablo Shmerkin.
\newblock Local entropy averages and projections of fractal measures.
\newblock {\em Ann. of Math. (2)}, 175(3):1001--1059, 2012.

\bibitem{Hutchinson81}
John~E. Hutchinson.
\newblock Fractals and self-similarity.
\newblock {\em Indiana Univ. Math. J.}, 30(5):713--747, 1981.

\bibitem{Katznelson-Weiss}
Yitzhak Katznelson and Benjamin Weiss.
\newblock A simple proof of some ergodic theorems.
\newblock {\em Israel J. Math.}, 42(4):291--296, 1982.

\bibitem{Kingman}
John F.~C. Kingman.
\newblock The ergodic theory of subadditive stochastic processes.
\newblock {\em J. Roy. Statist. Soc. Ser. B}, 30:499--510, 1968.

\bibitem{Lang}
Serge Lang.
\newblock {\em Introduction to transcendental numbers}.
\newblock Addison-Wesley Publishing Co,. Reading, Mass.-London-Don Mills, Ont.,
  1966.

\bibitem{Shmerkin-Annals}
Pablo Shmerkin.
\newblock On {F}urstenberg's intersection conjecture, self-similar measures,
  and the {$L^q$} norms of convolutions.
\newblock {\em Ann. of Math. (2)}, 189(2):319--391, 2019.

\bibitem{Shmerkin23}
Pablo Shmerkin.
\newblock Inverse theorems for discretized sums and ${L}^q$ norms of
  convolutions in $\mathbb{R}^d$.
\newblock Preprint, arXiv:2308.09846, 2023.

\bibitem{ShmerkinSolomyak23}
Pablo Shmerkin and Boris Solomyak.
\newblock Absolute continuity of complex {B}ernoulli convolutions.
\newblock {\em Math. Proc. Cambridge Philos. Soc.}, 161(3):435--453, 2016.

\bibitem{Wu19}
Meng Wu.
\newblock A proof of {F}urstenberg's conjecture on the intersections of $\times
  p$- and $\times q$-invariant sets.
\newblock {\em Annals of Mathematics}, 189(3):707--751, 2019.

\bibitem{Yu21}
Han Yu.
\newblock An improvement on {F}urstenberg's intersection problem.
\newblock {\em Transactions of the American Mathematical Society},
  374:6583--6610, 2021.

\end{thebibliography}
	
\end{document}